\documentclass[10pt,reqno]{amsart}
\usepackage[numbers, square]{natbib}
\usepackage{mathptmx}
\usepackage{amsmath}
\usepackage{amscd}
\usepackage{amssymb}
\usepackage{amsthm}
\usepackage{enumerate}
\usepackage{xspace}
\usepackage[all,tips]{xy}
\usepackage[dvips]{graphicx}
\usepackage{verbatim}
\usepackage{syntonly}
\usepackage{hyperref}
\usepackage{amsmath, amsthm, graphics, amssymb,fullpage,color, epsfig,url}
\usepackage{indentfirst}
\usepackage{color}

\theoremstyle{plain}
\newtheorem{thm}{Theorem}
\newtheorem{lem}{Lemma}
\newtheorem{prop}{Proposition}

\newtheorem{cor}{Corollary}

\theoremstyle{definition}
\newtheorem*{rem}{Remark}
\newtheorem*{rems}{Remarks}

\newenvironment{pf}
{\begin{proof}} {\end{proof}}

\newcommand{\disp}{\displaystyle}

\DeclareMathOperator{\sgn}{sgn}

\DeclareMathOperator{\bigO}{\Cal{O}}
\DeclareMathOperator{\Iff}{\Leftrightarrow}

\newcommand{\eps}{\varepsilon}
\newcommand{\vp}{\varphi}

\newcommand{\al}{\alpha}
\newcommand{\be}{\beta}
\newcommand{\ga}{\gamma}
\newcommand{\de}{\delta}
\newcommand{\De}{\Delta}
\newcommand{\Ga}{\Gamma}
\newcommand{\te}{\theta}
\newcommand{\la}{\lambda}

\newcommand{\om}{\omega}

\newcommand{\ol}{\overline}

\newcommand{\nid}{\noindent}

\newcommand{\iny}{\infty}
\newcommand{\del}{ \partial}

\newcommand{\LP}{\Delta}
\newcommand{\gr}{\nabla}
\newcommand{\pri}{\prime}

\newcommand{\hG}{\hat\Ga}
\newcommand{\hg}{\hat\ga}
\newcommand{\dV}{\sqrt{\bar{\ga}}\, \psi\, d\rho dt}
\newcommand{\rhoo}{\rho + \rho^{\al} }

\newcommand{\norm}[1]{\left\vert \left\vert #1\right\vert\right\vert}

\newcommand{\abs}[1]{\left\vert#1\right\vert}
\newcommand{\set}[1]{\left\{#1\right\}}
\newcommand{\brac}[1]{\left[#1\right]}
\newcommand{\pr}[1]{\left( #1 \right) }
\newcommand{\rhoa}[1]{\rho + #1\rho^{\al} }

\newcommand{\Cal}[1]{\ensuremath{\mathcal{#1}}}

\newcommand{\N}{\ensuremath{\mathbb{N}}}

\newcommand{\R}{\ensuremath{\mathbb{R}}}
\newcommand{\Z}{\ensuremath{\mathbb{Z}}}
\newcommand{\C}{\ensuremath{\mathbb{C}}}

\numberwithin{equation}{section}
\numberwithin{lem}{section}
\numberwithin{cor}{section}

\newcommand{\Keywords}[1]{\par\noindent 
{\small{\bf Keywords\/}: #1}}
\newcommand{\MSC}[1]{\par\noindent 
{\small{\bf Mathematics Subject Classification\/}: #1}}

\date{}
\begin{document}
\author{Blair Davey}
\address{Department of Mathematics, The University of Chicago, 5734 South University Ave, Chicago, IL, 60615, USA}
\email{bdavey@math.uchicago.edu}

\thanks{The author was supported by the National Science and Engineering Research Council of Canada, PGSD2-404040-2011.}

\title{Some quantitative unique continuation results for \\ eigenfunctions of the magnetic Schr\"odinger operator}
\maketitle

\begin{abstract}
We prove quantitative unique continuation results for solutions of $-\LP u + W\cdot \gr u + Vu = \la u$, where $\la \in \C$ and $V$ and $W$ are complex-valued decaying potentials that satisfy $|V(x)| \lesssim {\langle x\rangle^{-N}}$ and $|W(x)| \lesssim {\langle x\rangle^{-P}}$.  For $\disp \mathbf{M}(R) = \inf_{|x_0| = R}\norm{u}_{L^2\pr{B_1(x_0)}}$, we show that if the solution $u$ is non-zero, bounded, and $u(0) = 1$, then $\mathbf{M}(R) \gtrsim \exp\pr{-C R^{\be_0}(\log R)^{A( R)}}$, where $\disp \be_0 = \max\set{2 - 2P, \frac{4-2N}{3}, 1}$.  Under certain conditions on $N$, $P$ and $\la$, we construct examples (some of which are in the style of Meshkov) to prove that this estimate for $\disp \mathbf{M}\pr{R}$ is sharp.  That is, we construct functions $u, V$ and $W$ such that $-\LP u + W\cdot \gr u + Vu = \la u$, $|V(x)| \lesssim {\langle x\rangle^{-N}},$ $|W(x)| \lesssim {\langle x\rangle^{-P}}$ and $\abs{u(x)} \lesssim \exp\pr{-c|x|^{\be_0}\pr{\log |x|}^C}$. \\

 \\

\MSC{35J10, 35J15, 35B60}
\end{abstract}

\section{Introduction}

Since all bounded harmonic functions are constant, it seems natural to consider the behavior of bounded solutions to more general elliptic equations.  In \cite{BK}, \cite{K} and \cite{K2}, it was shown that if $u$ solves $\LP u + Vu = 0$, where $u$ and $V$ are bounded with $u$ normalized so that $u(0) = 1$, then $\disp \mathbf{M}(R) := \inf_{|x_0| = R}\norm{u}_{L^2\pr{B_1(x_0)}} \gtrsim \exp\pr{-cR^{4/3}\log R}$.  This result was first proved in \cite{BK}, where Bourgain and Kenig used a Carleman estimate to establish a quantitative unique continuation result, then applied it to a problem in Anderson localization.  The work in \cite{M} shows that when $u$ and $V$ are complex-valued,  this estimate for $\mathbf{M}\pr{R}$ is sharp .  That is, Meshkov constructed non-trivial, bounded, complex-valued functions $u$ and $V$ that satisfy  $\abs{u(x)} \lesssim e^{-c|x|^{4/3}}$ and $\LP u + V u = 0$.  Meshkov also gave a qualitative version of the result from \cite{BK}; he showed that if $u$ decays faster than $\exp\pr{-c |x|^{4/3}}$, then $u$ must equal zero.  Since the Carleman approach does not distinguish between real and complex values, this method does not improve the estimate for $\mathbf{M}(R)$ when we restrict to real-valued $u$ and $V$.  Perhaps it is possible to reduce the exponent of $4/3$ through a different approach.

In \cite{EKPV}, the authors were interested in determining the strongest possible decay rate for solutions of the equation $\del_t u = i\pr{\LP u + Vu}$.  Their results imply that if $u$ and $V$ are time-independent, and $V$ decays according to $|V(x)| \lesssim \langle x \rangle ^{-N}$, then there exists $c_0 > 0$ such that if
$$\int_{\R^n}e^{c_0|x|^{\be_0}}|u(x)|dx < \iny, \; \textrm{where} \;\be_0 = \frac{4-2N}{3},$$
then $u \equiv 0$. For all $0 \le N < 1/2$, the author of \cite{CS} constructed examples in the style of Meshkov \cite{M} to prove that this qualitative result is sharp.

In this paper, we establish quantitative versions of the results from \cite{EKPV} (summarized in the previous paragraph) by finding lower bounds for the function $\disp \mathbf{M}\pr{R}$.  We also consider what happens with the addition of a non-zero magnetic potential and how eigenfunctions behave.  That is, we study the growth of solutions to $-\LP u + W \cdot \gr u + V u = \la u$, where $\la \in \C$ and $V$ and $W$ are complex-valued potentials that decay at infinity.  And under certain conditions on $\la$, $V$ and $W$, we are able to construct examples to prove that our lower bounds for $\disp \mathbf{M}(R)$ are sharp.  If $V$ and $W$ do not both decay too quickly, these examples are done in the style of Meshkov.  Otherwise, the constructions are much simpler.

The fact that Carleman estimates, order of vanishing results, and (quantitative) unique continuation theorems have been useful in various areas, like geometry and physics for example, motivated this paper.  The author very much hopes that the work presented here will find applications in a variety of settings.  \\

Recall that $\langle x \rangle = \sqrt{1 + \abs{x}^2}$.  Let $\la \in \C$ and suppose that $u$ is a solution to 
\begin{equation}
- \LP u + W\cdot \gr u + Vu = \la u  \;\; \textrm{in} \;\; \R^n,
\label{epde}
\end{equation} 
where 
\begin{equation}
|V(x)| \le A_1{\langle x\rangle^{-N}}, 
\label{vBd}
\end{equation}
\begin{equation}
|W(x)| \le A_2{\langle x\rangle^{-P}},
\label{wBd}
\end{equation}
for $N, P, A_1, A_2 \ge 0$.  Assume also that $u$ is bounded,
\begin{equation}
\norm{u}_{\iny} \le C_0,
\label{uBd}
\end{equation}
and normalized,
\begin{equation}
u(0) \ge 1.
\label{L2lBd}
\end{equation}
Define $\be_c = \disp \max\set{2 - 2P, \frac{4-2N}{3}}$, $\be_0 = \max\set{\be_c, 1}$.  For large $R$, let
$$\mathbf{M}(R) = \inf_{|x_0| = R}\norm{u}_{L^2\pr{B_1(x_0)}}.$$

The following theorem is the main result of this paper.

\begin{thm}
Assume  that the conditions described above in (\ref{epde})-(\ref{L2lBd}) hold.  Then there exist constants $\tilde C_5(n)$, $C_6\pr{n, N, P}$, $C_7\pr{n, N, P, A_1, A_2}$, $R_0\pr{n, N, P, \la, A_1, A_2, C_0}$, such that for all $R \ge R_0$,
\begin{enumerate}[(a)]
\item if $\be_c > 1$ ($\be_0 = \be_c$), then
\begin{equation}
\mathbf{M}(R) \ge \tilde C_5\exp\pr{-C_7 R^{\be_0}(\log R)^{C_6}},
\label{Mesta}
\end{equation}
\label{partA}
\item if $\be_c < 1$ ($\be_0 = 1$), then
\begin{equation}
\mathbf{M}(R) \ge \tilde C_5\exp\pr{-C_7 R(\log R)^{C_6\log\log R}}.
\label{Mestb}
\end{equation}
\label{partB}
\end{enumerate}  
\label{MEst}
\end{thm}

\begin{rems}
It is interesting to note that the only value in Theorem \ref{MEst} that depends on the eigenvalue, $\la$, is  the starting point for the the radius, $R_0$.  
The missing case of $\be_c = 1$ will be explained later on.
\end{rems}

To prove this theorem, we will use an iterative argument based on two propositions.  The first step in the proof of each proposition uses the following Carleman estimate.  This Carleman estimate may be thought of as a corollary to Lemma \ref{DFCar} in Section \ref{carEst}.  Lemma \ref{DFCar} follows the same approach as the Carleman estimate by Donnelly and Fefferman in \cite{DF1} and \cite{DF2}, but establishes a different estimate.  Various Carleman estimates were used by Donnelly and Fefferman in their study of local geometric properties as well as global growth estimates.  In \cite{DF1} and \cite{DF2}, they established estimates for the Hausdorff measure of nodal sets of eigenfunctions of the Laplacian.  In \cite{DF}, Donnelly and Fefferman proved global estimates for the growth of bounded harmonic functions on non-compact manifolds.  This work is similar to the latter since it presents global estimates for the growth of bounded eigenfunctions of elliptic equations.

\begin{cor} [Corollary to Lemma \ref{DFCar}]
Let $\la \in \C$.  There exist constants $C_1, C_2, C_3$, depending only on the dimension $n$, and an increasing function $ w(r)$, $0 < r < 6$, so that $$\frac{1}{C_1} \le \frac{ w(r)}{r} \le C_1$$ and such that for all $f \in C^\iny_0\left(B_6(0)\setminus\{0\}\right)$, $\al > C_2\pr{1 + \sqrt{\abs{\la}}}$,
$$\al^3\int w^{-2-2\al}\abs{f}^2 + \al \int w^{-2\al}|\gr f|^2 \le C_3 \int  w^{-2\al}\abs{\LP f + \la f}^2.$$
\label{CT}
\end{cor}

\begin{rem}
While attempting to prove Lemma \ref{DFCar} through the approach of Donnelly and Fefferman, I first established an estimate for all $\al > C_2\pr{1 + \abs{\la}^{1/2 +\eps}}$.  In the scaling argument, since the eigenvalue $\la$ is replaced by $R^2 \la$, this gave $\al \gtrsim R^{1+2\eps}$.  Since $\mathbf{M}(R) \gtrsim \exp\pr{C \al}$, I knew that in order to prove Theorem \ref{MEst}(\ref{partB}), I needed to improve the lemma so that it held for all $\al > C_2\pr{1 + \sqrt{\abs{\la}}}$.  I learned from Carlos Kenig \cite{K0} that he and Claudio Mu\~noz had proven a Carleman estimate that is essentially the same as Corollary \ref{CT} by following the approach presented in \cite{BK}.  This motivated me to improve my result and establish an eigenfunction version of the Carleman estimate in \cite{BK}, \cite{K} and \cite{K2}.
\end{rem}

The first of the two propositions used in the iterative argument is a variation of a theorem in \cite{BK}, \cite{K} and \cite{K2}.  This proposition may be thought of as the base case.

\begin{prop} Assume that conditions (\ref{epde})-(\ref{L2lBd}) hold.  For $\disp |x_0| \ge \left\{\begin{array}{ll} \max\set{\frac{C_2^3 \pr{1 + \sqrt{\abs{\la}}}^3}{4 C_3 w^2\pr{3}A_1^2}, 2} & \textrm{if}\; W \equiv 0 \\ \max\set{\frac{C_2\pr{1 + \sqrt{\abs{\la}}}}{4 C_3 A_2^2}, \frac{w\pr{3} A_1}{4 C_3 A_2^3}, 2} & \textrm{if}\; W \not\equiv 0  \end{array}\right.$,
$$\int_{B_1(x_0)}\abs{u}^2 \ge C_5 \exp\pr{-C_4 |x_0|^{\be_1}\log |x_0|},$$
where $\disp \be_1 = \left\{\begin{array}{ll} 4/3 & \textrm{if}\; W \equiv 0 \\ 2 & \textrm{if}\; W \not\equiv 0 \end{array}\right.$.  Furthermore, $C_5 = C_5(n)$ and $\disp C_4 = C_4\pr{n, A_1, A_2}$.
\label{baseC}
\end{prop}

To prove this proposition, we will discuss the order of vanishing of a suitably normalized equation, as was done in \cite{K}.  To achieve this result, we will use ``3-ball'' inequalities, a technique that, according to \cite{K}, was first used by Hadamard for harmonic functions.  This proof will be presented in Section \ref{Van}.  Similar results appeared in \cite{B1}, {\cite{B2}}, \cite{B3}, where the authors used Carleman estimates and ``3-ball'' inequalities to establish order-of-vanishing results for linear elliptic equations with $C^1$ electric and magnetic potentials on compact smooth manifolds.  

The second of the two propositions is the step that will be iterated.  Before we may state the proposition, we need to introduce a number of constants.
Let $\om$, $\de$ be small positive constants that will be specified below.  For any $\be > \be_0$, let 
\begin{align*}
\hg &= \left\{\begin{array}{ll} 
\be - 1 + 2P & \be \ge 1 + P - N + \om - \de \\ 
3P - N + \om - \de &  { 1 + P - N + \frac{\om}{3} - \de < \be < 1 + P - N + \om - \de} \\
3( \be - 1) + 2N & \be \le 1 + P - N + \frac{\om}{3} - \de 
\end{array}  \right., \\
a &= \left\{\begin{array}{ll} 1 & \be > 1 + P - N + \frac{\om}{3} - \de \\3 & \be \le 1 + P - N + \frac{\om}{3} - \de \end{array}  \right., \\
\ga &= \hg + a\de, \\
\be^\prime &= \left\{\begin{array}{ll} 2 - \frac{2P}{\ga} & \be > 1 + P - N + \frac{\om}{3} - \de \\ \frac{4}{3} - \frac{2N}{3\ga} & \be \le 1 + P - N + \frac{\om}{3} - \de \end{array}  \right..
\end{align*}
Notice that $\ga > \hat\ga > 1$ and $\be^\prime < \be$.
\begin{prop}
Assume that conditions (\ref{epde})-(\ref{L2lBd}) hold.  Let $\disp  \be > \be_0$.  Let $x_0, y_0 \in \R^n$ be such that $\disp \frac{x_0}{|x_0|} = \frac{y_0}{|y_0|}$ and $\disp |y_0| = |x_0|^{\ga}$. 
Suppose 
\begin{equation}
\int_{B_1(x_0)} \abs{u}^2 \ge C_5\exp(-C_4 |x_0|^{ \be} \log|x_0|).
\label{L2lBd2}
\end{equation}
Then there exists a constant $T_0\pr{n, N, P, \la, A_1, A_2, C_0, C_4, C_5}$ such that whenever $|x_0| \ge T_0$, 
$$\int_{B_1(y_0)} \abs{u}^2 \ge{C}_5\exp(-\tilde{C_4} |y_0|^{ \be^\prime} \log|y_0|),$$
where $\de$ and $\om$ are chosen so that $\disp \brac{\frac{18 C_4}{4^{3+P+ N/3}C_3 w\pr{\tfrac{5}{4}}^{2/3}\pr{A_1^{2/3} + A_2^2} c_n}\log |x_0|} = |x_0|^\de$ (where $c_n$ is a dimensional constant that will be specified in the proof) and $\disp |x_0|^{3P - N} + |x_0| = |x_0|^{3P - N + \om}$.  Furthermore, $\tilde C_4 = \tilde C_4 \pr{n, N, P, A_1, A_2}$.
\label{IH}
\end{prop}

\begin{rem}
It should be pointed out that $\tilde C_4$ does not depend on $C_4$.  In particular, $\tilde C_4$ is a universal constant that depends only on the constants associated with the PDE (\ref{epde}) and the Carleman estimate.  This fact allows us to iterate Proposition \ref{IH}.
\end{rem}

The proof of Proposition \ref{IH} uses the same ideas as the one in \cite{BK}, but is slightly more complicated.  In \cite{BK}, the solution function $u$ is shifted and scaled so that a distant point $x_0$ is sent to zero, and the distance between the new origin and the original origin is normalized.  A Carleman estimate is then applied to this new function and information about the function $u$ at zero is used to establish information about $u$ near $x_0$.  For our purposes, since there is no information about the decay rate of the potentials at zero, this exact technique does not improve on the former estimates.  So instead, we choose two points, $x_0$ and $y_0$ lying on the same ray.  We shift and scale the solution $u$ so that $y_0$ becomes the origin and $|x_0 - y_0|$ becomes 1.  An application of the Carleman inequality to this new function uses information about $u$ near $x_0$ to give an estimate { for $u$ near $y_0$}.  Since we choose $x_0 >> 1$, we can actually use information about the decay rates of the potentials.  This gives the improvement that we need.  The details are presented in Section \ref{IHProof}.  Since Propositions \ref{baseC} and \ref{IH} are so similar, they could have been proven in similar ways. Because we wanted to include the result on the order of vanishing and to present two proof techniques, we decided to prove these propositions in different ways.

We will now explain the idea behind the proof of Theorem \ref{MEst}.  Proposition \ref{baseC} allows us to estimate a lower bound for the $L^2$-size of the solution in a $1$-ball around some point, $x_1$.  With this initial estimate as our hypothesis, we may apply Proposition \ref{IH} to get a lower bound for the $L^2$-size of the solution in a $1$-ball around some point, $x_2$, where $|x_2| >> |x_1|$.  We will then use our estimate for $x_2$ to get an estimate for $x_3$, where $|x_3| >> |x_2|$.  Since $\be^\prime < \be$, the exponent decreases each time we apply Proposition \ref{IH}, so we may form a decreasing sequence of exponents.  To establish the desired lower bound for the $L^2$-size of the solution in a $1$-ball, we will continue to apply Proposition \ref{IH} until the exponent is within a ``reasonable'' neighborhood of $\be_0$.  Since the sequence of exponents does not actually converge to $\be_0$ (due to the presence of the $\de$ terms), the issue of getting ``reasonably'' close to $\be_0$ becomes rather delicate.  The full details of the proof will be presented in Section \ref{MEstProof}.

Recall that $\be_c := \max\set{\frac{4-2N}{3}, 2-2P}$.  The reader may have noticed that there are no results for the cases when $\be_c =1$, or rather, when $\min\set{N, P} = 1/2$.  If we try the same iterative proof that works for $\be_c \ne 1$, we fail.  The reason for this is that the decreasing sequence of exponents has polynomial decay when $\be_c = 1$.  For $\be_c \ne 1$, the sequence has exponential decay.  The fact that the decay is much slower for $\be_c = 1$ means that we cannot reach a ``reasonable'' neighborhood of $\be_0$ while maintaining the other conditions that are required for our propositions to hold true.  Another approach to the case where $\be_c = 1$ is to consider the limit as $\be_c \downarrow 1$.  If $\be_c = 1$, then $V$ and $W$ satisfy the hypotheses for Theorem \ref{MEst}(\ref{partA}) for any $\be_c > 1$.  Therefore, for any $\eps > 0$, $\mathbf{M}( R ) \ge \tilde C_5\exp\pr{-C_7 R^{1+\eps}\log R ^{C_6}}$, where $\tilde C_5$ and $C_7$ are bounded, $C_6 = C_6(1 + \eps)$.  However, a close inspection of the proof shows that $C_6 \lesssim \frac{1}{\eps}$, so $\disp \lim_{\eps \to 0^+}\brac{R^\eps \pr{\log R}^{c/\eps}} \to \iny$ and we cannot establish any result for $\be_c = 1$ by looking at the limiting behavior of the result from Theorem \ref{MEst}(\ref{partA}).

The following theorem shows that, under reasonable conditions, there are constructions that prove that Theorem \ref{MEst} is sharp (up to logarithmic factors).  

\begin{thm}
For any $\la \in \C$, $N, P \ge 0$ chosen so that either 
\begin{enumerate}[(a)]
\item $\disp \be_0 = \be_c > 1$ and $n = 2$ or 
\label{consa}
\item $\be_c < 1$ and $\la \notin \R_{\ge 0}$,
\label{consb}
\end{enumerate}
there exist complex-valued potentials $V$ and $W$ (at least one of which is equal to zero) and a non-zero solution $u$ to (\ref{epde}) such that 
\begin{equation}
|V(x)| \le C{\langle x\rangle^{-N}},
\label{vBd2}
\end{equation}
\begin{equation}
|W(x)| \le C{\langle x\rangle^{-P}}.
\label{wBd2}
\end{equation}
Furthermore, $$|u(x)| \le C\exp\pr{-c|x|^{\be_0}\pr{\log |x|}^A},$$
for some constant $A \in \set{-1, 0}$.
\label{cons}
\end{thm}

To prove Theorem \ref{cons}(\ref{consa}), we will use a construction similar to that of Meshkov in \cite{M}.  However, to account for eigenvalues and decaying potentials, our construction is even more complicated.  The proof of Theorem \ref{cons}(\ref{consb}) is much simpler; it relies on a lemma which is proved by mathematical induction.  The idea behind the constructions for Theorem \ref{cons}(\ref{consb}) is based on the following observation: For $\la \in \C$ with $\arg \la \in [-\pi, \pi) \setminus \set{0}$, the function $u_1\pr{r} = \exp\pr{\sgn\pr{\arg \la}\sqrt{-\la} r}$ satisfies an equation of the form (\ref{epde}) with either $V = C r^{-1}$ and $W \equiv 0$, or $V \equiv 0$ and $W = C r^{-1}$.  Furthermore, $\abs{u_1\pr{r}} \lesssim \exp\pr{-C r}$.  We then notice that  if we add a carefully-chosen logarithmic term to the exponent of $u_1$, this new function satisifies equations of the form (\ref{epde}) with a potential that decays like $Cr^{-2}$.  The careful addition of more lower order terms to the exponent drives the potentials to decay faster and faster. Induction is used to establish each additional term in the exponent. { The eigenfunctions that are constructed in the proof of Theorem \ref{cons}(\ref{consb}) are radial. } It is interesting to note that the constructions for Theorem \ref{cons}(\ref{consb}) are real-valued when the eigenvalue $\la$ is real-valued.  This is not the case for Theorem \ref{cons}(\ref{consa}) in which the constructions rely heavily on complex values, even when the eigenvalues are real. These proofs are presented in Section \ref{Mesh}. 

We will now explain why the extra restriction on $\la$ for the case when $\be_c < 1$ is reasonable (when we restrict to $W \equiv 0$).  A classical result of Kato in \cite{Ka} shows that if $\disp \lim_{x\to\iny} |x| |V(x)| = 0$, then the operator $-\LP + V$ has no positive eigenvalues.  This rules out the possibility of constructing eigenfunctions for $\la > 0$ and $\be_c = \frac{4-2N}{3} < \frac{2}{3}$.  To eliminate the other possibilities, we need to use the following lemma, which is a Corollary to Theorem 2.4 in \cite{FHHH}.

\begin{lem}
Suppose $V$ is $o\pr{r^{-1/2}}$ on $\R^n\setminus B_R\pr{0}$.  If $-\LP u + V u = \la u$, then either $\disp e^{\al r}u \notin L^2\pr{\R^n\setminus B_R\pr{0}}$ for every $\al > \sqrt{\max\set{-\la,0}}$, or $u$ vanishes outside a compact set.
\label{noCons}
\end{lem}

Assume that we could construct a non-zero eigenfunction $u$ such that $-\LP u + V u = \la u$, where $\la \ge 0$, $\abs{V(x)} \le \langle x \rangle ^{-N}$,  for some $N > 1/2$ and $\abs{u(x)} \lesssim \exp\pr{-c |x|}$ for some $c > 0$.  Since $\be_c < 1$, we are in the case that Theorem \ref{cons} excludes.  Without loss of generality, we may assume that $u(0) = 1$.  Since $V$ is $o\pr{r^{-1/2}}$ on $\R^n\setminus B_R\pr{0}$ for every $R \ge 0$, we may apply Lemma \ref{noCons}.  If $u$ vanishes outside of a compact set, then by Theorem \ref{MEst}(\ref{partB}), $u$ must be identically zero, which is a contradiction to our original assumption.  Therefore, we must have that $\disp e^{\al r}u \notin L^2\pr{\R^n\setminus B_R\pr{0}}$ for every $\al > 0$.  Since this is clearly not true for any $\al \in \pr{0, c}$, we get another contradiction.  It follows that  no such eigenfunction can exist. \\

The paper is organized as follows.  Section \ref{carEst} presents the Carleman estimate then proves it in the style of Donnelly and Fefferman.  In Section \ref{Van}, the proof of Proposition \ref{baseC} will be presented after a result on the order of vanishing is established.  The proof of Proposition \ref{IH} will be given in Section \ref{IHProof}.  Section \ref{MEstProof} is devoted to explaining the iterative argument that proves Theorem \ref{MEst}.  Although the arguments are similar, $\be_c > 1$ and $\be_c <1$ will be considered separately.  We will conclude Section \ref{MEstProof} with a technical discussion of what happens when $\be_c = 1$.  And in Section \ref{Mesh}, the constructions that prove Theorem \ref{cons} will be presented.  First we will give the Meshkov-type constructions that prove Theorem \ref{cons}(\ref{consa}), then we will prove a lemma that gives the necessary functions to prove Theorem \ref{cons}(\ref{consb}).  Some useful but rather technical lemmas and their proofs may be found in the appendices.  Appendix \ref{AppA} includes a couple of results that are required in the proofs of Propositions \ref{baseC} and \ref{IH}.  In Appendix \ref{AppB}, we present a number of estimates for the elements in the sequences $\set{\ga_j}$ and $\set{\be_j}$.  These results are applied often in the proof of Theorem \ref{MEst}.  Appendix \ref{AppC} includes estimates for the product function $\Ga_j = \ga_1\ldots \ga_j$.  Section \ref{MEstProof} makes use of the results from Appendix \ref{AppC}.  Finally, Appendix \ref{AppD} presents a technical lemma that is used in the Meshkov-type constructions.

%
%
\section{The Carleman estimate for $\LP + \la$}
\label{carEst}

In this section, we will use the same notation that is used in Donnelly and Fefferman's papers: We will write our weight function $w\pr{r}$ as $\bar{r}$.  The following lemma uses the ideas and the notation of their papers to establish a lower bound for a weighted $L^2$-norm of $\pr{\LP + \la}u$.  This estimate differs from those of Donnelly and Fefferman since the exponents on $\al$ are different, a gradient term appears on the left, and there are no integrals over small balls.  However, we do still have the restriction that $\al > C_2\pr{1 + \sqrt{\abs{\la}}}$.

\begin{lem} Let $M$ be a smooth, connected { Riemannian manifold with bounded geometry}.  Let $p \in M$.  There exist constants $C_1, C_2, C_3, h$, depending only on $M$, and an increasing function $ \bar r\pr{r}$, $0 < r < h$, where $h$ { depends on $M$}, so that 
$$\frac{1}{C_1} \le \frac{ \bar r}{r} \le C_1,$$ 
and such that for all $u \in C^\iny_0\left(B(p, h)\setminus\{p\}\right)$, $\al > C_2\pr{1 +\sqrt{\abs{\la}}}$,
$$\al^3\int \bar r^{-2-2\al}\abs{u}^2 + \al \int \bar r^{-2\al}|\gr u|^2 \le C_3 \int  \bar r^{-2\al}\abs{\LP u + \la u}^2.$$
\label{DFCar}
\end{lem}

To prove this lemma, we will follow the approach of Donnelly and Fefferman from \cite{DF2} which built on the set-up from \cite{DF1}.  { For our purposes, we only require this result for  functions that are compactly supported on balls of arbitrary radius in $ \R^n$.  A scaling argument allows us to pass from Lemma \ref{DFCar} to Corollary \ref{CT} (with appropriate renaming of constants).}  

\begin{proof}
In $B\pr{p, h}$, we have geodesic polar coordinates $(r, t_1, \ldots, t_{n-1})$, where $r = r(x)$ and $(t_1, \ldots, t_{n-1})$ denotes the standard coordinates on $S^{n-1}$.  The metric is $\disp ds^2 = dr^2 + r^2 \ga_{ij}dt_idt_j$ and the volume element is given by $\disp d vol = r^{n-1}\sqrt{\ga} dr dt$, where $\ga = \det\pr{\ga_{ij}}$.  Note that in Euclidean space, $\ga_{ij}$ is independent of $r$.  Now we introduce a local conformal change in the metric and volume element.  For a large constant $\nu > 0$, let $\disp \bar{g}_{ij} = \exp\pr{-2\nu r^2}g_{ij}$.  The geodesic lines starting from $p$ for this new metric and the original metric coincide.  We see that $\disp \bar{r} = \int_0^r e^{-\nu s^2}ds$.  Near the origin, $\disp \bar{r} = r + \bigO\pr{r^3}$.  Let $\disp \psi = \exp\pr{\tfrac{2}{3}\nu (n-2)\bar{r}^2}$.  In geodesic polar coordinates, the metric is $\disp \ol{ds}^2 = \ol{dr}^2 + \bar{r}^2 \ol{\ga}_{ij}dt_i dt_j$ and we set the modified volume element to be $\disp \ol{d vol} = \bar{r}^{n-1}\psi \sqrt{\ol{\ga}}\,\ol{dr}dt$.  Let $\ol{\LP}$ denote the Laplacian associated with $\disp \pr{\ol{g}_{ij}}$ and set $\ol{\la} = \la \exp\pr{2 \nu r^2}$.  For $\disp u \in C_{0}^\iny\pr{B(p,h)}$ with $u = 0$ near $p$, and $\al > C_2 \pr{1 + \sqrt{|\la|}}$, we want to establish a lower bound for 
\begin{equation}
I = \int \bar{r}^{-2\al} \abs{\pr{\bar{\LP} + \bar{\la}}u}^2 \ol{d vol}.
\label{Idef}
\end{equation}
Let $\bar{r} = e^{-\rho}$, $\disp \be = \al - \frac{n}{2} + 2$ , $u = e^{-\be\rho} w$, $\disp \te = \frac{\del}{\del\rho}\pr{\log \sqrt{\bar{\ga}}}$, $\disp w^\pri = \frac{\del w}{\del \rho}$ and $\disp w^{\pri\pri} = \frac{\del^2 w}{\del \rho^2}$.  We may use polar coordinates to express $I$ as
\begin{equation}
I = \int \abs{w^{\pri\pri} - \pr{n-2+2\be - \te}w^\pri + \be\pr{\be + n - 2 - \te}w + \LP_\rho w + \bar{\la} e^{-2\rho} w}^2 \dV,
\label{Ipoldef}
\end{equation}
where $\disp \LP_{\rho}w = \frac{1}{\sqrt{\bar{\ga}}}\frac{\del}{\del t_i}\pr{\sqrt{\bar{\ga}}\; \ga^{ij}\frac{\del w}{\del t_j}}$.  By removing all $\te$ terms, we get $I_0$:
\begin{equation}
I_0 = \int \abs{w^{\pri\pri} - \pr{n-2+2\be}w^\pri + \be\pr{\be + n - 2}w + \LP_\rho w + \bar{\la} e^{-2\rho} w}^2 \dV.
\label{I0def}
\end{equation}
By the triangle inequality, we have that $\disp I \ge \frac{1}{2} I_ 0 - I_1$, where $\disp I_1 = \int \te^2 \abs{w^\pri - \be w}^2 \dV$.  As is explained in \cite{DF1}, $I_1 \le \frac{1}{4} I_0$ for $h$ sufficiently small.  Therefore, $I \ge \frac{1}{4} I_0$. \\

At this point, the proof begins to diverge from that of Donnelly and Fefferman in \cite{DF2}.  In their proof, they introduce an unknown function called $f$.  This proof does not require $f$.  We write $I_0 = I_2 + I_3 + I_4$, where
\begin{align}
I_2 &= \int \abs{w^{\pri\pri} + \be\pr{\be + n - 2}w + \LP_\rho w + \bar{\la} e^{-2\rho} w}^2 \dV, \\
I_3 &= \pr{2\be + n-2}^2 \int \abs{w^\pri }^2 \dV, \\
I_4 &= -2\pr{2\be + n-2} \int w^\pri\brac{w^{\pri\pri} + \be\pr{\be + n - 2}w + \LP_\rho w + \bar{\la} e^{-2\rho} w} \dV.
\end{align}
Since $I_2$ and $I_3$ are both non-negative, then $I_0 \ge I_4$ and our interest lies in $I_4$.  If we integrate by parts in $t$, we see that
$$I_4 = -2\pr{2\be + n-2} \int \brac{w^\pri w^{\pri\pri} + \be\pr{\be + n - 2} w^{\pri}w - \bar{\ga}^{ij}\frac{\del w^\pri}{\del t_i}\frac{\del w}{\del t_j} + \bar{\la}e^{-2\rho}w^\pri w} \dV.$$ 
Then we integrate by parts in $\rho$ and let $\disp \mu = \brac{\log\pr{\psi \sqrt{\bar{\ga}}}}^\pri$ to get
$$I_4 = \pr{2\be + n-2} \int \brac{\pr{w^\pri}^2 \mu + \be\pr{\be + n - 2}w^2 \mu - \frac{\pr{\psi \sqrt{\bar \ga}\; \bar \ga^{ij}}^\pri}{\psi \sqrt{\bar \ga}}\frac{\del w}{\del t_i}\frac{\del w}{\del t_j} + \frac{\pr{\bar\la e^{-2\rho}\psi \sqrt{\bar \ga}}^\pri}{\psi \sqrt{\bar \ga}} w^2} \dV.$$
By Lemma 2.3 in \cite{DF1}, if $\bar r < h$, then $\disp \mu \ge \nu e^{-2\rho}$ and $\disp - \pr{\bar \ga^{ij}}^\pri \ge \pr{\nu e^{-2\rho} + \mu}\bar \ga^{ij}$.  Therefore,
\begin{equation}
I_4 \ge \pr{2\be + n-2} \int \set{\nu e^{-2\rho}\brac{\be\pr{\be + n - 2}w^2 + \pr{w^\pri}^2 + \bar\ga^{ij}\frac{\del w}{\del t_i}\frac{\del w}{\del t_j}} + \frac{\pr{\bar\la e^{-2\rho}\psi \sqrt{\bar \ga}}^\pri}{\psi \sqrt{\bar \ga}} w^2} \dV.
\label{I4est}
\end{equation}
We will now try to establish an estimate for the last term.  Since
\begin{align*}
\frac{\pr{\bar\la e^{-2\rho}\psi \sqrt{\bar \ga}}^\pri}{\psi \sqrt{\bar \ga}} &= \la \frac{\pr{\exp\pr{2\nu r^2} e^{-2\rho}\psi \sqrt{\bar \ga}}^\pri}{\psi \sqrt{\bar \ga}} \\
&= \la{\exp\pr{2\nu r^2} e^{-2\rho}}\pr{4\nu r \frac{\del r}{\del \rho} - 2 + \mu} \\
&= \la{\exp\pr{2\nu r^2} e^{-2\rho}}\pr{-4\nu \bar r r \exp\pr{\nu r^2} - 2 + \mu}
\end{align*}
and $r, \bar r \le h$, then
$$\abs{\frac{\pr{\bar\la e^{-2\rho}\psi \sqrt{\bar \ga}}^\pri}{\psi \sqrt{\bar \ga}}} \le c|\la| \nu e^{-2\rho},$$
where $c$ depends on $h$ and $\nu$.
If we choose $C_2 \ge \sqrt{6c}$, then since we assume that $\al > C_2\pr{1 + \sqrt{|\la|}}$ we get $\disp |\la| \le \frac{\al^{2}}{C_2^2} \le \frac{\al^2}{6c}$.  Returning to (\ref{I4est}), we see that if $C_2 \ge \frac{n}{4}\pr{\frac{n}{2} - 2}$ then $\be\pr{\be + n - 2} \ge \al^2$, so
\begin{align*}
I_4 
&\ge \frac{5}{3}\al^3 \int \nu e^{-2\rho} w^2 \dV + 2\al \int \nu e^{-4\rho} \norm{\bar \gr w}^2_{\bar g} \dV.
\end{align*}
Since $\disp \bar\gr w = \bar r^{-\be} \bar \gr u - \be \frac{\bar \gr \bar r}{\bar r} w$ then
\begin{align*}
\norm{\bar\gr w}^2_{\bar g} &= \norm{\bar r^{-\be} \bar \gr u}^2_{\bar g} + \norm{\be \frac{\bar \gr \bar r}{\bar r} w}^2_{\bar g} - 2\bar g\pr{\bar r^{-\be} \bar \gr u, \be \frac{\bar \gr \bar r}{\bar r} w} \\
&= \bar r^{-2\be}\norm{\bar \gr u}^2_{\bar g} + \be^2 e^{2\rho}w^2 - 2\be e^{\rho}\bar r^{-\be} \frac{\del u}{\del \bar r}  w \\
&\ge \bar r^{-2\be}\norm{\bar \gr u}^2_{\bar g} + \be^2 e^{2\rho}w^2 - \frac{2\bar r^{-2\be}}{3}\pr{\frac{\del u}{\del \bar r}}^2 - \frac{3}{2} \be^2 e^{2\rho}w^2.
\end{align*}
Since $\disp \norm{\bar \gr u}^2_{\bar g} = \pr{\frac{\del u}{\del \bar r}}^2 + \frac{\bar \ga^{ij}}{\bar r ^2}\frac{\del u}{\del t_i}\frac{\del u}{\del t_j}$ and $\bar \ga^{ij}$ is a positive matrix, then because $\be \le \al$, we get 
\begin{align*}
\norm{\bar\gr w}^2_{\bar g} &\ge  \frac{\bar r^{-2\be}}{3}\norm{\bar \gr u}^2_{\bar g} - \frac{\al^2}{2} e^{2\rho}w^2
\end{align*}
and therefore,
\begin{equation}
\int \bar{r}^{-2\al} \abs{\pr{\bar{\LP} + \bar{\la}}u}^2 \ol{d vol} \ge \frac{\nu}{6}\al^3 \int \bar r^{-2\al - 2} \abs{u}^2 \; \ol{d vol} + \frac{\nu}{6}\al \int \bar r^{-2\al} \norm{\bar \gr u}^2_{\bar g} \; \ol{d vol}
\label{Ibarest}
\end{equation}
Now we need to establish the corresponding estimate for the original metric. That is, we need to eliminate the bars.  Since $\bar g = \phi g$, where $\phi = \exp\pr{-2\nu r^2}$ then $\norm{\bar \gr u }^2_{\bar g} = \phi^{-1} \norm{\gr u}^2_{g} \ge \norm{\gr u}^2_{g}$.  As was shown in \cite{DF1}, $K \gtrsim I$, where
$$K = \int \bar r^{-2\al}\abs{\pr{\LP + \la}u}^2 \ol{d vol},$$
giving the result.
\end{proof}

%
%
\section{Order of vanishing and a proof of Proposition \ref{baseC}}
\label{Van}

For a suitably normalized version of (\ref{epde}), we are interested in determining a lower bound of the form $$m(r) = \max_{|x| \le r}|u| \ge a_1 r^{a_2\be},$$ where $a_1, a_2$ are constants.  This estimate is interesting on its own, but can also be used to provide a proof of Proposition \ref{baseC}.

\begin{prop}
Suppose $u$ is a solution to $-\LP u + W\cdot\gr u + Vu = \tilde\la u $ in $B_{6}(0)$ with $\tilde \la = M^2 \la$, $||V||_\iny \le A_1 M^2$, $||W||_\iny \le A_2 M$ and $||u||_{\iny} \le C_0$.  In addition, assume that $\disp \norm{u}_{L^\iny\pr{B_1(0)}}\ge 1$. If $ M \ge \max\set{\frac{C_2 \pr{1 + \sqrt{\abs{\la}}}}{4 C_3 A_2^2}, \frac{w\pr{3} A_1}{4 C_3 A_2^3}, 1}$, then 
$$m(r) \ge a_1 r^{a_2M^{2}}.$$  
That is, with the notation above, $\be = M^2$.  Furthermore, $a_1 = a_1\pr{n}$ and $a_2 =  C\pr{n} A_2^2$.
\label{mest}
\end{prop}

\begin{rem}
If $W \equiv 0$ and $ M \ge \max\set{\frac{C_2^3 \pr{1 + \sqrt{\abs{\la}}}^3}{4 C_3 w^2\pr{3}A_1^2}, 1}$, then $m(r) \ge a_1 r^{a_2M^{4/3}}$, where $a_2 = C(n) A_1^{2/3}$.  This argument (for the case when $\la = 0$) is presented in \cite{K}.
\end{rem}

We use ``3 ball inequalities'' to prove Proposition \ref{mest}.

\begin{pf}[Proof of Proposition \ref{mest}]
Let $R_1 = 6$, $r_1 = 2$ and $2r_0 << r_1$.  \\
Let $\brac{a,b} = B_b(0)\setminus B_a(0)$ denote the spherical shell centered at zero with outer radius $b$ and inner radius $a$. \\
Set $K_1 = \brac{\tfrac{3}{2}r_0, \tfrac{1}{2}R_1}$, $K_2 = \brac{r_0, \tfrac{3}{2}r_0}$ and $K_3 = \brac{\tfrac{1}{2}R_1, \tfrac{3}{4}R_1}$. \\
Let $\zeta \in C^\iny_0(B_{R_1})$ be a smooth cutoff function such that $\zeta \equiv \left\{ \begin{array}{rl} 1 & \mathrm{on}\; K_1 \\ 0 & \mathrm{on} \; \brac{0, r_0}\cup\brac{\tfrac{3}{4}R_1, R_1} \end{array} \right.$, $|\gr\zeta| \le \left\{ \begin{array}{rl} \frac{C}{r_0} & \mathrm{on}\; K_2 \\ \frac{C}{R_1} & \mathrm{on} \; K_3 \end{array} \right.$ and $|\gr^2\zeta| \le \left\{ \begin{array}{rl} \frac{C}{r_0^2} & \mathrm{on}\; K_2 \\ \frac{C}{R_1^2} & \mathrm{on} \; K_3 \end{array} \right.$.  \\
If we assume that $\al > C_2\pr{1 +\sqrt{\abs{\tilde \la}}} = C_2\pr{1 + M\sqrt{\abs{\la}}}$, then we may apply Corollary \ref{CT} to $f = \zeta u$ to get
$$\al^3\int w^{-2-2\al}\abs{\zeta u}^2 + \al\int w^{-2\al}|\gr (\zeta u)|^2 \le C_3\int w^{-2\al}|\LP (\zeta u) + \tilde\la \zeta u|^2.$$
Since $|\LP u + \tilde \la u| \le A_1 M^2|u| + A_2 M|\gr u|$ and $\zeta \equiv 1$ on $K_1$, we see that
$$\al^3\int_{K_1} w^{-2-2\al}\abs{u}^2 + \al\int_{K_1} w^{-2\al}|\gr u|^2 \le 2C_3\int_{K_1} w^{-2\al}\pr{A_1^2M^4|u|^2 + A_2^2M^2|\gr u|^2} + J,$$
where $\disp J = C_3\int_{K_2\cup K_3} w^{-2\al}|\LP (\zeta u) + \tilde\la \zeta u|^2.$  
Thus,
\begin{align*}
\al^3\int_{K_1} w^{-2-2\al}\abs{u}^2 &+ \al\int_{K_1} w^{-2\al}| \gr u|^2 \le 2C_3w^2\pr{\tfrac{1}{2}R_1} A_1^2 M^4\int_{K_1} w^{-2-2\al}|u|^2 + 2C_3 A_2^2 M^2\int_{K_1}w^{-2\al}|\gr u|^2 + J.
\end{align*}
If $\al > 4C_3 A_2^2 M^2$, then for $ M > \max\set{\frac{C_2 \pr{1 + \sqrt{\abs{\la}}}}{4 C_3 A_2^2}, 1}$, $\al > C_2\pr{1 +\sqrt{\abs{\tilde \la}}} = C_2\pr{1 + M\sqrt{\abs{\la}}}$ holds. Furthermore, as long as $ M > \frac{w\pr{R_1/2} A_1}{4 C_3 A_2^3} $, then $\al^3 > 4C_3w^2\pr{\tfrac{1}{2}R_1} A_1^2 M^4$, so we may absorb the first two terms on the right into the left hand side to get 
\begin{equation}
\frac{\al^3}{2}\int_{K_1} w^{-2-2\al}\abs{u}^2 + \frac{\al}{2}\int_{K_1} w^{-2\al}|\gr u|^2 \le J.
\label{in1}
\end{equation}
Since $|\LP(\zeta u) + \tilde \la \zeta u| \le A_1 M^2 |u| + A_2 M |\gr u| + 2|\gr \zeta||\gr u| + |\LP \zeta||u|$, then
$$|\LP(\zeta u) + \tilde \la \zeta u| \le \left\{ \begin{array}{ll}  
A_1 M^2|u| + A_2 M|\gr u| + \frac{C}{r_0}|\gr u| +\frac{C}{r_0^2}|u| & \mathrm{on} \; K_2 \\
A_1 M^2|u| + A_2 M|\gr u| + \frac{C}{R_1}|\gr u| +\frac{C}{R_1^2}|u| & \mathrm{on} \; K_3 .\end{array} \right. $$
Therefore,
\begin{align*}
J 
  \le & 4C_3\brac{A_1^2 M^4 + \frac{C}{r_0^4}}w\pr{r_0}^{-2\al}\int_{K_2}|u|^2 + 
  4C_3\brac{A_2^2 M^2 + \frac{C}{r_0^2}}w\pr{r_0}^{-2\al}\int_{K_2}|\gr u|^2 + \\
  & 4C_3\brac{A_1^2 M^4 + \frac{C}{R_1^4}}w\pr{\tfrac{R_1}{2}}^{-2\al}\int_{K_3}|u|^2 +
  4C_3\brac{A_2^2 M^2 + \frac{C}{R_1^2}}w\pr{\tfrac{R_1}{2}}^{-2\al}\int_{K_3}|\gr u|^2 .
\end{align*}
By Cacciopoli (Lemma \ref{Cacc} in Appendix \ref{AppA}), 
\begin{align*}
\int_{K_2}|\gr u|^2 &\le C\brac{\frac{C}{r_0^2} + M^2\pr{\abs{\la} + A_1 + A_2^2}}\int_{B_{2r_0}\setminus B_{r_0/2}}|u|^2, \quad \mathrm{and} \\
\int_{K_3}|\gr u|^2 &\le C\brac{\frac{C}{R_1^2} + M^2\pr{\abs{\la} + A_1 + A_2^2}}\int_{B_{R_1}\setminus B_{R_1/4}}|u|^2 .
\end{align*}
Let $K_4 = \{ x \in K_1 : |x| \le r_1\}$.  Since $\al >> 1$, 
\begin{align*}
\int_{K_4} \abs{u}^2 \le &w\pr{r_1}^{2\al +2}\int_{K_4} w^{-2-2\al}\abs{u}^2 \\
\le &w\pr{r_1}^{2\al + 2}\brac{\frac{\al^3}{2}\int_{K_1} w^{-2-2\al}\abs{u}^2 + \frac{\al}{2}\int_{K_1} w^{-2\al}|\gr u|^2}.
\end{align*}
Then, if $\al > 4C_3 A_2^2 M^2$, by (\ref{in1}) and the estimates on $J$,
\begin{align*}
&\int_{K_4} \abs{u}^2 
\le 4C_3w\pr{r_1}^{2\al + 2}\brac{A_1^2 M^4 + \frac{C}{r_0^4}}w\pr{r_0}^{-2\al}\int_{K_2}|u|^2 
+ CC_3 w\pr{r_1}^{2\al + 2}\brac{\pr{\abs{\la} + A_1 + A_2^2}^2M^4 + \frac{1}{r_0^4}}w\pr{r_0}^{-2\al}\int_{B_{2r_0}\setminus B_{r_0/2}}|u|^2 \\
  & + 4C_3 w\pr{r_1}^{2\al + 2}\brac{A_1^2 M^4 + \frac{C}{R_1^4}}w\pr{\tfrac{R_1}{2}}^{-2\al}\int_{K_3}|u|^2
+ CC_3 w\pr{r_1}^{2\al + 2}\brac{\pr{\abs{\la} + A_1 + A_2^2}^2M^4 + \frac{1}{R_1^4}}w\pr{\tfrac{R_1}{2}}^{-2\al}\int_{B_{R_1}\setminus B_{R_1/4}}|u|^2 .
\end{align*}
Let $\eta = ||u||_{L^2(B_{2r_0})}$, $V = ||u||_{L^2(B_{R_1})}$ to get
$$\int_{|x| \le r_1}|u|^2 \le CC_3\brac{1 + \pr{\abs{\la} + A_1 + A_2^2}^2} \set{\brac{\frac{w\pr{r_1}}{w\pr{r_0}}}^{2\al}\eta^2 w\pr{r_1}^2 \pr{M^4 + \tfrac{1}{r_0^4}} + \brac{\frac{w\pr{r_1}}{w\pr{\tfrac{R_1}{2}}}}^{2\al}V^2 w\pr{r_1}^2 \pr{M^4 + \tfrac{1}{R_1^4}}}.$$
Now define $\eta_1^2 = \eta^2 w\pr{r_1}^2 \pr{ M^4 + \tfrac{1}{r_0^4}}$, $V_1^2 = V^2 w\pr{r_1}^2 \pr{ M^4 + \tfrac{1}{R_1^4}}$ and $A^2 = CC_3\brac{1 + \pr{\abs{\la} + A_1 + A_2^2}^2}$.  If $\al > 4C_3 A_2^2 M^2$, then
\begin{equation}
\int_{|x|\le r_1}\abs{u}^2 \le A^2\brac{\frac{w\pr{r_1}}{w\pr{r_0}}}^{2\al}\eta_1^2 + A^2\brac{\frac{w\pr{r_1}}{w\pr{\tfrac{R_1}{2}}}}^{2\al}V_1^2.
\label{L2bd}
\end{equation}
Let $k_0$ be given by $\frac{1}{k_0} = 1 + \frac{\log\brac{\frac{w\pr{r_1}}{w\pr{r_0}}}}{\log\brac{\frac{w\left(R_1/2\right)}{w\pr{r_1}}}}$, so that $\frac{1}{k_0} \simeq \log\frac{1}{r_0}$.  Set $\al_1 = \frac{k_0}{2\log\brac{\frac{w\pr{R_1/2}}{w\pr{r_1}}}}\log\brac{\left(\frac{V_1}{\eta_1} \right)^2}$.  Then, if $\al_1 > 4C_3 A_2^2 M^2$, we can use $\al_1$ in (\ref{L2bd}) to show that
\begin{align*}
||u||_{L^2(B_{r_1})} &\le \sqrt{2}A\eta_1^{k_0}V_1^{1-k_0}  = \sqrt{2}A\brac{||u||_{L^2(B_{2r_0})}w\pr{r_1}\pr{ M^4 + \tfrac{1}{r_0^4}}^{1/2}}^{k_0}\brac{||u||_{L^2(B_{R_1})}w\pr{r_1}\pr{ M^4 + \tfrac{1}{R_1^4}}^{1/2}}^{1-k_0}.
\end{align*}
Otherwise, if $\al_1 \le 4C_3 A_2^2 M^2$, since $||u||_{L^2(B_{r_1})} \le V$ and $\frac{k_0}{2\log\brac{\frac{w\pr{R_1/2}}{w\pr{r_1}}}}\log\brac{\left(\frac{V_1}{\eta_1} \right)^2} < 4C_3 A_2^2 M^2$, then $V_1^2 \le \eta_1^2\exp\pr{\frac{A_3 M^2}{k_0}}$, where $ A_3 = 8C_3\log\brac{\frac{w\pr{R_1/2}}{w\pr{r_1}}} A_2^2.$  So in this case,
$$||u||_{L^2(B_{r_1})} \le \pr{ \frac{ M^4 + \tfrac{1}{r_0^4} }{ M^4 + \tfrac{1}{R_1^4} } }^{1/2}||u||_{L^2(B_{2r_0})}\exp\pr{\frac{A_3 M^2}{k_0}}.$$ 
Combining these estimates, we obtain:
\begin{align*}
||u||_{L^2(B_{r_1})} \le &\sqrt{2}A\brac{||u||_{L^2(B_{2r_0})}w\pr{r_1}\pr{ M^4 + \tfrac{1}{r_0^4}}^{1/2}}^{k_0}\brac{||u||_{L^2(B_{R_1})}w\pr{r_1}\pr{ M^4 + \tfrac{1}{R_1^4}}^{1/2}}^{1-k_0} \\
& + \pr{ \frac{ M^4 + \tfrac{1}{r_0^4} }{ M^4 + \tfrac{1}{R_1^4} } }^{1/2}||u||_{L^2(B_{2r_0})}\exp\pr{\frac{A_3 M^2}{k_0}}.
\end{align*}
By elliptic regularity {(see \cite{HL}, for example)}, $||u||_{L^\iny(B_1)} \le C_nM^n||u||_{L^2(B_2)}$, for $M > 1$.
Thus, $||u||_{L^\iny(B_1)} \le I + II$, where
\begin{align*}
I = &\sqrt{2}AC_nM^n\brac{||u||_{L^2(B_{2r_0})}w\pr{r_1}\pr{ M^4 + \tfrac{1}{r_0^4}}^{1/2}}^{k_0}\brac{||u||_{L^2(B_{R_1})}w\pr{r_1}\pr{ M^4 + \tfrac{1}{R_1^4}}^{1/2}}^{1-k_0}, \\
II = &C_nM^n\pr{ \frac{ M^4 + \tfrac{1}{r_0^4} }{ M^4 + \tfrac{1}{R_1^4} } }^{1/2}||u||_{L^2(B_{2r_0})}\exp\pr{\frac{A_3 M^2}{k_0}}.
\end{align*}
By assumption, $||u||_{L^\iny(B_1)} \ge 1$. \\
If $I \le II$, then
\begin{align*}
1 &\le 2II \\
&\le 2C_nM^n\pr{ \frac{ M^4 + \tfrac{1}{r_0^4} }{ M^4 + \tfrac{1}{R_1^4} } }^{1/2}||u||_{L^2(B_{2r_0})}\exp\pr{\frac{A_3 M^2}{k_0}} \\
&\le 2C_nM^n\left(1 + \tfrac{1}{M^4r_0^4}\right)^{1/2}r_0^{n/2}\max_{|x| \le 2r_0}|u|\exp\pr{\frac{A_3 M^2}{k_0}} \\
&\le \tilde{C_n}\exp\pr{\frac{2A_3M^2}{k_0}}\max_{|x| \le 2r_0}|u| \\
&\le \tilde{C_n}r_0^{-C A_3M^2}\max_{|x| \le 2r_0}|u|,
\end{align*}
since $\frac{1}{k_0} \simeq \log(1/r_0)$.  Rearranging, we get
$$\max_{|x| \le 2r_0}|u| \ge \frac{r_0^{CA_3 M^2}}{\tilde{C_n}} ,$$
as desired. \\
If $II \le I$, then 
$$1 \le 2\sqrt{2}AC_nM^n\brac{||u||_{L^2(B_{2r_0})}w\pr{r_1}\pr{M^4 + \tfrac{1}{r_0^4}}^{1/2}}^{k_0}\brac{||u||_{L^2(B_{R_1})}w\pr{r_1}\pr{M^4 + \tfrac{1}{R_1^4}}^{1/2}}^{1-k_0}.$$
We raise both side to $1/k_0$ and use $||u||_{L^\iny} \le C_0$, to get
\begin{align*}
1 \le &(2\sqrt{2}AC_nM^n)^{1/k_0}||u||_{L^2(B_{2r_0})}w\pr{r_1}\pr{M^4 + \tfrac{1}{r_0^4}}^{1/2}\brac{||u||_{L^2(B_{R_1})}w\pr{r_1}\pr{M^4 + \tfrac{1}{R_1^4}}^{1/2}}^{1/k_0-1} \\
\le &(2\sqrt{2}AC_nM^n w(r_1))^{1/k_0}r_0^{n/2}||u||_{L^\iny(B_{2r_0})}M^2(C_nC_0R_1^{n/2})^{1/k_0 - 1}M^{2(1/k_0 - 1)} \\
\le & \pr{C_n A C_0}^{1/k_0}M^{(n+2)/k_0}||u||_{L^\iny(B_{2r_0})}.
\end{align*}
Since $\frac{1}{k_0} \simeq \log(1/r_0)$, we see that
$$1 \le \left( \frac{1}{r_0}\right)^{C\log\pr{C_n A C_0}} \left( \frac{1}{r_0}\right)^{C\log M} ||u||_{L^\iny(B_{2r_0})},$$
which gives a better bound than the first case and completes the proof.
\end{pf}

We now use Proposition \ref{mest} to prove Proposition \ref{baseC}.

\begin{pf}[Proof of Proposition \ref{baseC}]
Fix $x_0 \in \R^n$ so that $|x_0| = R$ and $\disp \mathbf{M}(R) = \inf_{|x_0| = R}\norm{u}_{L^2\pr{B_1(x_0)}}$.
Set $u_R(x) = u(Rx + x_0)$.  Then $||u_R||_{L^\iny}\le C_0$ and 
$$|\LP u_R + R^2 \la u_R| \le A_1 R^2|u_R| + A_2 R|\gr u_R| .$$ 
Note also that if $\widetilde{x_0} := -x_0/R$, then $| \widetilde{x_0}| = 1$ and $u_R(\widetilde{x_0}) = u(0) \ge 1$ so that $||u_R||_{L^\iny(B_1)} \ge 1$.  By our assumptions on $R$, we may apply Proposition \ref{mest} to $u_R$ with $M = R$ to get
\begin{align*}
\norm{u}_{L^\iny\pr{{B_{1/2}(x_0)}}} = & \norm{u_R}_{L^\iny\pr{B_{1/2R}(0)}}  \\
\ge &a_1(1/2R)^{a_2 R^2}  \qquad \textrm{(by Proposition \ref{mest})} \\
\ge &a_1\exp(-a_2 R^2\log 2R).
\end{align*}
Again, by elliptic regularity, $||u||_{L^\iny(B_{1/2}\pr{x_0})} \le C_nR^n||u||_{L^2(B_1\pr{x_0})}$, for $R > 1$.
If we rearrange then set $\disp C_5 = \frac{a_1}{C_n}$ and $C_4 = 2a_2 + n$, we get the desired result.
\end{pf}

\begin{rem}
If $W \equiv 0$, then the proof of Proposition \ref{baseC} is similar to the one given in \cite{K}.  With the same argument as above, but a different estimate for $m( r )$, we get an exponent of $4/3$ instead of $2$.
\end{rem}

%
%
\section{Proof of Proposition \ref{IH}}
\label{IHProof}

The proof presented below is a variation of the proof given in \cite{BK}.  In that proof, the idea was to pick $x_0$ with $|x_0| = R$ so that $\mathbf{M}( R ) = \norm{u(x)}_{L^\iny\pr{B_1(x_0)}}$.  Then you ``interchange $0$ and $x_0$" and ``rescale to $R = 1$''.  This provided a way of using the information about the behavior of the solution near $0$ to gain information about the behavior near $x_0$.  Proposition \ref{IH} uses information about the behavior near $x_0$ to gain information about the behavior near $y_0$.  Therefore, our idea is to ``shift" (interchange variables) and ``rescale to $S = 1$'', where $S = |y_0 - x_0|$.  We employ this more complicated approach because we want to be able to use information about the decay rates of the potentials, so we need to get away from zero.  

\begin{proof}[Proof of Proposition \ref{IH}]
Let $T = |x_0|$, $\disp S = |y_0 - x_0|$, $T^\prime = |y_0| = S + T$. \\
Let $\disp K_1 = \brac{\frac{1}{2S}, 1 + \frac{T}{2S}}$, $\disp K_2 = \brac{\frac{1}{4S}, \frac{1}{2S}}$, $\disp K_3=\brac{1+\frac{T}{2S},1+\frac{2T}{3S}}$.  \\
Choose a smooth cutoff function $\zeta$ so that $\zeta \equiv 1$ on $K_1$ and $\zeta \equiv 0$ on $\pr{K_1 \cup K_2 \cup K_3}^c$.  Then
$$|\gr \zeta | \lesssim \left\{\begin{array}{ll} S & \textrm{on}\;K_2 \\ \frac{S}{T} & \textrm{on}\;K_3  \end{array}\right., \quad  |\LP \zeta | \lesssim \left\{\begin{array}{ll} S^2 & \textrm{on}\;K_2 \\ \pr{\frac{S}{T}}^2 & \textrm{on}\;K_3  \end{array}\right..$$
Let $u_S(x) = u(y_0 + Sx)$, $\tilde \la = S^2 \la$.  Since
\begin{align*}
|y_0 + Sx| \ge |y_0| - S|x| \ge \left\{\begin{array}{ll}\frac{T}{2} & \textrm{on}\; K_1 \\ T & \textrm{on}\; K_2 \\ \frac{T}{3} & \textrm{on}\; K_3  \end{array} \right.,
\end{align*}
then
\begin{align*}
|\LP u_S + \tilde \la u_S| \le { S^2 |V(y_0 + Sx)| |u_S(x)| + S |W(y_0 + Sx)| |\gr u_S(x)|} \le \left\{\begin{array}{ll}
2^N A_1 E |u_S| + 2^P A_2 F |\gr u_S| & \textrm{on}\; K_1 \\  
A_1 E |u_S| + A_2 F |\gr u_S| & \textrm{on}\; K_2 \\ 
3^N A_1 E |u_S| + 3^P A_2 F |\gr u_S| & \textrm{on}\; K_3  \end{array} \right.,
\end{align*}
where $E = S^2 T^{-N}$, $F = ST^{-P}$.  Notice that we were able to use information about the decay rates of $V$ and $W$ at this point in the proof because we are away from zero. We now apply Corollary \ref{CT} to $f = \zeta u_S$, assuming that $\al > C_2\pr{1 + S\sqrt{\abs{\la}}}$, and use the above estimate on $K_1$.
\begin{align*}
\al^3\int w^{-2-2\al}\abs{f}^2 + \al\int w^{-2\al}|\gr f|^2 &\le C_3 \int_{K_1}  w^{-2\al}\abs{\LP u_S + \tilde\la u_S}^2 + C_3 \int_{K_2 \cup K_3}  w^{-2\al}\abs{\LP f + \tilde\la f}^2 \\
&\le 2^{1+2N}C_3  w^2\pr{1 + \frac{T}{2S}}A_1^2 E^2\int_{K_1}  w^{-2-2\al}\abs{u_S}^2 +2^{1+2P}C_3 A_2^2 F^2\int_{K_1}  w^{-2\al}|\gr u_S|^2 \\
&+ C_3 \int_{K_2 \cup K_3}  w^{-2\al}\abs{\LP f + \tilde\la f}^2.
\end{align*}
If we choose 
\begin{align*}
\al &= 4^{1+P+ {N}/{3}}C_3 w\pr{{5}/{4}}^{2/3}\pr{A_1^{2/3} + A_2^2} \cdot \left\{\begin{array}{ll} F^2 & \be > 1 + P - N + \frac{\om}{3} - \de \\  E^{2/3} & \be \le 1 + P - N + \frac{\om}{3} - \de \end{array}\right. \\
&= 4^{1+P+ {N}/{3}}C_3 w\pr{{5}/{4}}^{2/3}\pr{A_1^{2/3} + A_2^2} \cdot\max\set{F^2, E^{2/3}},
\end{align*}
then 
\begin{align*}
&2^{1+2N}C_3  w^2\pr{1 + \frac{T}{2S}}A_1^2 E^2 \le \frac{4^{1+N}}{2}C_3w^2\pr{\tfrac{5}{4}}A_1^2 E^2 \le \frac{\al^3}{2}, \\
&2^{1+2P}C_3  A_2^2 F^2 = \frac{4^{1+P}}{2}C_3 A_2^2 F^2 \le \frac{\al}{2},
\end{align*}
so we may absorb the first two terms on the right into the left to get
\begin{equation}
\frac{\al^3}{2}\int_{K_1} w^{-2-2\al}\abs{u_S}^2 + \frac{\al}{2}\int_{K_1} w^{-2\al}|\gr u_S|^2 \le J,
\label{est1}
\end{equation}
where $\disp J = C_3 \int_{K_2 \cup K_3}  w^{-2\al}\abs{\LP f + \tilde\la f}^2$. 
\\
Now we claim that if 
$T^{\de} > \max\set{2, \frac{C_2 \pr{1 + \sqrt{\abs{\la}}}}{2^{1 + 2P + 2N/3} C_3 w\pr{5/4}^{2/3} \pr{A_1^{2/3} + A_2^2} }},$
then $\al > C_2\pr{1 + S\sqrt{\abs{\la}}}$.  If $T^\de > 2$, then $T^{\ga - 1} > 2$, so $S = T^\ga - T > \frac{1}{2}T^\ga$.  If $\be \le 1 + P - N + \frac{\om}{3} - \de$, then
\begin{align*}
\al &= 4^{1+P+ {N}/{3}}C_3 w\pr{{5}/{4}}^{2/3}\pr{A_1^{2/3} + A_2^2} \cdot S^{4/3} T^{-2N/3} \\
&> \frac{4^{1+P+ {N}/{3}}C_3 w\pr{{5}/{4}}^{2/3}\pr{A_1^{2/3} + A_2^2}}{2^{1/3}} \pr{T^{\ga - 2N}}^{1/3} S \\
&>  2^{1+2P+ {2N}/{3}}C_3 w\pr{{5}/{4}}^{2/3}\pr{A_1^{2/3} + A_2^2} T^{\be - 1 + \de} S \\
&> C_2 S \pr{1 + \sqrt{\abs{\la}}},
\end{align*}
where we used the fact that $\be > 1$ and the assumption on the size of $T^\de$ to get to the last line.  Assuming that $S > 1$, we get the claimed result for this case.  The other cases follow in a similar way.  Thus, there is no problem with using the Carleman estimate for this choice of $\al$.   \\
Since $\LP f + \tilde\la f= \pr{\LP u_S + \tilde\la u_S}\zeta + 2\gr u_S \cdot \gr \zeta + u_S \, \LP \zeta$, then
\begin{align*}
|\LP f + \tilde \la f| &\lesssim \left\{\begin{array}{ll}  \pr{A_1 E + S^2} |u_S| + \pr{A_2 F + S}|\gr u_S| & \textrm{on} \; K_2 \\ \pr{A_1 E +  \pr{\frac{S}{T}}^2} |u_S| + \pr{A_2 F + {\frac{S}{T}}} |\gr u_S| & \textrm{on} \; K_3\end{array}\right..
\end{align*}
By Caccioppoli (Lemma \ref{Cacc} in Appendix \ref{AppA}), 
\begin{align*}
\int_{K_2} |\gr u_S|^2 &\lesssim \pr{S^2 + \abs{\la} + A_1 E + A_2^2 F^2}\int_{K_2^+}u_S^2 \\ 
\int_{K_3} |\gr u_S|^2 &\lesssim \pr{\pr{\frac{S}{T}}^2 + \abs{\la} + A_1 E + A_2^2 F^2} \int_{K_3^+}u_S^2,  
\end{align*}
where $\disp K_2^+ = \brac{\frac{1}{8S}, \frac{1}{S}}$ and $\disp K_3^+ = \brac{1 + \frac{T}{4S}, 1 + \frac{3T}{4S}}$.  Therefore,
\begin{align}
J &\le c_1C_3 w^{-2\al}\pr{\frac{1}{4S}} \pr{1 + \abs{\la} + A_1 + A_2^2}^2 S^4\int_{K_2^+}u_S^2 + c_2C_3 w^{-2\al}\pr{1 + \frac{T}{2S}} \pr{1 + \abs{\la} + A_1 + A_2^2}^2 S^4 C_0^2,
\label{JEst}
\end{align}
where the bound on the second term on the right follows from (\ref{uBd}). \\
Now we look for a lower bound for the term on the left hand side of (\ref{est1}).
\begin{align}
\frac{\al^3}{2}\int_{K_1} w^{-2-2\al}\abs{u_S}^2 + \frac{\al}{2}\int_{K_1} w^{-2\al}|\gr u_S|^2 &\ge \frac{\al^3}{2} w^{-2-2\al}\pr{1 + \frac{1}{S}}\int_{B_{1/S}\pr{\frac{x_0 - y_0}{S}}}\abs{u(y_0 + Sx)}^2dx 
\nonumber \\
&\ge  \frac{\al^3}{2}S^{-n} w^{-2-2\al}\pr{1 + \frac{1}{S}}\int_{B_{1}(x_0)}\abs{u}^2 
\nonumber \\
&\ge C_5\frac{\al^3}{2}S^{-n} w^{-2-2\al}\pr{1 + \frac{1}{S}}\exp(-C_4 T^{ \be} \log T),
\label{lowBd}
\end{align}
where the last line follows from the assumption (\ref{L2lBd2}). Combining (\ref{lowBd}), (\ref{est1}) and (\ref{JEst}) gives
\begin{align}
& C_5\frac{\al^3}{2}S^{-n} w^{-2-2\al}\pr{1 + \frac{1}{S}}\exp(-C_4 T^{ \be} \log T) \nonumber \\
\le &c_1C_3 w^{-2\al}\pr{\frac{1}{4S}} \pr{1 + \abs{\la} + A_1 + A_2^2}^2 S^4\int_{K_2^+}u_S^2 + c_2C_3 w^{-2\al}\pr{1 + \frac{T}{2S}} \pr{1 + \abs{\la} + A_1 + A_2^2}^2 S^4C_0^2.
\label{ineq}
\end{align}
If
\begin{equation*}
c_2C_3 w^{-2\al}\pr{1 + \frac{T}{2S}} \pr{1 + \abs{\la} + A_1 + A_2^2}^2 S^4C_0^2 \le  C_5\frac{\al^3}{4}S^{-n} w^{-2-2\al}\pr{1 + \frac{1}{S}}\exp(-C_4 T^{ \be} \log T),
\end{equation*}
or
\begin{equation}
\frac{c_2 C_0^2 \pr{1 + \abs{\la} + A_1 + A_2^2}^2 }{4^{2+3P+N}C_5 C_3^2}\brac{\frac{w\pr{1 + \frac{1}{S}}}{w\pr{\tfrac{5}{4}}}}^2\frac{S^{n+4}}{\max\{E^2, F^6\}} \exp\pr{C_4 T^{ \be}\log T} \le \brac{\frac{w\pr{1 + \frac{T}{2S}}}{w\pr{1 + \frac{1}{S}}}}^{2\al},
\label{abs2}
\end{equation}
then we may absorb the second term on the right of (\ref{ineq}) into the left hand side.  \\
Since we may choose $T_0 \ge T_*$, then by Lemma \ref{loglBd}, we see that for $c_n = c_n\pr{n, T_*}$,
\begin{equation}
\log\brac{\frac{ w\pr{1 + \frac{T}{2S}}}{ w\pr{1+ \frac{1}{S}}}} \ge c_{n}\frac{T}{S}.
\label{logEst}
\end{equation}
If we also assume that $ T \ge  \max\set{\frac{c_2 C_0^2 \pr{1 + \abs{\la} + A_1 + A_2^2}^2}{4^{2+3P+N}C_5 C_3^2}, \frac{8\pr{n{ \mathcal{G}} + 2N + 1}}{C_4}}$, where ${ \mathcal{G}}$ is an upper bound for $\ga$ that depends only on $N$ and $P$, then 
\begin{align*}
\log\pr{\frac{c_2 C_0^2 \pr{1 + \abs{\la} + A_1 + A_2^2}^2 }{4^{2+3P+N}C_5 C_3^2}\brac{\frac{w\pr{1 + \frac{1}{S}}}{w\pr{\tfrac{5}{4}}}}^2\frac{S^{n+4}}{\max\{E^2, F^6\}}} &\le \log\pr{T \cdot \frac{S^{n+4}}{S^4 T^{-2N}} } \le \pr{n{ \mathcal{G}} + 2N + 1}\log T \le \frac{C_4}{8} T \log T.
\end{align*}
If we take logarithms on both sides of (\ref{abs2}), and use (\ref{logEst}) and the fact that $\be > 1$, we see that (\ref{abs2}) holds if
\begin{align*}
& \tfrac{9}{8} C_4 T^\be \log T \le 2\al c_n\tfrac{T}{S} \\
\Iff\;& \tfrac{9 C_4}{4^{3+P+ {N}/{3}}C_3 w\pr{{5}/{4}}^{2/3}\pr{A_1^{2/3} + A_2^2} c_n} T^{\be - 1}\log T \le { \max\set{ST^{-2P}, S^{1/3}T^{-2N/3}}} \\
{\Leftarrow}\;& \pr{\tfrac{9 C_4}{4^{3+P+ {N}/{3}}C_3 w\pr{{5}/{4}}^{2/3}\pr{A_1^{2/3} + A_2^2} c_n} \log T}^a T^{\hg} \le S \\
\Leftarrow\;& 2\pr{\tfrac{9 C_4}{4^{3+P+ {N}/{3}}C_3 w\pr{{5}/{4}}^{2/3}\pr{A_1^{2/3} + A_2^2} c_n} \log T}^a T^{\hg} \le T^\prime,
\end{align*}
Since $ \pr{\frac{18 C_4}{4^{3+P+ {N}/{3}}C_3 w\pr{\tfrac{5}{4}}^{2/3}\pr{A_1^{2/3} + A_2^2} c_n} \log T}^a T^{\hg} = T^\ga$ and $T^\prime = T^\ga$ by definition, then condition (\ref{abs2}) is satisfied and (\ref{ineq}) reduces to
\begin{align*}
C_5\frac{\al^3}{4}S^{-n} w^{-2-2\al}\pr{1 + \frac{1}{S}}\exp(-C_4 T^{ \be} \log T) \le c_1C_3 w^{-2\al}\pr{\frac{1}{4S}} \pr{1 + \abs{\la} + A_1 + A_2^2}^2 S^4\int_{K_2^+}u_S^2
\end{align*}
By recalling that $u_S(x) = u(y_0 + Sx)$ and simplifying, we get
\begin{align}
\int_{B_{1}(y_0)}\abs{u}^2 &\ge C_5 \tfrac{4^{2+3P+N}C_3^2 \pr{A_1^{2/3} + A_2^2}^3}{c_1 \pr{1 + \abs{\la} + A_1 + A_2^2}^2} \frac{\max\{E^2, F^6\}}{S^4}\brac{\tfrac{ w\pr{ \tfrac{5}{4}}}{ w\pr{1 + \tfrac{1}{S}}}}^2 \brac{\tfrac{ w\pr{\tfrac{1}{4S}}}{ w\pr{1 + \tfrac{1}{S}}}}^{2\al}\exp\pr{-C_4 T^{ \be}\log T} 
\nonumber \\
&\ge{C}_5 \exp\left\{ -\brac{\log\pr{\tfrac{c_1 \pr{1 + \abs{\la} + A_1 + A_2^2}^2 }{4^{2+3P+N}C_3^2 \pr{A_1^{2/3} + A_2^2}^3}} +\pr{{4 - 3\be_c}} \log T  + (2\al -2)(1 + \tfrac{1}{4})\log S + C_4 T^{ \be}\log T}\right\} 
\label{abs} \\
&\ge C_5 \exp\left\{ -\brac{2\al(1 + \tfrac{1}{4})\log S + C_4\pr{1 + \tfrac{1}{4}} T^{ \be}\log T}\right\},
\label{ll}
\end{align}
where we assumed that $ T \ge \pr{\frac{c_1 \pr{1 + \abs{\la} + A_1 + A_2^2}^2 }{4^{2+3P+N}C_3^2 \pr{A_1^{2/3} + A_2^2}^3}}^{2/5}$ and $ T \ge { \frac{4\pr{4 - 3\be_c}}{C_4}}$ in order to absorb the first and second terms in (\ref{abs}), respectively.  { Furthermore, we used the fact that $S \ge 1$ and the definition of $\be_c$ to get the estimate for the second term in (\ref{abs})}.  We now want to show that we can absorb the second term in (\ref{ll}) into the first term.  Notice that $ C_4 T^{ \be} = \frac{C_4 \al}{4^{1+P+ {N}/{3}}C_3 w\pr{{5}/{4}}^{2/3}\pr{A_1^{2/3} + A_2^2} \cdot\max\set{F^2, E^{2/3}} T^{-\be}}$.  If $\max\{F^2, E^{2/3} \} = F^2$, then 
\begin{align*}
F^2T^{-\be} = \pr{\frac{S}{T^\prime}}^2 \pr{T^\prime}^{2}T^{-2P-\be} > \frac{1}{2}T^{2\ga - 2P - \be} = \frac{1}{2}T^{\be - \be_0 + 2\de} > \frac{1}{2}\pr{\tfrac{18 C_4}{4^{3+P+ {N}/{3}}C_3 w\pr{{5}/{4}}^{2/3}\pr{A_1^{2/3} + A_2^2} c_n} \log T}^2.
\end{align*}
Otherwise, $\max\{F^2, E^{2/3} \} = E^{2/3}$ and
\begin{align*}
E^{2/3}T^{-\be} = \pr{\frac{S}{T^\prime}}^{4/3}\pr{T^\prime}^{4/3}T^{-\frac{2N}{3}-\be}> \frac{1}{2}T^{(4\ga -2N - 3\be)/3} = \frac{1}{2}T^{3(\be - \be_0) + 4\de} > \frac{1}{2}\pr{\tfrac{18 C_4}{4^{3+P+ {N}/{3}}C_3 w\pr{{5}/{4}}^{2/3}\pr{A_1^{2/3} + A_2^2} c_n}\log T}^4.
\end{align*}
As long as $ \log T \ge \frac{8 c_n}{9 C_4}\max\left\{5 C_4 , 4^{1+P+ {N}/{3}}C_3 w\pr{{5}/{4}}^{2/3}\pr{A_1^{2/3} + A_2^2}  \right\}$, then $ C_4 \pr{1 + \tfrac{1}{4}} T^{ \be}\log T \le \frac{\al}{2} \log S$.  It follows that
$$ \int_{B_{1}(y_0)}\abs{u}^2 \ge C_5\exp\left\{ -{3 \cdot 4^{1+P+ {N}/{3}}C_3 w\pr{{5}/{4}}^{2/3}\pr{A_1^{2/3} + A_2^2} \cdot {T^\prime} ^{\be^\prime} \log T^\prime}\right\}.$$
If we set $\tilde C_4 = 3 \cdot 4^{1+P+ N/3} C_3 w\pr{{5}/{4}}^{2/3}\pr{A_1^{2/3} + A_2^2},$ then the proof is complete.
\end{proof}

%
%
\section{Proof of Theorem \ref{MEst}}
\label{MEstProof}

By inspecting the proof just given, we notice that $\be^\prime < \be$.  That is, each application of Proposition \ref{IH} reduces the exponent. This motivates us to define a decreasing sequence of exponents $\{ \be_j \}$, for each fixed $R \ge R_0> 0$, with the desired property that $\disp \lim_{j\to\iny}\be_j \approx \be_0$.  However, due to the presence of the $\de$ terms in the definitions of $\be$ and $\ga$, we do not get that $\disp \lim_{j\to\iny}\be_j = \be_0$.  It will be shown that we can get close enough under certain conditions.  Since $|y_0| > |x_0|$, we will also define an increasing sequence of positive numbers $\set{T_j}$ with the property that $T_{m+1} = R$, for some $m$ sufficiently large.  

The argument goes as follows: First, we will apply Proposition \ref{baseC} to get an initial estimate.  With the result as our hypothesis, we will then apply Proposition \ref{IH} to get another estimate.  With the newest estimate as our hypothesis, we apply Proposition \ref{IH} again.  This process will be repeated many times.  For
\begin{align*}
h_j &:= 1 + P - N + \om_j - \de_j, \\
\ell_j &:= 1 + P - N + \frac{\om_j}{3} - \de_j, 
\end{align*}
let
\begin{align*}
\be_1 &= \left\{\begin{array}{ll} 2 & W \not\equiv 0 \\ \frac{4}{3} & W \equiv 0 \end{array} \right., \\
\be_{j+1} &= \left\{\begin{array}{ll} 
2 - \frac{2P}{\ga_j} &  \be_j \ge h_j \\ 
\frac{4}{3} - \frac{2N}{3\ga_j} &  \be_j \le \ell_j 
\end{array}\right. \quad \textrm{for all}\; j \ge 1,
\end{align*}
where 
\begin{equation}
\ga_j = \left\{\begin{array}{ll} 
\be_j -1 + 2P + \de_j &  \be_j \ge h_j  \\ 
3(\be_j - 1) + 2N + 3\de_j &\be_j \le \ell_j
\end{array}\right..
\label{gamDef}
\end{equation}
At each stage, if $\disp \be_j \in \pr{\ell_j,  h_j}$, or if $\be_j < h_j$ but $\be_{j-1} > h_{j-1}$, then we replace $\be_j$ with $h_j$.  For $x_1$ to be specified, define
\begin{align*}
T_1 &= |x_1|, \\
T_{j+1} &= T_j^{\ga_j}  \quad \textrm{for all}\; j \ge 1.
\end{align*}
We choose each $\de_j$ as it is defined in Proposition \ref{IH}, where $C_4$ is replaced by $\tilde C_4$ after the first step.  Thus
\begin{align}
T_1^{\de_1} &=\frac{18 C_4}{4^{3 + P + N/3} C_3 w\pr{5/4}^{2/3} \pr{A_1^{2/3} + A_2^2}c_n } \log T_1, 
\label{T1ga1} \\
T_j^{\de_j} &={\frac{18 \tilde C_4}{4^{3 + P + N/3} C_3 w\pr{5/4}^{2/3} \pr{A_1^{2/3} + A_2^2} c_n} \log T_j} = \frac{27}{8 c_n} \log T_j \quad \textrm{for all}\; j \ge 2.
\label{Tjgaj}
\end{align}
Each $\om_j$ is chosen so that
\begin{equation}
T_j^{3P - N} + T_j = T_j^{3P - N + \om_j}.
\label{omDef}
\end{equation}
Finally, let
\begin{equation*}
\Ga_j = \ga_1\cdot \ga_2 \cdot \ldots \cdot \ga_j,
\end{equation*}
so that
$$T_{j+1} = T_1^{\Ga_j}.$$

We observe that $\be_{j+1} < \be_j$ and therefore $\ga_{j+1} < \ga_j$.  By studying the proof of Proposition \ref{IH}, we notice that since $T^\de > 2$, then $T_{j+1} > 2T_j$ and consequently, $\de_{j+1} < \de_j$.  Since $\disp \om_j = \frac{\log\pr{1 + T_j^{1 - 3P + N}}}{\log T_j}$, if $3P - N \ge 1$, then $\om_{j+1} < \om_j$ and $\om_j << \de_j$.

Depending on the values of $N$ and $P$, the sequences $\set{\ga_j}$ and $\set{\be_j}$ are defined in one of three ways: 
\begin{itemize}
\item Case 1: we always use the first choice ($\be_j \ge h_j$ for all $j \ge 1$), 
\item Case 2: we always use the second choice ($\be_j \le \ell_j$ for all $j\ge 1$), 
\item Case 3: we start out using the first choice then switch and use only the second choice ($\be_j \ge h_j$ for all $j \le J-1$, $\be_j \le \ell_j$ for all $j \ge J$).  \end{itemize}
Given the values of $N$ and $P$, it is possible to determine which case we fall into.  Our analysis is done separately in each of these three cases.

We will carefully choose $x_1$ and $m$ so that $T_{m+1} = R$.  In order to achieve the bound in Theorem \ref{MEst}(\ref{partA}), we need $m$ to be large enough so that 
\begin{equation}
\be_{m+1}- \be_0 \le (C_6-1) \frac{\log\log R}{\log R}.
\label{bem1}
\end{equation}
And to get the bound in Theorem \ref{MEst}(\ref{partB}), we need to show that
\begin{equation}
\pr{\be_{m+1}- 1}\log T_{m+1} \le \pr{C_6-1}\pr{\log\log R}^2.
\label{bem2}
\end{equation}

In order for Propositions \ref{baseC} and \ref{IH} to apply, we need to ensure that $T_1$ is sufficiently large.  Since $C_4$, $\tilde C_4$ and $C_5$ depend only on the constants associated with the PDE and the Carleman estimate, then for our argument, the $T_0$ that appears in Proposition \ref{IH} is a universal constant.  In a number of the following arguments, we require that each $T_j$ be sufficiently large (in a universal sense).  As we will see, when we combine all of these largeness assumptions, we will have an appropriate starting point for $T_1$.  From now on, we will write $T >> 1$ to mean that $T$ is bounded below by some implicit universal constant.

The cases of $\be_c > 1$ and $\be_c < 1$ are considered separately.

%
%
\subsection{Proof of Theorem \ref{MEst}(\ref{partA})}
\label{Proof1a}

The first goal of this subsection is to find an $m$ so that $R = T_{m+1}$ and so that estimate (\ref{bem1}) holds.  To do this, we start by defining a simpler (but very much related) recursive sequence of exponents, $\{ \hat\be_j \}$.  This is done (in most cases) by ignoring all $\de$  terms.  We will then compare the new sequence to our original one, and it will be clear how large $m$ must be.

Let $\lceil \ \rceil : \R \to \Z$ denote the ceiling function.  That is, for any $x \in \R$, $\lceil x \rceil = x + p \in \Z$, where $p \in [0, 1)$.

\begin{lem}
If we let
\begin{equation}
 m = \left\{\begin{array}{ll} 
\disp \left\lceil\tfrac{\log\log R}{\log(1/2P)} \right\rceil & \textrm{in Case 1: } 2 - 2P \ge \frac{4 - 2N}{3} \\  [\bigskipamount]
\disp \left\lceil\tfrac{\log\log R}{\log(1/2N)} \right\rceil & \textrm{in Case 2: }W \equiv 0 \\ [\bigskipamount]
\disp \left\lceil\tfrac{\log\log R}{\log(1/2N)} \right\rceil + J & \textrm{in Case 3: }W \not\equiv 0 \;\textrm{and}\; \frac{4-2N}{3} > 2 - 2P \end{array}  \right. ,
\label{mVals}
\end{equation}
where $J \le \mathcal{J}\pr{N,P, { T_1}}$ is determined by Lemma \ref{JLem}, then there exists $C_6\pr{N, P}$ such that $\disp \be_{m+1}- \be_0 \le (C_6-1) \frac{\log\log R}{\log R}$.
\label{mLem}
\end{lem}

Lemma \ref{mLem} will be proved after a few results have been established for each of the three cases. \\

%
%
\nid \textbf{Case 1:} $\disp 2-2P \ge \frac{4-2N}{3}$.

By Lemma \ref{case1}, if $T_1 >> 1$, then each $T_j >> 1$ and $\be_j \ge h_j$ for all $j \ge 1$.  Thus, we are in Case 1.  By setting all of the $\de_j$ terms equal to zero, we define
\begin{align}
\hat{ \be}_1 &= 2, 
\nonumber \\
\hat{\be}_{j+1} &= 2 - \frac{2P}{\hg_j} \quad \textrm{for all} j \ge 1,
\label{PbeHatDef}
\end{align} 
where 
$$\hat{\ga}_j = \hat{\be}_j - 1 + 2P.$$  
If we set $\hat\be_{j+1} = \hat\be_j$, we get two solutions: $2-2P$ and $1$.  Since $\be_0 = 2 - 2P > 1$ then $\disp \lim_{j\to\iny}\hat\be_j  = \be_0$.  By Lemma \ref{beHatEstP},
\begin{align}
\be_{j+1} &\le \hat\be_{j+1} + \frac{(2P)^j}{(\hg_1\ldots \hg_j)^2}\de_1 + \frac{(2P)^{j-1}}{(\hg_2\ldots \hg_{j})^2}\de_2  + \ldots + \frac{2P}{(\hg_j)^2}\de_j ,
\label{betaEstC1}
\end{align}
for all $j\ge 1$.  We will use (\ref{betaEstC1}) to show that the difference between $\be_j$ and $\hat\be_j$ is on the order of $\de_j $. But first we require a couple of facts.
 
 Since 
 $$\de_j \log T_j = \left\{ \begin{array}{ll} 
\log\pr{\frac{18 C_4 \log T_j}{4^{3 + P + N/3} C_3 w\pr{5/4}^{2/3} \pr{A_1^{2/3} + A_2^2} c_n  }} & j = 1 \\ 
\log\pr{\frac{27 \log T_j}{8 c_n}} & j \ge 2 
\end{array}\right. ,$$
then we get the following.
 
 \begin{lem}
 For $T_1 >> 1$,
 \begin{equation}
 \frac{\sqrt{3}}{2}\frac{\log\log T_j}{\log T_j} \le \de_j \le \frac{2}{\sqrt{3}} \frac{\log\log T_j}{\log T_j}.
 \label{epsEst}
 \end{equation}
 \end{lem}
 
 Let $C_8 = \frac{1 + 2P}{4P}$.  Since $2 - 2P > 1$, then $2P < 1$, $C_8 > 1$ and $2P C_8 < 1$, so the corresponding geometric series is summable.  The following lemma is used as an intermediate step.
 
 \begin{lem} For $T_1 >> 1$, $\disp \frac{\ga_j}{\hg_j^2} \le C_8$.
 \label{C7Lem}
 \end{lem}
 
 \begin{proof}
 We will use induction to show that this holds.
\quad \\
Base case: $\disp \frac{\ga_1}{\hg_1^2} \le \frac{1 + 2P + \de_1}{(1 + 2P)^2} = \frac{1 + 2P + \de_1}{1 + 4P + 4P^2} \le 1 < C_8$ for $\de_1\le 2P$, which holds if $T_1 >> 1$. \\
Assume that $\disp \frac{\ga_j}{\hg_j^2} \le C_8$ for all $j \le k$. \\
By (\ref{betaEstC1}) and the definition of $\ga_{k+1}$ given by (\ref{gamDef}),
\begin{align*}
\ga_{k+1} &\le \hg_{k+1} + \frac{(2P)^k}{(\hg_1\ldots \hg_k)^2}\de_1 + \frac{(2P)^{k-1}}{(\hg_2\ldots \hg_{k})^2}\de_2  + \ldots + \frac{2P}{(\hg_k)^2}\de_k + \de_{k+1}.
\end{align*}
 For $j \le k$, $T_{k+1} = T_j^{\ga_j \ga_{j+1}\ldots \ga_{k}} \Rightarrow \frac{\log\log T_j}{\log T_j} = \ga_j \ga_{j+1}\ldots \ga_{k}\frac{\log\log T_{k+1} - \log(\ga_j \ldots \ga_{k})}{\log T_{k+1} } \le \ga_j \ga_{j+1}\ldots \ga_{k}\frac{\log\log T_{k+1}}{\log T_{k+1}}$, since $\ga_\ell > 1$ for all $\ell$.   So by (\ref{epsEst}), 
 \begin{equation}
 \de_j \le \frac{4}{3} \ga_j\ldots\ga_k \;\de_{k+1}
 \label{epsjk}
 \end{equation}
Thus,
\begin{align*}
\ga_{k+1} &\le \hg_{k+1} + \frac{4}{3}\brac{(2P)^k\frac{\ga_1\ldots \ga_k}{(\hg_1\ldots \hg_k)^2} + (2P)^{k-1}\frac{\ga_2\ldots \ga_{k}}{(\hg_2\ldots \hg_{k})^2}  + \ldots + 2P\frac{\ga_k}{(\hg_k)^2} + 1}\de_{k+1} - \frac{1}{3}\de_{k+1} \\
&\le \hg_{k+1} + \frac{4}{3}\brac{(2P C_8)^k + (2P C_8)^{k-1}  + \ldots + 2P C_8 + 1 }\de_{k+1} - \frac{1}{3}\de_{k+1} \\
&< \hg_{k+1} + \frac{7 + 2P}{3(1 - 2P)}\de_{k+1} \\
&\le  \hg_{k+1} + C_8 - 1
\end{align*}
 for $T_1 >> 1$. Then
 $$\frac{\ga_{k+1}}{\hg_{k+1}^2} \le \frac{1}{\hg_{k+1}} + \frac{C_8 - 1}{\hg_{k+1}^2} \le C_8,$$
 since $\hg_j \ge 1$ for all $j$, completing the proof.
 \end{proof}
 
 Notice that
 \begin{align*}
 \be_j &= \ga_j + 1 - 2P - \de_j \\
 &\le \hg_j + \frac{7 + 2P}{3(1 - 2P)}\de_j + 1 - 2P - \de_j \quad \textrm{by (\ref{gamEstC1})}\\
 &= \hat\be_j + \frac{4 + 8P}{3(1-2P)}\de_j.
 \end{align*}
This and the proof of Lemma \ref{C7Lem} give the desired corollary.
 
\begin{cor}
If $T_1 >> 1$, then for all $j \ge 1$,
\begin{align}
\ga_{j} &\le \hg_{j} + \frac{7 + 2P}{3(1 - 2P)}\de_j
\label{gamEstC1} \\
\be_j &\le \hat\be_j + \frac{4 + 8P}{3(1-2P)}\de_j.
\label{beEstC1}
\end{align}
\label{gambeCorP}
\end{cor}

%
%
\nid \textbf{Case 2:} $W \equiv 0$. \\
Since $W \equiv 0$, then we always define our sequences $\set{\be_j}$ and $\set{\ga_j}$ by using the second choice, putting us in Case 2.  
We define
\begin{align*}
\hat{ \be}_1 &= \frac{4}{3}, \\
\hat{\be}_{j+1} &= \frac{4}{3} - \frac{2N}{3\hg_j} \quad \textrm{for all} j \ge 1,
\end{align*} 
where 
$$\hat{\ga}_j = 3(\hat{\be}_j - 1) + 2N.$$  
Setting $\hat\be_{j+1} = \hat\be_j$, we get two solutions: $\disp \frac{4-2N}{3}$ and $1$.  Since $\disp \be_0 = \frac{4-2N}{3} > 1$ then $\disp \lim_{j\to\iny}\hat\be_j  = \be_0$.  As was done in Case 1, we estimate $\be_j$ in terms of $\hat\be_j$ (assuming that $T_1 >> 1$ so that $\de_1 \le 1/3$).
\begin{align}
{ \be}_1 &= \hat{\be}_1, 
\nonumber \\
\Rightarrow {\be}_2 &= \frac{4}{3} - \frac{2N}{3(\hat{\ga}_1 + 3\de_1)} 
\le \frac{4}{3} - \frac{2N}{3\hat{\ga}_1} + \frac{2N}{(\hg_1)^2}  \de_1
= \hat \be_2 + \frac{2N}{(\hg_1)^2}\de_1, 
\nonumber \\
 &\vdots 
 \nonumber \\
\Rightarrow \be_{j+1} &\le \hat\be_{j+1} + \frac{(2N)^j}{(\hg_1\ldots \hg_j)^2}\de_1 + \frac{(2N)^{j-1}}{(\hg_2\ldots \hg_{j})^2}\de_2  + \ldots + \frac{2N}{(\hg_j)^2}\de_j 
\label{betaEstC2}
\end{align}
 We will use (\ref{betaEstC2}) to show that the difference between $\be_j$ and $\hat\be_j$ is on the order of $\de_j $. 
  
 Let $C_9 = \frac{1 + 2N}{4N}$.  Since $\disp \be_0 = \frac{4-2N}{3}> 1$, then $2N < 1$, $C_9 > 1$ and $2N C_9 < 1$, so the corresponding geometric series is summable.  The following lemma is used as an intermediate step.
 
 \begin{lem} For $T_1 >> 1$, $\disp \frac{\ga_j}{\hg_j^2} \le C_9$.
 \label{C8Lem}
 \end{lem}
 
 \begin{proof}
 We will use induction to show that this holds.
\quad \\
Base case: $\disp \frac{\ga_1}{\hg_1^2} \le \frac{1 + 2N + 3\de_1}{(1 + 2N)^2} = \frac{1 + 2N + 3\de_1}{1 + 4N + 4N^2} \le 1 < C_9$ for $3\de_1\le 2N$, which holds if $T_1 >> 1$. \\
Assume that $\disp \frac{\ga_j}{\hg_j^2} \le C_9$ for all $j \le k$. \\
By (\ref{betaEstC2}) and the definition of $\ga_{k+1}$,
\begin{align*}
\ga_{k+1} &\le \hg_{k+1} + \frac{(2N)^k}{(\hg_1\ldots \hg_k)^2}3\de_1 + \frac{(2N)^{k-1}}{(\hg_2\ldots \hg_{k})^2}3\de_2  + \ldots + \frac{2N}{(\hg_k)^2}3\de_k + 3\de_{k+1}.
\end{align*}
Using (\ref{epsjk}), we get
\begin{align*}
\ga_{k+1} &\le \hg_{k+1} + 4\brac{(2N)^k\frac{\ga_1\ldots \ga_k}{(\hg_1\ldots \hg_k)^2} + (2N)^{k-1}\frac{\ga_2\ldots \ga_{k}}{(\hg_2\ldots \hg_{k})^2}  + \ldots + 2N\frac{\ga_k}{(\hg_k)^2} + 1}\de_{k+1} - \de_{k+1} \\
&\le \hg_{k+1} + 4\brac{(2N C_9)^k + (2N C_9)^{k-1}  + \ldots + 2N C_9 + 1 }\de_{k+1} - \de_{k+1} \\
&< \hg_{k+1} + \frac{7 + 2N}{1 - 2N}\de_{k+1} \\
&\le  \hg_{k+1} + C_9 - 1
\end{align*}
 for $T_1 >> 1$. Then
 $$\frac{\ga_{k+1}}{\hg_{k+1}^2} \le \frac{1}{\hg_{k+1}} + \frac{C_9 - 1}{\hg_{k+1}^2} \le C_9,$$
 since $\hg_j \ge 1$ for all $j$, completing the proof.
 \end{proof}
 
 Notice that
 \begin{align*}
 \be_j &= \frac{\ga_j}{3} + 1 - \frac{2N}{3} - \de_j \\
 &\le \frac{\hg_j}{3} + \frac{7 + 2N}{3(1 - 2N)}\de_j + 1 - \frac{2N}{3} - \de_j \quad \textrm{by (\ref{gamEstC2})}\\
 &= \hat\be_j + \frac{4 + 8N}{3(1-2N)}\de_j.
 \end{align*}
This gives the desired corollary.
 
\begin{cor}
If $T_1 >> 1$, then for all $j \ge 1$,
\begin{align}
\ga_j &\le \hg_j + \frac{7 + 2N}{1 - 2N}\de_j
\label{gamEstC2} \\
\be_j &\le \hat\be_j + \frac{4 + 8N}{3(1-2N)}\de_j.
\label{beEstC2}
\end{align}
\label{gambeCorN}
\end{cor}

%
%
\nid \textbf{Case 3:} $W \not\equiv 0$ and $\frac{4-2N}{3} > 2 - 2P$. \\
We will now consider the case where $W \not\equiv 0$, so $\be_1 = 2$, but $\disp \be_0 = \frac{4-2N}{3}$.  We have $3P - N = 1 + \De$, for some $\De > 0$.  Notice that $P - N = 1 + \De - 2P$, so if $P < 1/2$, then $P - N > 0$.  And if $P \ge 1/2$, since $N < 1/2$, we again get that $P - N > 0$.  \\
By Lemma \ref{JLem}, there exists a $J $ such that $\be_{J-1} = h_{J-1}$.  Furthermore, $\be_k \ge h_k$ for all $k \le J-1$.

\begin{lem} 
For $T_1 >> 1$, if $\be_{j-1} \le h_{j-1}$ then $\be_j \le \ell_j$.
\label{decLem}
\end{lem}

{
\begin{proof}
If $\be_{j-1} \le \ell_{j-1}$, then $\ga_{j-1} \le 3P - N + \om_{j-1}$ and $\be_j \le \frac{4}{3} - \frac{2N}{3(3P - N + \om_{j-1})}$.  
\begin{align*}
& \frac{4}{3} - \frac{2N}{3(3P - N + \om_{j-1})} \le \ell_j \\
\Iff\;& \frac{\De + \om_{j-1}}{\De + \om_j - 3\de_j} \le \frac{3(P-N) + 2N + \om_{j-1}}{2N} \\
\Leftarrow\;& \frac{\De + \de_{j}}{\De - 3\de_j} \le 1 + \frac{3(P-N)}{2N} \\
\Iff\;& \frac{4\de_{j}}{\De - 3\de_j} \le  \frac{3(P-N)}{2N},
\end{align*}
for $T_j >> 1$.  If we ensure that $\de_1 \le \frac{3\De(P-N)}{8N + 9(P-N)}$, since $\disp \de_j \le \de_1$, then the claim follows. \\
Now suppose $\be_{j-1} \le h_{j-1}$ but $\be_{j-1} > \ell_{j-1}$.  Then $\be_{j-1} = h_{j-1}$ and by Lemma \ref{utol}, we get the result.
\end{proof}
}



\begin{cor}
For $T_1 >> 1$, $\be_j \le \ell_j$ for all $j \ge J$.
\label{decLemCor}
\end{cor}

\begin{proof}
By Lemma \ref{JLem}, there is a $J$ such that $\be_{J-1} \le h_{J-1}$.  Lemma \ref{decLem} implies that $\be_{J} \le \ell_{J}$.  Since $\ell_j < h_j$ for all $j$, Lemma {\ref{decLem}} shows that $\be_j \le \ell_j$ for all $j > J$ as well.
\end{proof}

These results show that we initially use the first choice to define our sequences, but then we switch and use only the second choice.  So, indeed, we are in Case 3.

For Case 3, our simpler sequence of exponents will begin at $J + 1$ instead of $1$.  Let
\begin{align*}
\hat\be_{J+1} &= \frac{4}{3}, \\
\hat\be_{J+j +1} &= \frac{4}{3} - \frac{2N}{3\hg_{J+j}}, \quad \textrm{for all }\; j\ge 1,
\end{align*}
where
$$\hg_{J+j} = 3(\hat\be_{J+j} - 1) + 2N.$$

As in Case 2, if we set $\hat\be_{J+j+1} = \hat\be_{J+j}$, we get two solutions: $\disp \frac{4-2N}{3}$ and $1$.  Since $\disp \be_0 = \frac{4-2N}{3} > 1$ then $\disp \lim_{j\to\iny}\hat\be_{J+j}  = \be_0$.  We estimate $\be_j$ in terms of $\hat\be_j$ (assuming that $\de_{J+1} \le 1/3$) for all $j > J$ using the same idea that was used in Case 2.  But first observe that since $\be_J \le \ell_J$, then $\ga_J \le 1 + \De + \om_J$.
\begin{align}
{ \be}_{J+1} &= \frac{4}{3} - \frac{2N}{3(1 + \De + \om_J)} \le \frac{4}{3} = \hat{\be}_{J+1}, 
\nonumber \\
{\be}_{J+2} &\le \frac{4}{3} - \frac{2N}{3(\hat{\ga}_{J+1} + 3\de_{J+1})} 
\le \frac{4}{3} - \frac{2N}{3\hat{\ga}_{J+1}} + \frac{2N}{(\hg_{J+1})^2}  \de_{J+1}
= \hat \be_{J+2} + \frac{2N}{(\hg_{J+1})^2}\de_{J+1}, 
\nonumber \\
 &\vdots 
 \nonumber \\
\be_{J+j+1} &\le \hat\be_{J+j+1} + \frac{(2N)^j}{(\hg_{J+1}\ldots \hg_{J+j})^2}\de_{J+1} + \frac{(2N)^{j-1}}{(\hg_{J+2}\ldots \hg_{J+j})^2}\de_{J+2}  + \ldots + \frac{2N}{(\hg_{J+j})^2}\de_{J+j} 
\label{betaEstC3}
\end{align}
 We will use (\ref{betaEstC3}) to show that the difference between $\be_j$ and $\hat\be_j$ is on the order of $\de_j $ for all $J > j$.  Using the same proof as that of Lemma \ref{C8Lem}, we get the following lemma.

 \begin{lem}
 For $T_1 >> 1$, $\frac{\ga_j}{{\hg_j}^2} \le C_9$ for all $j > J$.
 \label{C3gamLem}
 \end{lem}
 
 \begin{cor}
 If $T_1 >> 1$, then for all $j > J$, 
 \begin{equation}
 \ga_j \le \hg_j +  \frac{7 + 2N}{1-2N}\de_j,
 \label{C3gaEst}
 \end{equation}
 \begin{equation}
 \be_j \le \hat\be_j +  \frac{4 + 8N}{3(1-2N)}\de_j.
 \label{C3begaEst}
 \end{equation}
 \label{C3begaEst}
 \end{cor}
 
The proof of Corollary {\ref{C3begaEst}} is analogous to that of Corollary \ref{gambeCorN}.
\\

We now have enough information to prove Lemma \ref{mLem}.

%
%
\begin{proof}[Proof of Lemma \ref{mLem}]
If $2 - 2P \ge \frac{4 - 2N}{3}$, then a computation shows that for all $j \ge 1$,
\begin{equation}
\hg_j = \frac{1 + 2P + \ldots + (2P)^j}{1 + 2P + \ldots + (2P)^{j-1}} = 1 + \frac{(2P)^j}{1 + 2P + \ldots + (2P)^{j-1}} \le 1 + (2P)^j
\label{hatGamP}
\end{equation}
Thus, $\hg_{m} \le 1 + \pr{2P}^{\left\lceil\frac{\log\log R}{\log(1/2P)} \right\rceil } = 1 + \frac{C}{\log R} \le 1 +  \frac{\log\log R}{ \log R}$ so, by (\ref{gamEstC1}) in Corollary \ref{gambeCorP}, $\ga_{m} \le 1 + C_P \frac{\log\log R}{\log R}$. Since 
$$\be_{m+1} = 2 - \frac{2P}{\ga_m} \le 2 - \frac{2P}{1 + C_P \frac{\log\log R}{\log R}} \le 2 - 2P + 2PC_P \frac{\log\log R}{\log R} = \be_0 + \tilde C_P \frac{\log\log R}{\log R},$$
then we get the desired inequality. \\
If $W \equiv 0$, then a computation shows that for all $j \ge 1$,
\begin{equation}
\hg_j = \frac{1 + 2N + \ldots + (2N)^j}{1 + 2N + \ldots + (2N)^{j-1}} = 1 + \frac{(2N)^j}{1 + 2N + \ldots + (2N)^{j-1}} \le 1 + (2N)^j
\label{hatGamN}
\end{equation}
Then $\hg_{m} \le 1 + \pr{2N}^{\left\lceil\frac{\log\log R}{\log(1/2N)} \right\rceil} = 1 + \frac{C}{\log R} \le 1 +  \frac{\log\log R}{ \log R}$ so, by (\ref{gamEstC2}) in Corollary \ref{gambeCorN}, $\ga_{m} \le 1 + C_N \frac{\log\log R}{\log R}$.\\
If $W \not\equiv 0$ and $\disp \frac{4-2N}{3} > 2 - 2P$, then by comparison with Case 2, \begin{equation}
\hg_{J + j} \le 1 + (2N)^{j}
\label{hatGamNJ}
\end{equation}
Thus, $\hg_{m} \le 1 + \pr{2N}^{\left\lceil\frac{\log\log R}{\log(1/2N)} \right\rceil} = 1 + \frac{C}{\log R} \le 1 +  \frac{\log\log R}{ \log R}$ so, by (\ref{C3gaEst}), $\ga_{m} \le 1 + C_N \frac{\log\log R}{\log R}$.
In the last two cases, since
$$\be_{m+1} = \frac{4}{3} - \frac{2N}{3\ga_m} \le  \frac{4}{3} - \frac{2N}{3\pr{1 + C_N \frac{\log\log R}{\log R}}} \le \frac{4-2N}{3} + 2NC_N\frac{\log\log R}{\log R} = \be_0 + \tilde C_N \frac{\log\log R}{\log R},$$
then we get the desired inequality.
\end{proof}

In order to apply Proposition \ref{IH} at each step of the iteration, we must show that each $T_j >> 1$.  Since $T_j \ge T_1$ for all $j$, then it suffices to show that $T_1 >> 1$.  Thus, the second goal of this section is to show that for each fixed $N$ and $P$, there exists an $R_0 > 0$ so that for every $R \ge R_0$, the corresponding $T_1$ is sufficiently large.  That is, if $T_1$ satisfies $R = T_1^{\ga_1\ga_2\ldots\ga_m}$, where each $\ga_j$ depends on $T_k$ for all $k \le j$ and $m$ is a function of $R$ as given in Lemma \ref{mLem}, then we can guarantee that $T_1 >> 1$.

First we will show that if we increase $T_1$, then for each $k$, $T_k$ also increases.  Recall that $T_{k+1} = T_k^{\ga_k}$ for every $k \ge 1$, so $T_{k+1} = T_1^{\ga_1\ldots\ga_k}$.  Also, $\ga_j = \ga_j(T_j)$ for each $j$.  Thus, $T_2 = T_1^{\ga_1}$ is function of $T_1$.  It follows that $\ga_2$ is a function of $T_1$.  If we continue on, we see that for any $j$, $T_j = T_j(T_1)$ and $\ga_j = \ga_j(T_1)$.

\begin{lem}
If $T_1 >> 1$, then for each $k \ge 2$, $T_k$ is an increasing function of $T_1$.
\label{Tinc}
\end{lem}

As explained above, $T_k =T_1^{\ga_1\ldots\ga_{k-1}}$, where each $\ga_j$ is a function of $T_1$.  Since $ \frac{d\de_j}{dT_1} < 0$, then$ \frac{d\ga_j}{dT_1} < 0$ and the exponent $\ga_1\ldots\ga_{k-1}$ decreases with respect to $T_1$.  However, as the proof will show, the exponent does not decrease fast enough to ``beat'' the base of $T_1$.

\begin{proof}
Instead of showing the claimed fact, we will prove that $ \frac{T_{k}}{T_{k-1}}$ is an increasing function of $T_1$ for each $k$.  Since $ T_k = \frac{T_k}{T_{k-1}}\ldots\frac{T_2}{T_1} T_1$, and the product of positive increasing functions is increasing, then this fact is sufficient.  Induction will be used. \\
Base case:  Since
$$\frac{T_2}{T_1} = T_1^{\ga_1 - 1} = \left\{\begin{array}{ll} 
T_1^{2P}\brac{\frac{18 C_4}{4^{3 + P + N/3} C_3 w\pr{5/4}^{2/3} \pr{A_1^{2/3} + A_2^2}c_n } \log T_1} & W \not\equiv 0, 2 > \ell_1 \\ 
T_1^{2 + 2N}\brac{\frac{18 C_4}{4^{3 + P + N/3} C_3 w\pr{5/4}^{2/3} \pr{A_1^{2/3} + A_2^2}c_n } \log T_1}^3 & W \not\equiv 0, 2 \le \ell_1 \\ 
T_1^{2N}\brac{\frac{18 C_4}{4^{3 + P + N/3} C_3 w\pr{5/4}^{2/3} \pr{A_1^{2/3} + A_2^2}c_n } \log T_1}^3 & W \equiv 0 \end{array}\right. ,$$
then it is clear that the statement holds for $k = 2$. \\
Assume that $ \frac{T_{j}}{T_{j-1}}$ increases with $T_1$ for all $j \le k$. \\
$\disp \frac{T_{k+1}}{T_k} = T_k^{\ga_k - 1}$, where
$$\ga_k = \left\{\begin{array}{ll} \be_k - 1 + 2P + \de_k & \be_k \ge h_k \\ 3P - N + \om_k & \be_k \in \pr{\ell_k, h_k} \\ 3(\be_k - 1) + 2N + 3\de_k & \be_k \le \ell_k  \end{array}\right..$$
If $\be_k \ge h_k$, then $\be_{k-1} \ge h_{k-1}$ so $\be_k = 2 - \frac{2P}{\ga_{k-1}}$ and
$$T_k^{\ga_k - 1} = T_k^{2 - \frac{2P}{\ga_{k-1}} - 1 + 2P + \de_k - 1} = \pr{\frac{T_k}{T_k^{1/\ga_{k-1}}}}^{2P}\pr{\frac{27}{8 c_n} \log T_k} = \pr{\frac{T_k}{T_{k-1}}}^{2P}\pr{\frac{27}{8 c_n} \log T_k},$$ 
which increases by the inductive hypothesis. \\
If $\ga_k = 3P - N + \om_k$, then by the definition of $\om_k$ as in (\ref{omDef}),
$T_{k+1} = T_k^{\ga_k} = T_k^{3P - N + \om_k} = T_k^{3P - N} + T_k.$
Since $T_{k+1} > 2T_k$, then $3P - N > 1$ in this case.  And
$$T_k^{\ga_k - 1} = \frac{T_k^{\ga_k}}{T_k} = \frac{T_k^{3P - N} + T_k}{T_k} = T_k^{3P - N -1} + 1,$$
which increases with $T_k$ (and hence with $T_1$ by the inductive hypothesis) since $3P - N - 1 > 0$ is constant. \\
If $\be_k \le \ell_k$, then, by our construction, either $\be_{k-1} \le \ell_{k-1}$ or $\be_{k-1} = h_{k-1}$.  If $\be_{k-1} \le \ell_{k-1}$, then $\disp \be_k = \frac{4}{3} - \frac{2N}{3\ga_{k-1}}$ and
$$T_k^{\ga_k - 1} = T_k^{3\pr{\frac{4}{3} - \frac{2N}{3\ga_{k-1}} - 1} + 2N + 3\de_k - 1} = \pr{\frac{T_k}{T_k^{1/\ga_{k-1}}}}^{2N}\pr{\frac{27}{8 c_n} \log T_k}^3 = \pr{\frac{T_k}{T_{k-1}}}^{2N}\pr{\frac{27}{8 c_n} \log T_k}^3,$$
which increases by the inductive hypothesis. Now we consider the case where $\be_k \le \ell_k$ and $\be_{k-1} = h_{k-1}$.  That is, there is a switch in the way the sequences are defined so we are in Case 3.  Thus, $3P - N = 1+ \De$ for some $\De > 0$.  Since $\be_{k-1} = h_{k-1}$, then $\ga_{k-1} = 3P - N + \om_{k-1}$ and $\disp \be_k = 2 - \frac{2P}{\ga_{k-1}}$.  Thus,
\begin{align*}
\ga_k - 1
&= \frac{2}{\ga_{k-1}}\brac{N(3P - N + \om_{k-1} - 1) + \om_{k-1}} + 3\de_k,
\end{align*}
so that
\begin{align*}
T_k^{\ga_k -1} &= T_k^{\frac{2}{\ga_{k-1}}\brac{N(3P - N + \om_{k-1} - 1) + \om_{k-1}} + 3\de_k} \\
&= T_{k-1}^{2\brac{N(3P - N + \om_{k-1} - 1) + \om_{k-1}}}T_{k}^{ 3\de_k} \\
&= \pr{\pr{ \frac{T_{k-1}^{3P - N + \om_{k-1}}}{T_{k-1}}}^N T_{k-1}^{\om_{k-1}}}^2 \pr{\frac{27}{8 c_n} \log T_k}^3 \\
&= \pr{T_{k-1}^{\De} + 1 }^{2N} \pr{1 + T_{k-1}^{-\De}}^2 \pr{\frac{27}{8 c_n} \log T_k}^3,
\end{align*}
where the third line follows from the fact that $T_{k-1}^{3P - N + \om_{k-1}} = T_{k-1}^{3P-N} + T_{k-1} \Rightarrow T_{k-1}^{\om_{k-1}} = 1 + T_{k-1}^{1 - (3P - N)}.$
Let $f(x) = (1+x)^{N}(1 + x^{-1})$.  Then $\disp f^\prime(x) = \frac{(1+x)^{N-1}}{x^2}\brac{Nx^2 + (N-1)x - 1}$ and $f^\prime(x) \ge 0$ if $x \ge \frac{1}{N}$ or $x \le -1$.  If we choose $T_1 >> 1$ so that $T_1^\De \ge \frac{1}{N}$, then $T_{k-1}^\De \ge \frac{1}{N}$ also and $\pr{T_{k-1}^{\De} + 1 }^{N} \pr{1 + T_{k-1}^{-\De}}$ will increase with $T_{k-1}$, since $\De > 0$ is fixed.  It follows that $T_k^{\ga_k - 1}$ increases with $T_1$.
\end{proof}

\begin{lem}
Let $ c_P = \frac{7 + 2P}{3(1-2P)}$, $ c_N = \frac{7 + 2N}{3(1-2N)}$.  Recall that $J \le \mathcal{J}$ is given by Lemma \ref{JLem}.
Using the choices for $m$ given by Lemma \ref{mLem}, there exists $\mathcal{T}_0\pr{n, N, P, \la, A_1, A_2, C_0}$ so that if we let
\begin{equation}
 \log\log R_0 = \left\{\begin{array}{ll} 
\disp 2\brac{\log\pr{1 + c_P\de_1} + \log\pr{\tfrac{1}{1-2P}} + \log\log \mathcal{T}_0} & \textrm{in Case 1: } 2 - 2P \ge \frac{4 - 2N}{3} \\  [\bigskipamount]
\disp 2 \brac{\log\pr{1 + c_N\de_1} + \log\pr{\tfrac{1}{1-2N}} + \log\log \mathcal{T}_0} & \textrm{in Case 2: }W \equiv 0 \\ [\bigskipamount]
\disp 2 \brac{\mathcal{J}\log\ga_1 + \log\pr{1 + c_N\de_{J+1}} + \log\pr{\tfrac{1}{1-2N}} + \log\log \mathcal{T}_0} & \textrm{in Case 3: } W \not\equiv 0, \frac{4-2N}{3} > 2 - 2P \end{array}  \right. ,
\label{RNPVals}
\end{equation}
where $\de_1 = \de_1\pr{\mathcal{T}_0}$, $\de_{J+1}$ correspond to $\mathcal{T}_0^{\Ga_J}$,
then for every $R \ge R_0$, the corresponding $T_1$ is large enough so that all of the above lemmas and inequalities hold.
\label{RNPLem}
\end{lem}

\begin{proof}[Proof of Lemma \ref{RNPLem}]
Fix $\mathcal{T}_0 >> 1$ so that all of the relevant assertions and lemmas above hold and so that 
\begin{equation}
\frac{1}{2} \ge \left\{\begin{array}{ll} 
\disp \tfrac{\log\pr{1 + c_ P\de_{1}}}{\log(1/2P)} & \textrm{in Case 1: } 2 - 2P \ge \frac{4 - 2N}{3} \\  [\bigskipamount]
\disp \tfrac{\log\pr{1 + c_ N\de_{1}}}{\log(1/2N)}  & \textrm{in Case 2: }W \equiv 0 \\ [\bigskipamount]
\disp \tfrac{\log\pr{1 + c_ N\de_{J+1}}}{\log(1/2N)} & \textrm{in Case 3: } W \not\equiv 0, \frac{4-2N}{3} > 2 - 2P \end{array}  \right. ,
\label{conBdC1}
\end{equation}
Fix $R \ge R_0$, where $R_0$ is given in the statement.  Let $m = m\pr{R}$ be given by Lemma \ref{mLem}. Note that $m$ is also fixed.  \\
If $2 - 2P \ge \frac{4 - 2N}{3}$, an extension of (\ref{gamEstC1}) gives
\begin{align*}
\ga_j &< \hg_j + c_P \de_j < \pr{1 + c_P \de_j} \frac{1 + 2P + \ldots + (2P)^j}{1 + 2P + \ldots + (2P)^{j-1}} \quad \textrm{by (\ref{hatGamP})}.
\end{align*}
Thus,
\begin{align}
\log\Ga_{m} &= \log\ga_1 + \log\ga_2 + \ldots + \log\ga_{m} 
\nonumber \\
&< {\log\pr{1 +  c_P\de_1} + \log(1 + 2P)} +  \log\pr{1 + c_P\de_2} + \log(1 + 2P + (2P)^2) - \log(1 + 2P) + \ldots
\nonumber \\
&+ \log\pr{1 + c_P\de_{m}} + \log(1 + 2P + \ldots + (2P)^{m}) - \log(1 + 2P + \ldots + (2P)^{m-1}) 
\nonumber \\
&= \sum_{j=1}^{m} \log\pr{1 +  c_P\de_j} + \log(1 + 2P + \ldots + (2P)^m)
\nonumber \\
&< m\log\pr{1 + c_P\de_1} + \log\pr{\tfrac{1}{1-2P}}
\label{C1GamBd}
\end{align}
Then for $R \ge R_0$,
\begin{align}
\log\log\pr{\mathcal{T}_0^{\Ga_{m}}}& = \log\Ga_{m} + \log\log \mathcal{T}_0 
\nonumber \\
&< m\log\pr{1 + c_P\de_1} + \log\pr{\tfrac{1}{1-2P}} + \log\log \mathcal{T}_0 \quad \textrm{by (\ref{C1GamBd})} 
\nonumber \\
&\le \pr{\tfrac{\log\log R}{\log(1/2P)} + 1}\log\pr{1 + c_P\de_1} + \log\pr{\tfrac{1}{1-2P}} + \log\log \mathcal{T}_0 \quad \textrm{by (\ref{mVals})} 
\nonumber \\
&\le \tfrac{1}{2}\log\log R + \tfrac{1}{2}\log\log R_0 \quad \textrm{by (\ref{conBdC1}) and the choice for $R_0$}
\nonumber \\
&\le \log\log R.
\label{case1logBd}
\end{align}
If $W \equiv 0$, an extension of (\ref{gamEstC2}) gives 
\begin{align*}
\ga_j &< \hg_j + c_N \de_j < \pr{1 + c_N \de_j} \frac{1 + 2N + \ldots + (2N)^j}{1 + 2N + \ldots + (2N)^{j-1}} \quad \textrm{by (\ref{hatGamN})}.
\end{align*}
By the same argument as above,
\begin{align*}
\log\Ga_{m} 
&< m\log\pr{1 + c_N\de_1} + \log\pr{\tfrac{1}{1-2N}},
\end{align*}
so for all $R \ge R_0$,
\begin{align}
\log\log\pr{\mathcal{T}_0^{\Ga_{m}}} &\le \log\log R.
\label{case2logBd}
\end{align}
If $W \not\equiv 0$ and $\disp \frac{4-2N}{3} > 2 - 2P$, an extension of (\ref{C3gaEst}) gives 
\begin{align*}
\ga_{J+j} &< \hg_{J+j} + c_N \de_{J+j} < \pr{1 + c_N \de_{J+j}} \frac{1 + 2N + \ldots + (2N)^j}{1 + 2N + \ldots + (2N)^{j-1}} \quad \textrm{by (\ref{hatGamNJ})}.
\end{align*}
Thus,
\begin{align}
\log\Ga_{m} &= \log\ga_1 + \ldots + \log\ga_J + \log\ga_{J+1} + \log\ga_{J+2} + \ldots \log\ga_{m} 
\nonumber \\
&< \log\ga_1 + \ldots + \log\ga_J  
\nonumber \\
&+ \log\pr{1 +  c_N\de_{J+1}} + \log(1 + 2N) +  \log\pr{1 + c_N\de_{J+2}} + \log(1 + 2N + (2N)^2) - \log(1 + 2N) + \ldots 
\nonumber \\
&+ \log\pr{1 + c_N\de_{m}} + \log(1 + 2N + \ldots + (2N)^{m}) - \log(1 + 2N + \ldots + (2N)^{m-1})
\nonumber \\
&= \sum_{j=1}^{J} \log\ga_j  + \sum_{j=J+1}^{m} \log\pr{1 +  c_N\de_j} + \log(1 + 2N + \ldots + (2N)^m) \nonumber \\
&< J\log\ga_1 + (m-J)\log\pr{1 + c_N\de_{J+1}} + \log\pr{\tfrac{1}{1-2N}}.
\label{C3GamBd}
\end{align}
Then for all $R \ge R_0$,
\begin{align}
\log\log\pr{\mathcal{T}_0^{\Ga_{m}}}& = \log\Ga_{m} + \log\log \mathcal{T}_0 \nonumber \\
&< J\log\ga_1 + (m-J)\log\pr{1 + c_N\de_{J+1}} + \log\pr{\tfrac{1}{1-2N}} + \log\log \mathcal{T}_0 \quad \textrm{by (\ref{C3GamBd})} \nonumber \\
&\le \mathcal{J}\log\ga_1 + \pr{\tfrac{\log\log R}{\log(1/2N)} + 1}\log\pr{1 + c_N\de_{J+1}} + \log\pr{\tfrac{1}{1-2N}} + \log\log \mathcal{T}_0 \quad \textrm{by (\ref{mVals})} \nonumber \\
&\le \tfrac{1}{2}\log\log R + \tfrac{1}{2}\log\log R_0 \quad \textrm{by (\ref{conBdC1}) and the choice for $R_0$} \nonumber \\
&\le \log\log R.
\label{case3logBd}
\end{align}
By (\ref{case1logBd}), (\ref{case2logBd}) and (\ref{case3logBd}), $\mathcal{T}_0^{\Ga_{m}} \le R$ for all $R \ge R_0$.  Lemma \ref{Tinc} shows that $\disp \frac{d \pr{T_1^{\Ga_{m\pr{R}}(T_1)}}}{d T_1} > 0$.  Therefore, for any $R \ge R_0$, there exists a $T_1 \ge \mathcal{T}_0$ such that $T_{1}^{\Ga_m} = R$, as required.
\end{proof} 

We have shown that there is always an $m$ large enough so that estimate (\ref{bem1}) holds.  With this choice of $m$, we have shown that we can choose $R_0$ so that for all $R \ge R_0$, the corresponding starting point, $T_1$, is always sufficiently far enough away from the origin for all of the above claims to hold.  We are now prepared to prove Theorem \ref{MEst}(\ref{partA}).

\begin{proof}[Proof of Theorem \ref{MEst}(\ref{partA})]
Fix $N, P \ge 0$ so that $\be_0 > 1$. If $\be_0 = \be_1$, then by Proposition \ref{baseC}, the result follows.  Otherwise, define a sequence of positive real numbers $\{T_j\}_{j=1}^{m+1}$ such that $T_{j+1} = T_j^{\ga_j}$ for each $j$ and $T_{m+1} = R$, where $m$ is given by Lemma \ref{mLem} and the determination of $\ga_j$ is described above in (\ref{gamDef}).  Since $R \ge R_0$ then by Lemma \ref{RNPLem}, we know that $T_1$ is large enough for all of the lemmas and conditions above to hold.

Let $x_0 \in \R^n$ be such that $|x_0| = R$ and $\disp \mathbf{M}(R) = \pr{\int_{B_1(x_0)}\abs{u}^2}^{1/2}$.  For each $j = 1, 2, \ldots, m$, let $\disp x_j = \frac{x_0}{|x_0|}T_j$.  Notice that $x_0 = x_{m+1}$.

\nid By Proposition \ref{baseC},
\begin{equation}
\int_{B_1(x_1)}\abs{u}^2 \ge C_5 \exp\pr{-C_4T_1^{\be_1} \log T_1}.
\label{it1a}
\end{equation}
By (\ref{it1a}) and Proposition \ref{IH},
\begin{equation}
\int_{B_1(x_2)} \abs{u}^2 \ge C_5 \exp\pr{ -\tilde C_4 T_2^{ \be_2}\log T_2 }.
\label{it2a}
\end{equation}
Repeating the argument for each $j$, we see that
\begin{equation*}
\int_{B_1(x_j)} \abs{u}^2 \ge C_5 \exp\pr{ -\tilde C_4 T_j^{ \be_j}\log T_j }.
\end{equation*}
Therefore,
\begin{align*}
\int_{B_1(x_0)} \abs{u}^2 &\ge C_5 \exp\pr{ -\tilde C_4T_{m+1}^{\be_{m+1}}\log T_{m+1} } \\
&=  C_5\exp\pr{ -\tilde C_4 R^{\be_{m+1}}\log R } \\
&\ge C_5 \exp\pr{  -\tilde C_4 R^{\be_0}R^{(C_6-1)\frac{\log\log R}{\log R}}\log R } \quad \textrm{(by Lemma \ref{mLem})}\\
&= C_5 \exp\pr{  -\tilde C_4 R^{\be_0}(\log R)^{C_6} }.
\end{align*}
If we set $C_7 = \frac{\tilde C_4}{2}$ and $\tilde C_5 = \sqrt{C_5}$, the proof is complete.
\end{proof}

%
%
\subsection{Proof of Theorem \ref{MEst}(\ref{partB})}

Before we prove Theorem \ref{MEst}(\ref{partB}), we will show that for some $C_6(N,P)$,
\begin{equation*}
\pr{\be_{m+1}- 1}\log T_{m+1} \le \pr{C_6-1}\pr{\log\log R}^2.
\end{equation*}
This requires a couple of lemmas.  Note that since $\be_c < 1$, then $2N, 2P > 1$.

\begin{lem}
For any $j \in \N$,
\begin{align*}
&\pr{\be_{j+1}-1}\log T_{j+1} = \pr{\be_1 - 1}\log T_1 +  \log \pr{\frac{C_4}{3 \cdot 4^{1 + P + N/3} C_3 w\pr{5/4}^{2/3} \pr{A_1^{2/3} + A_2^2} } } + j\log \pr{\frac{27}{8c_n}} + \sum_{k=1}^j \log \log T_k.
\end{align*}
\label{lem1}
\end{lem}

\begin{lem}
If $\log T_1 \lesssim \pr{\log\log R}^2$, then
$$ \pr{\be_{m+1}-1}\log T_{m+1} \lesssim \pr{\log\log R}^2.$$
\label{lem2}
\end{lem}

And to show that this method cannot do any better, we will prove that
\begin{lem}
$$ \pr{\be_{m+1}-1}\log T_{m+1} \gtrsim \pr{\log\log R}^2.$$
\label{lem3}
\end{lem}

Our analysis depends on the values of $N$ and $P$, so it is broken into the 3 cases. \\
\\
\nid \textbf{Case 1:} (a) $\disp 2 - 2P \ge \frac{4-2N}{3}$ or (b) $\disp 2 - 2P < \frac{4-2N}{3}$ with $P \le N$.\\

In Case 1(a), by Lemma \ref{case1}, $2-2P \ge h_j$ for all $j\ge 1$.  Since $\be_j \ge 1 > 2 - 2P$ for all $j \ge 1$, then we always define our sequences using the first choice.  \\

For Case 1(b), since $P \le N$ and $3P - N > 1$, then $\om_j << \de_j$ so that $h_j \le 1$ for all $j$, so again, we always define our sequences using the first choice.

\begin{proof}[Proof of Lemma \ref{lem1} for Case 1]
Notice that
\begin{align*}
\pr{\be_{j+1} -1}\log T_{j+1} &= \pr{1 - \frac{2P}{\ga_j}}\log\pr{ T_{j}^{\ga_j}} \\
&= \pr{\frac{\ga_j - 2P}{\ga_j}}{\ga_j}\log\pr{ T_{j}} \\
&= \pr{\be_j -1 + 2P + \de_j - 2P}\log T_{j} \\
&= \pr{\be_j -1}\log T_{j} + \de_j\log T_{j}.
\end{align*}
So by repeatedly applying this rule, we see that
\begin{align*}
\pr{\be_{j+1} -1}\log T_{j+1} &= \pr{\be_{1} -1}\log T_{1} + \de_{1}\log T_{1} + \de_2\log T_2 + \ldots +  \de_j\log T_{j} \\
&= \pr{\be_{1} -1}\log T_1 + \sum_{k=1}^j \log \pr{T_k^{\de_k}} 
\end{align*}
If we { use $\disp T_1^{\de_1} = \frac{C_4}{3\cdot 4^{1 + P + N/3} C_3 w\pr{5/4}^{2/3} \pr{A_1^{2/3} + A_2^2}} \frac{27}{8 c_n} \log T_1,$ (from (\ref{T1ga1})) and} (\ref{Tjgaj}) to simplify this expression, we get the claimed result.
\end{proof}

\begin{proof}[Proof of Lemma \ref{lem2} for Case 1]
Since the $T_j$s increase, then $T_j \le R$ and using the expression from Lemma \ref{lem1} we see that
$$\pr{\be_{m+1} -1}\log T_{m+1} \le \log T_1 +  \log \pr{\frac{C_4}{3 \cdot 4^{1 + P + N/3} C_3 w\pr{5/4}^{2/3} \pr{A_1^{2/3} + A_2^2}} } + m \log \pr{\frac{27}{8 c_n}} + m \log \log R.$$  
However, $\disp \Ga_m \ge \pr{2P}^m$, so $\disp m \le \frac{\log \Ga_m}{\log \pr{2P}} = \frac{\log\log R - \log\log T_1}{\log \pr{2P}} < \frac{\log\log R}{\log \pr{2P}}$.  Combining this with the hypothesis on $T_1$, we see that
$$\pr{\be_{m+1} -1}\log T_{m+1} \lesssim \pr{\log\log R}^2.$$
\end{proof}

\begin{proof}[Proof of Lemma \ref{lem3} for Case 1]
Since $T_{j+1} = T_1^{\Ga_j}$, then
\begin{align*}
\sum_{j=1}^m \log \log T_j &= \log\pr{\log T_1 \log T_2 \ldots \log T_m} \\
&= \log\pr{\Ga_1 \Ga_2 \ldots \Ga_{m-1} \pr{\log T_1}^m} \\
&= m\log\log T_1 + \log\pr{\Ga_1 \Ga_2 \ldots \Ga_{m-1}}
\end{align*}
Since $\Ga_j \ge \pr{2P}^j$, then
$$\sum_{j=1}^m \log \log T_j \ge m\log\log T_1 + \log\brac{\pr{2P}^{\sum_{j=1}^{m-1}j}} > m\log\log T_1 + cm^2 \log\pr{2P} $$
We know that $T_1^{\Ga_m} = R$, so $\disp \Ga_m = \frac{\log R}{\log T_1}$.  Since $\Ga_m \le \ga_1^m$, then $\disp m \ge \frac{\log \Ga_m}{\log \ga_1} = \frac{\log\log R - \log\log T_1}{\log \ga_1}$.  Thus, there exists some $c > 0$ so that $m \ge c \log\log R$.  It follows that
$$\pr{\be_{m+1} -1}\log T_{m+1} \gtrsim \pr{\log\log R}^2.$$
\end{proof}

\nid \textbf{Case 2:} $W \equiv 0$ \\

Since $W \equiv 0$, we always make the second choice to define $\be_{j+1}$.

\begin{proof}[Proof of Lemma \ref{lem1} for Case 2]
Notice that
\begin{align*}
\pr{\be_{j+1} -1}\log T_{j+1} &= \pr{\frac{1}{3} - \frac{2N}{3\ga_j}}\log\pr{ T_{j}^{\ga_j}} \\
&= \pr{\frac{\ga_j - 2N}{3\ga_j}}{\ga_j}\log\pr{ T_{j}} \\
&= \pr{\frac{3(\be_j -1) + 2N + 3\de_j - 2N}{3}}\log T_{j} \\
&= \pr{\be_j -1}\log T_{j} + \de_j\log T_{j}.
\end{align*}
The rest of the proof is as in Case 1.
\end{proof}

\begin{proof}[Proof of Lemma \ref{lem2} for Case 2]
We use the fact that $\disp \Ga_m \ge \pr{2N}^m$ and proceed as in Case 1.
\end{proof}

\begin{proof}[Proof of Lemma \ref{lem3} for Case 2]
We use the fact that $\disp \Ga_m \ge \pr{2N}^m$ and proceed as in Case 1.
\end{proof}

\nid \textbf{Case 3:} $\disp 2 - 2P < \frac{4-2N}{3}$, $W \not\equiv 0$ and $P > N$. \\ 

Since $\disp 2 - 2P < \frac{4-2N}{3}$, then there exists $\De > 0$ such that $3P - N = 1 + \De$.  By Lemma \ref{JLem}, there exists a $J$ such that $\be_{J-1} = h_{J-1}$.  Furthermore, $\be_k \ge h_k$ for all $k \le J-1$.  { Since Lemma \ref{decLem} only uses that $P - N > 0$, and does not require that $\be_0 > 1$, Lemma \ref{decLem} and Corollary \ref{decLemCor} still hold for case (b).  Thus,} we initially use only the first choice to define our sequences.  And once we switch to using the second choice, which inevitably happens, we use only the second choice.  That is, we are in Case 3.

\begin{proof}[Proof of Lemma \ref{lem1} for Case 3]
We note that $\be_j \le \ell_j$ for all $j \ge J$, $\be_{J-1} = h_{J-1}$ and $\be_j \ge h_j$ for all $j \le J - 2$.  As in the proofs for Case 1 and 2, we see that
\begin{align*}
\pr{\be_{j+1} -1}\log T_{j+1} &= \pr{\be_j -1}\log T_{j} + \de_j\log T_{j} \quad \textrm{for all $j$}.
\end{align*}
So, as in Cases 1 and 2, we can repeatedly apply this rule to get the result.
\end{proof}

\begin{proof}[Proof of Lemma \ref{lem2} for Case 3]
 We use the fact that $\disp \Ga_m \ge \pr{2P}^{J-1}\pr{2N}^{m+1-J} \ge \pr{2N}^m$ since $P > N$, and proceed as in Case 1.
\end{proof}

\begin{proof}[Proof of Lemma \ref{lem3} for Case 3]
We use the fact that $\Ga_j \ge \pr{2N}^j$ and proceed as in Case 1.
\end{proof}

We have now proved our lemmas for all possible cases.  We need $T_1$ to be sufficiently large in order for the above results to be valid.  Also, in order to use Lemma \ref{lem2}, we must ensure that $T_1$ is bounded above by some function of $R$.  

\begin{lem}
Let $T_{N,P}$ be a value of $T_1$ that is large enough for { Lemma \ref{Tinc} and Corollary \ref{GaQuo}} to hold.  Set $R_0 = e^{T_{N,P}}$.  There exists a constant $C$, depending on $N$, $P$ and $T_{N,P}$, such that for every $R \ge R_0$, there exists $T_1 \in \brac{T_{N,P}, T_{N,P}^{C}}$ and $m \in \N$, such that $T_1^{\Ga_m(T_1)} = R$.  Moreover, $\log T_1 \le C\pr{\log\log R}^2$.
\label{TChoiceb}
\end{lem}

\begin{proof}
Fix $R \ge R_0$.  Let $m$ be the largest integer such that $T_{N,P}^{\Ga_m(T_{N,P})} \le R$.  Then $R < T_{N,P}^{\Ga_{m+1}(T_{N,P})}$.  By Lemma \ref{Tinc} (which also applies when $\be_0 = 1$), there exists $T_1 \ge T_{N,P}$ such that $T_1^{\Ga_m(T_1)} = R$.  We need to show that $T_1$ is bounded above.  Combining inequalities, we see that $$T_1^{\Ga_m(T_1)} = R < T_{N,P}^{\Ga_{m+1}(T_{N,P})} = \pr{T_{N,P}^{\Ga_{m}(T_{N,P})}}^{\ga_{m+1}(T_{N,P})},$$
or 
$$T_1 < \pr{T_{N,P}^{\pr{\Ga_{m}(T_{N,P})/\Ga_m(T_1)}}}^{\ga_{m+1}(T_{N,P})}.$$  
By Corollary \ref{GaQuo}, $\disp \Ga_{m}(T_{N,P})/\Ga_m(T_1) \le c$, where $c$ depends on $N$, $P$ and $T_{N,P}$.  Furthermore, $\ga_{m+1}(T_{N,P}) \le \ga_{1}(T_{N,P}) \le \tilde{c}$.  Letting $C = c \tilde{c}$, we see that $T_1 \le T_{N,P}^C$.
In addition, $\log T_1 \le C\log T_{N,P} = C \log\log R_0 \le C \pr{\log\log R}^2$.
\end{proof}

We may now proceed to the proof of the main theorem.

\begin{proof}[Proof of Theorem \ref{MEst}(\ref{partB})]
Fix $N, P > 1/2$.  Consider $R \ge R_0 = e^{T_{N,P}}$, where $T_{N,P}$ is introduced in Lemma \ref{TChoiceb}.  By Lemma \ref{TChoiceb}, there exists $T_1 \ge T_{N,P}$, $m\in\N$ such that $T_1^{\Ga_m} = R$.  Furthermore, $\log T_1 \le C \pr{\log\log R}^2$, so we may apply Lemma \ref{lem2}.

Define a sequence of positive real numbers $\{T_j\}_{j=1}^{m+1}$ such that $T_{j+1} = T_j^{\ga_j}$ for each $j$ and $T_{m+1} = R$, where the determination of $\ga_j$ is described above in (\ref{gamDef}).

Let $x_0 \in \R^n$ be such that $|x_0| = R$ and $\disp \mathbf{M}(R) =\pr{ \int_{B_1(x_0)}\abs{u}^2}^{1/2}$.  For each $j = 1, 2, \ldots, m$, let $\disp x_j = \frac{x_0}{|x_0|}T_j$.  Notice that $x_0 = x_{m+1}$.

\nid By Proposition \ref{baseC},
\begin{equation}
\int_{B_1(x_1)}\abs{u}^2 \ge C_5 \exp\pr{-C_4T_1^{\be_1} \log T_1}.
\label{it1}
\end{equation}
By (\ref{it1}) and Proposition \ref{IH},
\begin{equation}
\int_{B_1(x_2)} \abs{u}^2 \ge C_5 \exp\pr{ -\tilde C_4 T_2^{ \be_2}\log T_2 }.
\label{it2}
\end{equation}
Repeating the argument for each $j$, we see that
\begin{equation*}
\int_{B_1(x_j)} \abs{u}^2 \ge C_5 \exp\pr{ -\tilde C_4 T_j^{ \be_j}\log T_j }.
\end{equation*}
Therefore,
\begin{align}
\int_{B_1(x_0)} \abs{u}^2 &\ge C_5 \exp\pr{ -\tilde C_4 T_{m+1}^{\be_{m+1}}\log T_{m+1} } \nonumber \\
&= C_5 \exp\pr{ -\tilde C_4 R\log R \;T_{m+1}^{\be_{m+1}-1}}.
\label{mth}
\end{align}
By Lemma \ref{lem2}, 
$$T_{m+1}^{\be_{m+1}-1} = \exp\brac{\pr{\be_{m+1}-1}\log T_{m+1}} \le \exp\brac{C \pr{\log\log R}^2} = \pr{\log R}^{C\log\log R}$$
Returning to inequality (\ref{mth}), we see that
$$\int_{B_1(x_0)} \abs{u}^2 \ge C_5 \exp\pr{ -\tilde C_4 R \pr{\log R}^{C_6\log\log R}}.$$
If we set $C_7 = \frac{\tilde C_4}{2}$ and $\tilde C_5 = \sqrt{C_5}$, the proof is complete.
\end{proof}

\subsection{Comments on $\be_c = 1$}

Recall that $\be_c := \max\set{\frac{4-2N}{3}, 2-2P}$.  The reader may have noticed that there are no results for the cases when $\be_c =1$.  In this subsection, we will explain why the iterative argument breaks down and what happens in the limit when $\be_0 \to 1$ from above.

From the proof of Theorem \ref{MEst}, we have that
$$\int_{B_1(x_0)} \abs{u}^2 \ge  C_5 \exp\pr{ -\tilde C_4 T_{m+1}^{\be_{m+1}}\log T_{m+1} } = C_5 \exp\pr{ -\tilde C_4 R\log R \;R^{\be_{m+1}-1}}.$$
Therefore, if we want to get a result similar to Theorem \ref{MEst} for the cases when $\be_c  = 1$, we need to ensure that $R^{\be_{m+1}-1} \le R^\eps$ for any $\eps > 0$.  If we choose $m = C\frac{\log R}{\pr{\log\log R}^k}$ for some $k \in \N$, then by Lemma \ref{be1Diff}, $R^{\be_{m+1}-1} \gtrsim \pr{\log R}^{\pr{\log\log R}^{k-1}}$, so there may be a hope of achieving this bound.  However, if we make $m$ of a smaller order, then we cannot establish the desired bound.  Suppose we have $R = T_{m+1} = T_1^{\Ga_m}$ with $m = C\frac{\log R}{\pr{\log\log R}^k}$ for some $k \in \N$.  By Lemma \ref{lT1Zero}, $\log T_1 \to 0$ as $R \to \iny$.  However, in Proposition \ref{baseC}, it is imperative that $\log T_1$ be much larger than 0.  Thus, there exists $R_0 \in \R$ such that for all $R \ge R_{0}$, the corresponding $T_1$ is not large enough for one of our main propositions to apply.  This shows that the iterative argument breaks down.

It is interesting that the iterative argument works for both $\be_c > 1$ and $\be_c < 1$, but fails when $\be_c = 1$  However, it is perhaps not that surprising since the behavior of our sequences is exponential for $\be_c \ne 1$ but linear for $\be_c = 1$.

If $\be_c = 1$, then $V$ and $W$ satisfy the hypotheses for Theorem \ref{MEst}(\ref{partA}) for any $\be_c > 1$.  Therefore, for any $\eps > 0$, $\mathbf{M}( R ) \ge \tilde C_5\exp\pr{-C_7 R^{1+\eps}\log R ^{C_6}}$, where $\tilde C_5$ and $C_7$ are bounded, $C_6 = C_6(1 + \eps)$.  If we study the proof of Lemma \ref{mLem}, we see that $C_6 \lesssim \frac{1}{\eps}$.  Since $\disp \lim_{\eps \to 0}\brac{R^\eps \pr{\log R}^{c/\eps}} \to \iny$, we cannot establish any result for $\be_c = 1$ by looking at the limiting behavior of the result from Theorem \ref{MEst}(\ref{partA}).

%
%
\section{Proof of Theorem \ref{cons}}
\label{Mesh}

To prove Theorem \ref{cons}(\ref{consa}), we will prove Proposition \ref{consP} below with a Meshkov-type construction.   Some straightforward calculations prove Proposition \ref{consD}, which gives us Theorem \ref{cons}(\ref{consb}).  These propositions show that we do not need both potentials.  That is, if $V$ is the dominant potential (meaning that the decay of $W$ is faster), then we can construct a solution to equation (\ref{epde}) with $W \equiv 0$.  And if $W$ is the dominant potential, then we can construct a solution to equation (\ref{epde}) with $V \equiv 0$.  We will consider each proposition separately.  

\subsection{Proof of Theorem \ref{cons}(\ref{consa})}

\begin{prop}  For any $\la \in \C$, we have the following.
\begin{enumerate}[(a)]
\item If $\disp \be_0 = \frac{4-2N}{3} > 1$, then there exists a potential $V$ and an eigenfunction $u$ such that
\begin{equation}
\LP u + \la u = V u,
\label{VPDE}
\end{equation}
where
\begin{equation}
|V(x)| \le C \langle x \rangle ^{-N}
\label{Vbd}
\end{equation}
and
\begin{equation}
|u(x)| \le C\exp\pr{-c\frac{|x|^{\be_0}}{\log |x|}}.
\label{Vubd}
\end{equation}
\item  If $\be_0 = 2-2P > 1$, then there exists a potential $W$ and an eigenfunction $u$ such that
\begin{equation}\LP u + \la u = W \cdot \gr u,
\label{WPDE}
\end{equation}
where
\begin{equation}
|W(x)| \le C\langle x \rangle ^{-P}
\label{Wbd}
\end{equation}
and
\begin{equation}
|u(x)| \le C\exp\pr{-c |x|^{\be_0}}.
\label{Wubd}
\end{equation}
\end{enumerate}
\label{consP}
\end{prop}

Before presenting the constructions for Proposition \ref{consP}, we need a couple of lemmas.  For $\la \in \C$, use the principal branch to define
$$ \mu_n( r ) := \exp\pr{n\brac{\log\pr{\sqrt{1 - \frac{\la r^2}{n^2}}+1}- \sqrt{1 - \frac{\la r^2}{n^2}} - \log 2 + 1}}.$$
As we will specify below, $n >> r$, so $ \abs{\frac{\la r^2}{n^2}} < 1$. It follows that $ \Re \pr{1 - \frac{\la r^2}{n^2}} > 0$, so all square root terms are well defined (and have positive real part) with this choice of branch cut.  Since the argument for the logarithmic term has real part greater than 1, that term, and hence the function $\mu_n$, is well defined with this branch choice.
A power series expansion of the exponent gives
$$\mu_n\pr{r} = \exp\pr{\frac{\la r^2}{4n} + \frac{\la^2 r^4}{32 n^3} + \frac{\la^3 r^6}{96 n^5} + \ldots}.$$
Whenever $ \abs{\frac{\la r^2}{n^2}} < 1$, the power series in the exponent converges everywhere.

\begin{lem}
Suppose $\rho$ is a large positive number $(\rho > \rho_0 > 0)$, $\disp \be_0 = \frac{4-2N}{3} > 1$, $n \in \N$ is such that $\disp \abs{n - \frac{\rho^{\be_0}}{\log \rho}} \le 1$ and $k \in \N$ is such that $\disp \abs{k - 6\pr{\be_0 - \frac{1}{\log \rho}}\frac{\rho^{\be_0/2}}{\log \rho} } \le 10$.  Let $\disp \al =1 - \tfrac{\be_0}{2}$.  Then in the annulus $\brac{\rho, \rhoa{6}}$ it is possible to construct an equation of the form (\ref{VPDE}) and a solution $u$ of this equation such that the following hold:
\begin{enumerate}
\item (\ref{Vbd}), where $C$ does not depend on $\rho$, $n$ or $k$.
\item If $r \in \brac{\rho, \rhoa{0.1}}$, then $u = r^{-n}e^{-in\vp}\mu_n$. \\
If $r \in \brac{\rhoa{5.9}, \rhoa{6}}$, then $u = ar^{-(n+k)}e^{-i(n+k)\vp}\mu_{n+k}$, for some $a\in \C\setminus\set{0}$.
\item Let $m(r) = \max\set{|u(r, \vp)| : 0 \le \vp \le 2\pi}$.  Then there exists a $c > 0$, not depending on $\rho$, $n$ or $k$, such that
\begin{equation}
\ln m(r) - \ln m(\rho) \le - c\int_\rho^r \frac{t^{\be_0-1}}{\log t}dt + \ln 2
\label{Nint}
\end{equation}
for any $r \in \brac{\rho, \rhoa{6}}$.
\end{enumerate}
\label{meshN}
\end{lem}

\begin{lem}
Suppose $\rho$ is a large positive number $(\rho > \rho_0 > 0)$, $\be_0 = 2-2P > 1$, $n \in \N$ is such that $\disp \abs{n - \rho^{\be_0}} \le 1$ and $k \in \N$ is such that $\disp \abs{k -6\be_0\rho^{\be_0/2}} \le 40$.  Let $\disp \al =1 - \tfrac{\be_0}{2}$.  Then in the annulus $\brac{\rho, \rhoa{6}}$ it is possible to construct an equation of the form (\ref{WPDE}) and a solution $u$ of this equation such that the following hold:
\begin{enumerate}
\item (\ref{Wbd}), where $C$ does not depend on $\rho$, $n$ or $k$.
\item If $r \in \brac{\rho, \rhoa{0.1}}$, then $u = r^{-n}e^{-in\vp}\mu_n$. \\
If $r \in \brac{\rhoa{5.9}, \rhoa{6}}$, then $u = ar^{-(n+k)}e^{-i(n+k)\vp}\mu_{n+k}$, for some  $a\in \C\setminus\set{0}$.
\item Let $m(r) = \max\set{|u(r, \vp)| : 0 \le \vp \le 2\pi}$.  Then there exists a $c > 0$, not depending on $\rho$, $n$ or $k$, such that
\begin{equation}
\ln m(r) - \ln m(\rho) \le - c\int_\rho^r t^{\be_0-1}dt + \ln 2
\label{Pint}
\end{equation}
for any $r \in \brac{\rho, \rhoa{6}}$.
\end{enumerate}
\label{meshP}
\end{lem}

Since these lemmas are so similar, it is not surprising that their proofs are as well.  We will present the more complicated proof first, that of Lemma \ref{meshN}.  We will then show the proof of Lemma \ref{meshP}.

\begin{pf}[Proof of Lemma \ref{meshN}]
As $r$ increases from $\rho$ to $\rhoa{6}$, we rearrange equation (\ref{VPDE}) and its solution $u$ so that all of the above conditions are met.  This process is broken down into four major steps.

Throughout this proof, the number $C$ is a constant that is independent of $\rho$, $n$ and $k$. \\
\nid\emph{Step 1: $r \in \brac{\rho, \rhoa{2}}$}.  During this step, the function $u_1 = r^{-n}e^{-in\vp}\mu_{n}( r )$ is rearranged to $u_2 = -br^{-n+2k}e^{iF(\vp)}\mu_{n-2k}( r )$, both of which satisfy an equation of the form (\ref{VPDE}), where $b$ is a complex number and $F$ is a function that will be defined shortly.

Let $\vp_m = 2\pi m/(2n + 2k)$, for $m = 0, 1 , \ldots, 2n + 2k -1$.  Then $\disp \set{\vp_m}_{m=0}^{2n+2k-1}$ is the set of all solutions to $\disp e^{-in\vp} - e^{i(n+2k)\vp} = 0$ on $\set{0 \le \vp \le 2\pi}$.  Let $T = \pi/(n+k)$.  On $\brac{0, T}$, we define $f$ to be a $C^1$ function such that $f(\vp) = -4k$ for $\vp \in \brac{0, T/5} \cup \brac{4T/5, T}$.  We also require that $f$ satisfies the following:
\begin{align}
&-4k \le f(\vp) \le 5k, \;\; 0 \le \vp \le T,
\label{f1} \\
&\int_0^T f(\vp) d\vp = 0,
\label{f2} \\
&|f^\pri(\vp)| \le Ck/T = Ck(k+n)/\pi,  \;\; 0 \le \vp \le T.
\label{f3}
\end{align}
We extend $f$ periodically (with period $T$) to all of $\R$ and set
$$\Phi(\vp) = \int_0^\vp f(t) dt.$$
By (\ref{f2}), $\Phi$ is $T$-periodic and $\Phi(\vp_m) = \Phi(mT) = 0$.  Furthermore, $\Phi$ is $2\pi$-periodic.  By (\ref{f1})-(\ref{f3}), the following facts hold for all $\vp \in \R$:
\begin{align}
&|\Phi(\vp)| \le 5kT = 5\pi k/(n+k),
\label{Phi1} \\
&|\Phi^\pri(\vp)| \le 5k,
\label{Phi2} \\
&|\Phi^{\pri\pri}(\vp)| \le Ckn.
\label{Phi3} 
\end{align}
Also, for all $\vp \in \set{|\vp - \vp_m| \le T/5}$, 
\begin{equation}
\Phi(\vp) = -4k(\vp - \vp_m) = -4k\vp + b_m,
\label{Phi4}
\end{equation}
where $b_m$ is some real number.

Set 
\begin{equation}
F(\vp) = (n+2k)\vp + \Phi(\vp).
\label{FDef}
\end{equation}  
If $|\vp - \vp_m| \le T/5$, then $u_2 = -be^{ib_m}r^{-(n-2k)}e^{i(n-2k)\vp}\mu_{n-2k}( r )$.  \\

Choose $b = (\rhoo)^{-2k}\frac{\mu_n\pr{\rhoo}}{\mu_{n-2k}\pr{\rhoo}}$ so that $|u_1(\rhoo, \vp)| = |u_2(\rhoo, \vp)|$.  Since $|u_2(r, \vp)/u_1(r, \vp)| =\abs{ br^{2k}\frac{\mu_{n-2k}( r )}{\mu_n\pr{r}}}$, then by the assumptions on $k$ and $\rho$ and the behavior of $\mu_n$ and $\mu_{n-2k}$,
\begin{align}
&|u_2(r, \vp)/u_1(r, \vp)| \le e^{-C}, \quad r \in \brac{\rho, \rhoa{\tfrac{2}{3}}}
\label{uBd1} \\
&|u_2(r, \vp)/u_1(r, \vp)| \ge  e^{C}, \quad r \in \brac{\rhoa{\tfrac{4}{3}}, \rhoa{2}}.
\label{uBd2}
\end{align}
 
 Choose smooth cutoff functions $\psi_1$, $\psi_2$, $\psi_3$, $\psi_4$ such that $\psi_1(r) = \left\{ \begin{array}{rl} 1 & \mathrm{if} \quad r \le \rhoa{\tfrac{4}{3}} \\ 0 & \mathrm{if} \quad r \ge \rhoa{\tfrac{5}{3}} \end{array}\right.$, \\
$\psi_2(r) = \left\{ \begin{array}{rl} 0 & \mathrm{if} \quad r \le \rhoa{\tfrac{1}{3}} \\ 1 & \mathrm{if} \quad r \ge \rhoa{\tfrac{2}{3}} \end{array}\right.$, $\psi_3\pr{r} = \left\{ \begin{array}{rl} 1 & \mathrm{if} \quad r \le \rhoa{\tfrac{5}{3}} \\ 0 & \mathrm{if} \quad r \ge \rhoa{1.9} \end{array}\right.$ and $\psi_4\pr{r} = \left\{ \begin{array}{rl} 0 & \mathrm{if} \quad r \le \rhoa{0.1} \\ 1 & \mathrm{if} \quad r \ge \rhoa{\tfrac{1}{3}} \end{array}\right.$.  Moreover, we require that
\begin{equation}
0 \le |\psi_i(r)| \le 1 \quad\mathrm{and}\quad |\psi_i^{(j)}(r)| \le Cr^{-j\al} \quad \forall r \in \R^+, i= 1,2,3,4, j = 1, 2.  
\label{1psi}
\end{equation}  

Let
\begin{align*}
\phi_{a,b}\pr{r} &= -\frac{1}{2}\int \frac{\la r}{\sqrt{a^2 - \la r^2}\sqrt{b^2 - \la r^2}} dr \\
&= -\frac{1}{4}\log \pr{2\sqrt{a^2 - \la r^2}\sqrt{b^2 - \la r^2} + 2\la r^2 - a^2 - b^2}
\end{align*}
When $a, b = n + \bigO\pr{k}$, we see that
\begin{align*}
\phi \pr{r}&= \bigO\pr{\log r} \\
\phi^\prime\pr{r} &= \bigO\pr{r^{1-2\be_0}\pr{\log r}^2} \\
\phi^{\prime\prime}\pr{r} &= \bigO\pr{r^{-2\be_0}\pr{\log r}^2}
\end{align*}

We set 
$$u = \psi_1u_1\exp\pr{\psi_4 \phi_{n, n-2k}} + \psi_2u_2\exp\pr{\psi_3 \phi_{n, n-2k}}.$$
For the rest of step 1, we will abbreviate $\phi_{n,n-2k}$ with $\phi$. \\

\nid \emph{Step 1A: $r \in \brac{\rho, \rhoa{\tfrac{2}{3}}}$}.

Since $\psi_1 \equiv 1$ and $\psi_3 = \psi_4$ on the support of $\psi_2$, then on this annulus,
$$u = u_1\exp\pr{\psi_4 \phi} + \psi_2u_2\exp\pr{\psi_3 \phi} =  \brac{u_1 + \psi_2u_2}\exp\pr{\psi_4 \phi}.$$
Therefore, by (\ref{uBd1}),
\begin{equation}
\exp\pr{-\psi_4 \phi}|u| \ge |u_1| - |u_2| \ge (1 - e^{-C})|u_1| \ge e^{C^\pri}|u_2| > 0.
\label{1AuN}
\end{equation}
We see that
\begin{align*}
\LP u + \la u &= \pr{D_1  + \tilde D_1}u_1\exp\pr{\psi_4 \phi},
\end{align*}
where 
\begin{align} 
D_1 &= \brac{\psi_2^\pri\tfrac{1 - 2\sqrt{\pr{n-2k}^2 - \la r^2}}{r}+ \psi_2^{\pri\pri} - \psi_2\pr{\tfrac{8nk + 2(n+2k)\Phi^\pri + (\Phi^\pri)^2 - i\Phi^{\pri\pri}}{r^2} - \tfrac{\la}{\sqrt{(n-2k)^2 - \la r^2}}}}\frac{u_2}{u_1} + \tfrac{\la}{\sqrt{n^2 - \la r^2}}  \nonumber \\
&= \bigO\pr{n\cdot r^{\be_0/2 -2}}
\label{D_1ord} \\
\tilde D_1 &= 2\brac{-\tfrac{\sqrt{n^2 - \la r^2}}{r} +\pr{\psi_2^\pri - \psi_2\tfrac{\sqrt{\pr{n-2k}^2 - \la r^2}}{r}}\frac{u_2}{u_1}}\pr{\psi_4^\pri \phi + \psi_4 \phi^\pri} \nonumber \\
&+ \brac{\frac{1}{r}\pr{\psi_4^\pri \phi + \psi_4 \phi^\pri} + \pr{\psi_4^\pri \phi + \psi_4 \phi^\pri}^2 + \pr{\psi_4^{\pri\pri} \phi + 2\psi_4^\pri \phi^\pri+ \psi_4 \phi^{\pri\pri}}}\pr{1 + \psi_2\frac{u_2}{u_1}} \nonumber \\
&= \bigO\pr{r^{3\be_0/2-2}}.
\label{tD_1ord} 
\end{align}
Let $\disp V = \pr{D_1+\tilde D_1}\frac{u_1\exp\pr{\psi_4 \phi}}{u}$ so that by (\ref{1AuN}), (\ref{D_1ord}) and (\ref{tD_1ord}), $|V| \le Cr^{3\be_0/2 -2} = Cr^{-N}$.  
This completes step 1A. \\ \\
\emph{Step 1B: $r \in \brac{\rhoa{\tfrac{4}{3}}, \rhoa{2}}$}.

Since $\psi_2 \equiv 1$ and $\psi_3 = \psi_4$ on the support of $\psi_1$, then on this annulus,
$$u = u_1\exp\pr{\psi_4 \phi} + \psi_2u_2\exp\pr{\psi_3 \phi} = \brac{\psi_1 u_1 + u_2} \exp\pr{\psi_3 \phi}.$$
On this annulus, by (\ref{uBd2}),
\begin{equation}
\exp\pr{-\psi_3 \phi}|u| \ge |u_2| - |u_1| \ge (1 - e^{-C})|u_2| \ge e^{C^\pri}|u_1| > 0.
\label{1BuN}
\end{equation}
We see that
\begin{align*}
\LP u + \la u &= \pr{E_1 + \tilde E_1}u_2\exp\pr{\psi_3 \phi},
\end{align*}
where 
\begin{align} 
E_1 &= \brac{\psi_1^\pri \tfrac{1 - 2\sqrt{n^2 - \la r^2}}{r} + \psi_1^{\pri\pri} +\psi_1\tfrac{\la}{\sqrt{n^2 - \la r^2}}}\frac{u_1}{u_2} + \tfrac{\la}{\sqrt{(n-2k)^2 - \la r^2}} - \tfrac{8nk + 2(n+2k)\Phi^\pri + (\Phi^\pri)^2 - i\Phi^{\pri\pri}}{r^2} \nonumber \\
&= \bigO\pr{n \cdot r^{\be_0/2 - 2}}
\label{E_1ord} \\
\tilde E_1 &=  2\brac{\psi_1^\pri \frac{u_1}{u_2} -\psi_1\tfrac{\sqrt{n^2 - \la r^2}}{r}\frac{u_1}{u_2}  -\tfrac{\sqrt{\pr{n-2k}^2 - \la r^2}}{r}}\pr{\psi_3^\pri \phi + \psi_3 \phi^\pri} \nonumber \\
&+ \brac{\frac{1}{r}\pr{\psi_3^\pri \phi + \psi_3 \phi^\pri} + \pr{\psi_3^\pri \phi + \psi_3 \phi^\pri}^2 + \pr{\psi_3^{\pri\pri} \phi + 2\psi_3^\pri \phi^\pri+ \psi_3 \phi^{\pri\pri}}}\pr{\psi_1 \frac{u_1}{u_2} + 1} \nonumber \\
&= \bigO\pr{r^{3\be_0/2-2}}.
\label{tE_1ord} 
\end{align}
Let $\disp V = \pr{E_1 + \tilde E_1}\frac{u_2\exp\pr{\psi_3 \phi}}{u}$ so that by (\ref{1BuN}), (\ref{E_1ord}) and (\ref{tE_1ord}), $|V| \le Cr^{3\be_0/2 -2} = Cr^{-N}$.  
This completes step 1B. \\ 
\\
\emph{Step 1C: $r \in \brac{\rhoa{\tfrac{2}{3}}, \rhoa{\tfrac{4}{3}}}$.}

On this annulus, $\psi_j \equiv 1$ for $j = 1, \ldots, 4$, so 
$$u = \pr{u_1 + u_2}\exp\pr{\phi} = \pr{r^{-n}e^{-in\vp}\mu_n\pr{r} - br^{-n+2k}e^{iF(\vp)}\mu_{n-2k}\pr{r}}\exp\pr{\phi},$$
Since $u_2 = -br^{-n+2k}e^{i(n-2k)}\mu_{n-2k}( r )$ on $\set{|\vp -\vp_m| \le \frac{T}{5}}$ for $m = 0, 1, \ldots, 2n+2k-1$, then we will first consider these regions.  We have
\begin{align*}
\LP u + \la u &= J_1u,
\end{align*}
where 
\begin{align} 
J_1 &=  \frac{\la}{\sqrt{(n-2k)^2 - \la r^2}} +  \frac{\la}{\sqrt{n^2 - \la r^2}} + \frac{\phi^\pri}{r} + \pr{\phi^\pri}^2 + \phi^{\pri\pri} = \bigO\pr{r^{-\be_0}\log r}.
\label{J_1ordN} 
\end{align}
\\
As long as $\rho_0 >> 1$, then $V = J_1 \le C r^{-N}$.
\\
Now we will consider the annular sectors 
$$P_m = \set{(r, \vp) : r \in \brac{\rho + \tfrac{2}{3}\rho^\al, \rho + \tfrac{4}{3}\rho^\al}, \vp_m + \frac{T}{5} \le \vp \le \vp_m + \frac{4T}{5}}, \quad \textrm{for} \; m = 0, 1, \ldots, 2n+2k-1.$$
Notice that
\begin{equation}
|u| = |br^{-n+2k}\mu_{n-2k}( r )\exp\pr{\phi}|\abs{e^{i(F(\vp) + n\vp)}- \frac{\mu_n}{br^{2k}\mu_{n-2k}}} = |u_2|\abs{\exp\pr{\phi}}\abs{e^{i(F(\vp) + n\vp)}- \frac{ r^{-2k}\mu_n}{b\mu_{n-2k}}}.
\label{s4uBd}
\end{equation}
We study the behaviour of $S(\vp) = F(\vp) + n\vp$.  On the segment $[\vp_m, \vp_{m+1}]$, by (\ref{FDef}), $S(\vp) = (2n+2k)\vp + \Phi(\vp)$.  Thus $S(\vp_m) = 2\pi m$ and $S(\vp_{m+1}) = 2\pi(m+1)$.  Moreover,
$$S^\pri(\vp) = 2n +2k + \Phi^\pri(\vp) = 2n +2k + f(\vp).$$
By (\ref{f1}) and the conditions on $n$ and $k$, it may be assumed that $S^\pri(\vp) > n > 0$.  That is, $S$ increases monotonically on $[\vp_m, \vp_{m+1}]$. Therefore, if
$$\vp_m + T/5 \le \vp \le \vp_m + 4T/5,$$
then
$$2\pi m + \frac{nT}{5} \le S(\vp) \le 2\pi(m+1) - \frac{nT}{5},$$
or
$$2\pi m + \frac{n\pi}{5(n+k)} \le S(\vp) \le 2\pi(m+1) - \frac{n\pi}{5(n+k)}.$$
Since $k = \bigO(n^{1/2})$, then for $\disp \vp \in \brac{\vp_m + \frac{T}{5}, \vp_m + \frac{4T}{5}}$ we may assume that 
\begin{equation}
2\pi m + \frac{\pi}{7} \le S(\vp) \le 2\pi(m+1) - \frac{\pi}{7}.
\label{Sbd}
\end{equation}
From Lemma \ref{imEst} and (\ref{Sbd}), it follows that $\disp \abs{e^{iS(\vp)} -\frac{ r^{-2k}\mu_n}{b\mu_{n-2k}}} \ge \frac{1}{2}\sin\pr{\frac{\pi}{7}}$.  Therefore, by (\ref{s4uBd}),
\begin{equation}
|u(r,\vp)| \ge\frac{1}{2} |u_2(r, \vp)|\abs{\exp\pr{\phi}}\sin\pr{\frac{\pi}{7}}, \quad (r,\vp) \in P_m, \;\; m = 0, 1, \ldots, 2n +2k-1.
\label{1CuN}
\end{equation}

Then
\begin{align*}
\LP u + \la u &= J_1 u + K_1 u_2 \exp\pr{\phi},
\end{align*}
where 
\begin{align} 
K_1 &= -\brac{\frac{8nk + 2(n+2k)\Phi^\pri + (\Phi^\pri)^2 - i\Phi^{\pri\pri}}{r^2}}  = \bigO\pr{\frac{r^{3\be_0/2-2}}{\pr{\log r}^2}}.
\label{K_1ord}
\end{align}

Let $\disp V = J_1 + K_1\frac{u_2 \exp\pr{\phi}}{u}$.  It follows from (\ref{J_1ordN}), (\ref{1CuN}) and (\ref{K_1ord}) that $|V| \le Cr^{3\be_0/2 - 2} = Cr^{-N}$.
This completes step 1C. \\ 
\\
\emph{Step 2: $r \in \brac{\rhoa{2}, \rhoa{3}}$}. The solution $u_2 = -br^{-n+2k}e^{iF(\vp)}\mu_{n-2k}( r )$ is rearranged to $u_3 = -br^{-n+2k}e^{i(n+2k)\vp}\mu_{n-2k}( r )$.

Choose a smooth cutoff function $\psi$ such that $\psi(r) = \left\{ \begin{array}{rl} 1 & r \le \rhoa{\tfrac{7}{3}} \\ 0 & r \ge \rhoa{\tfrac{8}{3}} \end{array} \right.$ and 
\begin{equation}
|\psi^{(j)}(r)| \le Cr^{-j\al} \quad j=0,1,2, \quad r \in\R^+.
\label{psiB}
\end{equation}
Set $u = -br^{-n+2k}\exp i\brac{\psi(r)\Phi(\vp) + (n+2k)\vp}\mu_{n-2k}( r )$.  Then
\begin{align*}
\LP u + \la u &= D_2 u,
\end{align*}
where
\begin{align}
D_2 &= -\tfrac{8nk}{r^2} - (\psi^\pri\Phi)^2 + i\Phi\pr{\frac{\psi^\pri}{r} + \psi^{\pri\pri}} -2\tfrac{(n+2k)}{r^2}\psi\Phi^\pri - \frac{(\psi\Phi^\pri)^2}{r^2} + i\frac{\psi\Phi^{\pri\pri}}{r^2} + \tfrac{\la}{\sqrt{(n-2k)^2 - \la r^2}} - 2i\psi^\pri \Phi\tfrac{\sqrt{(n-2k)^2 -\la r^2}}{r}
\nonumber \\
&= \bigO\pr{n \cdot k \cdot r^{-2}}.
\label{D_2ord}
\end{align}
Let $\disp V = D_2$ so that by (\ref{D_2ord}), $|V| \le Cr^{3\be_0/2 -2} = Cr^{-N}$.  
This completes step 2. \\ 
\\
\emph{Step 3: $r \in \brac{\rhoa{3}, \rhoa{4}}$}.  The solution $u_3 = -br^{-n+2k}e^{i(n+2k)\vp}\mu_{n-2k}( r )$ is rearranged to $u_4 = -b_1r^{-n-2k}e^{i(n+2k)\vp}\mu_{n+2k}( r )$.

Choose a smooth cutoff function $\psi$ such that $\psi(r) = \left\{ \begin{array}{rl} 1 & r \le \rhoa{\tfrac{10}{3}} \\ 0 & r \ge \rhoa{\tfrac{11}{3}} \end{array} \right.$ and $\psi$ satisfies condition (\ref{psiB}).  Let $d = (\rhoa{3})^{4k}\frac{\mu_{n-2k} (\rhoa{3})}{\mu_{n+2k}(\rhoa{3})}$, so that $g( r ) = dr^{-4k}\frac{\mu_{n+2k} ( r )}{\mu_{n-2k}( r )}$ satisfies $1 \ge \abs{g\pr{r}} \ge e^{-C}$ for all $r \in \brac{\rhoa{3}, \rhoa{4}}$.  Set $b_1 = bd$.  Let 
$$u = u_3\brac{\psi + (1-\psi)g} = \left\{ \begin{array}{rl} u_3 & r \le \rhoa{\tfrac{10}{3}} \\ u_4 & r \ge \rhoa{\tfrac{11}{3}} \end{array}\right..$$
Let $h(r)=\psi + (1-\psi)g$.  Since $\disp g^\pri( r ) = \brac{-\tfrac{4k}{r} - \tfrac{2\la k r}{(n+2k)(n-2k)} + \bigO\pr{\tfrac{kr^3}{n^4}}}g( r )$ and \\
$\disp g^{\pri\pri}( r ) = \brac{\tfrac{4k}{r^2} - \tfrac{2\la k}{(n+2k)(n-2k)} + \tfrac{16k^2}{r^2} + \tfrac{16\la k^2}{(n+2k)(n-2k)} + \bigO\pr{\tfrac{k^2r^2}{n^4}}}g( r )$, then for all $r \in \brac{\rhoa{3}, \rhoa{4}}$,
\begin{align}
\abs{h\pr{r}} &\ge e^{-\tilde{C}},
\label{hBd} \\
\abs{h^\pri\pr{r}} &\le C\frac{r^{\be_0/2-1}}{\log r},
\label{hpBd} \\
\abs{\LP h\pr{r}} &\le C\frac{r^{\be_0-2}}{\pr{\log r}^2}.
\label{LhBd}
\end{align}
Then
\begin{align*}
\LP u + \la u &= D_3 u,
\end{align*}
where
\begin{align}
D_3 &= -\tfrac{8nk}{r^2} + \tfrac{\la}{\sqrt{(n-2k)^2 - \la r^2}} - 2\tfrac{\sqrt{(n-2k)^2 - \la r^2}}{r}\frac{h^\pri}{h} + \frac{\LP h}{h} = \bigO\pr{n \cdot k \cdot r^{-2} }.
\label{D_3ord}
\end{align}
Let $\disp V = D_3$ so that by (\ref{D_3ord}), $|V| \le Cr^{3\be_0/2 -2} = Cr^{-N}$.
This completes step 3. \\ 
\\
\emph{Step 4: $r \in \brac{\rhoa{4}, \rhoa{6}}$}.  The solution $u_4 = -b_1r^{-n-2k}e^{i(n+2k)\vp}\mu_{n+2k}( r )$ is rearranged to $u_5 = ar^{-(n+k)}e^{-i(n+k)\vp}\mu_{n+k}( r )$.

Choose $a = b_1(\rhoa{5})^{-k}\frac{\mu_{n+2k}(\rhoa{5})}{\mu_{n+k}(\rhoa{5})}$ so that $|u_4(\rhoa{5}, \cdot)| = |u_5(\rhoa{5}, \cdot)|$.  Since $|u_5(r, \vp)/u_4(r, \vp)| = \abs{\pr{\frac{r}{\rhoa{5}}}^k\frac{\mu_{n+2k}(\rhoa{5})}{\mu_{n+k}(\rhoa{5})}\frac{\mu_{n+k}( r )}{\mu_{n+2k}( r )}}$, then by the assumptions on $k$ and $\rho$,
\begin{align}
|u_5(r, \vp)/u_4(r, \vp)| &\le e^{-C}, \quad r \in \brac{\rhoa{4}, \rhoa{\tfrac{14}{3}}},
\label{uBd3} \\
|u_5(r, \vp)/u_4(r, \vp)| &\ge  e^{C}, \quad r \in \brac{\rhoa{\tfrac{16}{3}}, \rhoa{6}}.
\label{uBd4}
\end{align}
Choose smooth cutoff functions $\psi_1$, $\psi_2$, $\psi_3$, $\psi_4$ such that $\psi_1(r) = \left\{ \begin{array}{rl} 1 & \mathrm{if} \quad r \le \rhoa{\tfrac{16}{3}} \\ 0 & \mathrm{if} \quad r \ge \rhoa{\tfrac{17}{3}} \end{array}\right.$, \\
$\psi_2(r) = \left\{ \begin{array}{rl} 0 & \mathrm{if} \quad r \le \rhoa{\tfrac{13}{3}} \\ 1 & \mathrm{if} \quad r \ge \rhoa{\tfrac{14}{3}} \end{array}\right.$, $\psi_3\pr{r} = \left\{ \begin{array}{rl} 1 & \mathrm{if} \quad r \le \rhoa{\tfrac{17}{3}} \\ 0 & \mathrm{if} \quad r \ge \rhoa{5.9} \end{array}\right.$ and $\psi_4\pr{r} = \left\{ \begin{array}{rl} 0 & \mathrm{if} \quad r \le \rhoa{4.1} \\ 1 & \mathrm{if} \quad r \ge \rhoa{\tfrac{13}{3}} \end{array}\right.$.  Moreover, we require that each cutoff function satisfy condition (\ref{1psi}). 

We set 
$$u = \psi_1u_4\exp\pr{\psi_4 \phi_{n+k, n+2k}} + \psi_2u_5\exp\pr{\psi_3 \phi_{n+k, n+2k}}.$$
For the rest of step 4, we will abbreviate $\phi_{n+k,n+2k}$ with $\phi$. \\
\\
\emph{Step 4A: $r \in \brac{\rhoa{4}, \rhoa{\tfrac{14}{3}}}$}.
Since $\psi_1 \equiv 1$ and $\psi_3 = \psi_4$ on the support of $\psi_2$, then
$$u = u_4 \exp\pr{\psi_4 \phi} + \psi_2 u_5 \exp\pr{\psi_3 \phi} = \brac{u_4 + \psi_2 u_5}\exp\pr{\psi_4 \phi}.$$
On this annulus, by (\ref{uBd3}),
\begin{equation}
\exp\pr{-\psi_4 \phi} |u| \ge |u_4| - |u_5| \ge (1 - e^{-C})|u_4| \ge e^{C^\pri}|u_5| > 0.
\label{uBd5}
\end{equation}
We have
\begin{align*}
\LP u + \la u&= \pr{D_4 + \tilde D_4}u_4\exp\pr{\psi_4 \phi},
\end{align*}
where
\begin{align}
D_4 &= \tfrac{\la}{\sqrt{(n+2k)^2 - \la r^2}} + \brac{\frac{\psi_2^\pri}{r} + \psi_2^{\pri\pri}- \tfrac{2\sqrt{(n+k)^2 - \la r^2}}{r}\psi_2^{\pri}+ \psi_2\tfrac{ \la}{\sqrt{(n+k)^2 - \la r^2}}}\frac{u_5}{u_4} \nonumber \\
&= \bigO\pr{n \cdot r^{\be_0/2 - 2}}
\label{D_4ord} \\
\tilde D_4 &= 2\brac{-\tfrac{\sqrt{(n+2k)^2 - \la r^2}}{r} + \pr{\psi_2^\pri - \psi_2\tfrac{\sqrt{(n+k)^2 - \la r^2}}{r}}\frac{u_5}{u_4}}\pr{\psi_4^\pri\phi + \psi_4 \phi^\pri} \nonumber \\
&+ \pr{1 + \psi_2 \frac{u_5}{u_4}}\brac{\frac{1}{r}\pr{\psi_4^\pri \phi + \psi_4 \phi^\pri} + \pr{\psi_4^\pri \phi + \psi_4 \phi^\pri}^2 + \pr{\psi_4^{\pri\pri} \phi + 2\psi_4^\pri \phi^\pri+ \psi_4 \phi^{\pri\pri}}} \nonumber \\
&= \bigO\pr{r^{3\be_0/2-2}}.
\label{tD_4ord}
\end{align}
Let $\disp V = \pr{D_4 + \tilde D_4}\frac{u_4 \exp\pr{\psi_4 \phi}}{u}$ so that by (\ref{uBd5}) - (\ref{tD_4ord}), $|V| \le Cr^{3\be_0/2 -2} = Cr^{-N}$.  
This completes step 4A. \\ 
\\
\emph{Step 4B: $r \in \brac{\rhoa{\tfrac{16}{3}}, \rhoa{6}}$}.

On this annulus, since $\psi_2 \equiv 1$ and $\psi_3 = \psi_4$ on the support of $\psi_1$, then
$$u = \psi_1 u_4\exp\pr{\psi_4 \phi} + u_5 \exp\pr{\psi_3 \phi} = \brac{\psi_1 u_4 + u_5}\exp\pr{\psi_3 \phi}.$$
By (\ref{uBd4}),
\begin{equation}
\exp\pr{-\psi_3 \phi} |u| \ge |u_5| - |u_4| \ge (1 -e^{-C})|u_5| \ge e^{C^\pri}|u_4| > 0.
\label{uBd6}
\end{equation}
We have
\begin{align*}
\LP u + \la u &= \pr{E_4 + \tilde E_4} u_5 \exp\pr{\psi_3 \phi},
\end{align*}
where
\begin{align}
E_4 &= \brac{\frac{\psi_1^\pri}{r} + \psi_1^{\pri\pri}- 2\psi_1^{\pri}\tfrac{\sqrt{(n+2k)^2 - \la r^2}}{r}+ \psi_1\tfrac{\la}{\sqrt{(n+2k)^2 - \la r^2}}}\frac{u_4}{u_5} + \tfrac{\la}{\sqrt{(n+k)^2 - \la r^2}} \nonumber \\
&= \bigO\pr{n \cdot r^{\be_0/2 - 2}},
\label{E_4ord} \\
\tilde E_4 &= 2\brac{\pr{\psi_1^\pri - \psi_1\tfrac{\sqrt{(n+2k)^2 - \la r^2}}{r}}\frac{u_4}{u_5} - \tfrac{\sqrt{(n+k)^2 - \la r^2}}{r}}\pr{\psi_3^\pri\phi + \psi_3 \phi^\pri} \nonumber \\
&+ \pr{\psi_1 \frac{u_4}{u_5} + 1}\brac{\frac{1}{r}\pr{\psi_3^\pri \phi + \psi_3 \phi^\pri} + \pr{\psi_3^\pri \phi + \psi_3 \phi^\pri}^2 + \pr{\psi_3^{\pri\pri} \phi + 2\psi_3^\pri \phi^\pri+ \psi_3 \phi^{\pri\pri}}} \nonumber \\
&= \bigO\pr{r^{3\be_0/2-2}}.
\label{tE_4ord}
\end{align}
Let $\disp V = \pr{E_4 + \tilde E_4}\frac{u_5 \exp\pr{\psi_3 \phi}}{u}$ so that by (\ref{uBd6}) - (\ref{tE_4ord}), $|V| \le Cr^{3\be_0/2 -2} = Cr^{-N}$. 
This completes step 4B. 
\\ 
\emph{Step 4C: $r \in \brac{\rhoa{\tfrac{14}{3}}, \rhoa{\tfrac{16}{3}}}$}.

On this annulus, since all cutoff functions are equivalent to zero, $u = \pr{u_4 + u_5}\exp\pr{\phi}$ and
\begin{align*}
\LP u + \la u &= J_4u,
\end{align*}
where 
\begin{align} 
J_4 &=  \frac{\la}{\sqrt{(n+2k)^2 - \la r^2}} +  \frac{\la}{\sqrt{(n+k)^2 - \la r^2}} + \frac{\phi^\pri}{r} + \pr{\phi^\pri}^2 + \phi^{\pri\pri} = \bigO\pr{r^{-\be_0}\log r} ,
\label{J_4ord} 
\end{align}
\\
As long as $\rho_0 >> 1$, then with $V = J_4$, $\abs{V} \le C r^{-N}$.
\\ 
We now prove the last statement of the lemma.  We first define a function $M(r)$ that will help estimate $m(r) = \max\set{|u(r,\vp)| : 0 \le \vp \le 2\pi}$.  Let
$$ M(r) = \left\{ \begin{array}{ll} 
r^{-n}\mu_{n}( r )\exp\pr{\psi_4 \phi_{n, n-2k}} & \rho \le r \le \rhoo \\
br^{-n+2k}\mu_{n-2k}( r )\exp\pr{\psi_3 \phi_{n, n-2k}} & \rhoo \le r \le \rhoa{2} \\
br^{-n+2k}\mu_{n-2k}( r ) & \rhoa{2} \le r \le \rhoa{3} \\
br^{-n+2k}\mu_{n-2k}( r )h( r ) & \rhoa{3} \le r  \le \rhoa{4} \\
b_1r^{-(n+2k)}\mu_{n+2k}( r )\exp\pr{\psi_4 \phi_{n+k, n+2k}} & \rhoa{4} \le r \le \rhoa{5} \\
ar^{-(n+k)}\mu_{n+k}( r )\exp\pr{\psi_3 \phi_{n+k, n+2k}} & \rhoa{5} \le r \le \rhoa{6} \end{array}\right..$$
Note that $M(r)$ is equal to the modulus of the functions $u_1, \ldots, u_5$ from which our solution $u(r, \vp)$ is constructed.  Also, $M(r)$ is a continuous, piecewise smooth function for which $m(r) \le 2M(r)$ and $M(\rho) = m(\rho)$.  Therefore,
$$\ln m(r) - \ln m(\rho) \le \ln 2 + \ln M(r) - \ln M(\rho).$$
Since $ \mu_n( r ) = \exp\pr{\frac{\la r^2}{4n} + \bigO\pr{\frac{r^4}{n^3}}}$, $\psi_j^\pri = \bigO\pr{r^{\be_0/2 - 1}}$, $\phi = \bigO\pr{\log r}$ and $\disp \phi^\pri\pr{r} = \bigO\pr{\frac{r}{n^2}}$, then everywhere on $\brac{\rho, \rho + 6\rho^{\al}}$, except at a finite number of points where $M(r)$ is not differentiable, we have
$$\frac{d}{dr}\ln M(r) = \frac{-n + \bigO(k)}{r} + \bigO\pr{\frac{r}{n}} + \bigO\pr{r^{\be_0/2 -1}\log r} \le -c \frac{r^{\be_0-1}}{\log r},$$
by the conditions on $n$, $k$.  Therefore,
$$ \ln m(r) - \ln m(\rho) \le \ln 2 + \int_\rho^r(\ln M(t))^\pri dt \le \ln 2 - c\int_\rho^r \frac{t^{\be_0-1}}{\log t}dt,$$
proving the lemma.
\end{pf}

\begin{rem}
If $\la = 0$, then $\mu_{\cdot} \equiv 1$ and $\phi_{\cdot, \cdot} \equiv 0$, so the above construction simplifies greatly.  If fact, if we take $n \sim \rho^{\be_0}$ and $k \sim \rho^{\be_0/2}$, then condition (\ref{Vbd}) is satisfied and we get (\ref{Pint}) as we did with the previous construction.
\end{rem}

We will now present the proof of Lemma \ref{meshP}, the slightly less-complicated construction.  Many of the steps in this proof are identical to those in the proof of Lemma \ref{meshN}, so we will often refer to them.

\begin{pf}[Proof of Lemma \ref{meshP}]
As $r$ increases from $\rho$ to $\rhoa{6}$, we rearrange equation (\ref{WPDE}) and its solution $u$ so that all of the above conditions are met.  This process is broken down into four major steps.

\nid\emph{Step 1: $r \in \brac{\rho, \rhoa{2}}$}.  During this step, the function $u_1 = r^{-n}e^{-in\vp}\mu_{n}( r )$ is rearranged to $u_2 = -br^{-n+2k}e^{iF(\vp)}\mu_{n-2k}( r )$, both of which satisfy an equation of the form (\ref{WPDE}), where $b$ and $F$ are as in the proof of Lemma \ref{meshN}.

Choose smooth cutoff functions $\psi_1$, $\psi_2$ as in step 1 of the proof of Lemma \ref{meshN}.
We set 
$$u = \psi_1u_1 + \psi_2u_2.$$

\nid \emph{Step 1A: $r \in \brac{\rho, \rhoa{\tfrac{2}{3}}}$}.
On this annulus, $\psi_1 \equiv 1$ and by (\ref{uBd1}),
\begin{equation}
|u| \ge |u_1| - |u_2| \ge (1 - e^{-C})|u_1| \ge e^{C^\pri}|u_2| > 0.
\label{1Au}
\end{equation}
Let $W = w(i\sin\vp, -i\cos\vp)$ for $w$ to be determined.  Then
\begin{align*}
W\cdot\gr u &= w d_1 u_1,\\
\LP u + \la u &= D_1u_1,
\end{align*}
where 
\begin{align} 
d_1 &= -\frac{n}{r} + \psi_2\frac{n + 2k + \Phi^\pri}{r}\frac{u_2}{u_1} = \bigO\pr{r^{\be_0-1}}, 
\label{d_1ord}
\end{align}
and $D_1$ is as in (\ref{D_1ord}).
Since $\disp \left| \frac{u_2}{u_1} \right| \le e^{-C}$ then $\disp d_1 \ne 0$ and $\disp |d_1| \ge Cr^{\be_0-1}$.  Let $\disp w = \frac{D_1}{d_1}$ so that by (\ref{d_1ord}) and (\ref{D_1ord}), $|W| \le Cr^{\be_0/2 -1} = Cr^{-P}.$ 
This completes step 1A. \\ 
\\
\emph{Step 1B: $r \in \brac{\rhoa{\tfrac{4}{3}}, \rhoa{2}}$}.
On this annulus, $\psi_2 \equiv 1$ and by (\ref{uBd2}),
\begin{equation}
|u| \ge |u_2| - |u_1| \ge (1 - e^{-C})|u_2| \ge e^{C^\pri}|u_1| > 0.
\label{1Bu}
\end{equation}
Let $W = w(i\sin\vp, -i\cos\vp)$ for $w$ to be determined.  Then
\begin{align*}
W\cdot \gr u &= we_1u_2, \\
\LP u + \la u &= E_1u_2,
\end{align*}
where 
\begin{align} 
e_1 &= -\psi_1\frac{n}{r}\frac{u_1}{u_2} + \frac{n + 2k + \Phi^\pri}{r} = \bigO\pr{r^{\be_0-1}}, 
\label{e_1ord}
\end{align}
and $E_1$ is as in (\ref{E_1ord}).
Since $\disp \abs{\frac{u_1}{u_2}} \le e^{-C}$ then $\disp e_1 \ne 0$ and $\disp |e_1| \ge Cr^{\be_0-1}$.  Let $\disp w = \frac{E_1}{e_1}$ so that by (\ref{e_1ord}) and (\ref{E_1ord}), $|W| \le Cr^{\be_0/2 -1} = Cr^{-P}.$
This completes step 1B. \\ 
\\ 
\emph{Step 1C: $r \in \brac{\rhoa{\tfrac{2}{3}}, \rhoa{\tfrac{4}{3}}}$.}

On this annulus, $\psi_1 \equiv 1 \equiv \psi_2$, so 
$$u = u_1 + u_2 = r^{-n}e^{-in\vp}\mu_n\pr{r} - br^{-n+2k}e^{iF(\vp)}\mu_{n-2k}( r ),$$
Let $W = w_1(\cos \vp, \sin \vp) + w_2(i\sin\vp, -i\cos\vp)$ for $w_1$, $w_2$ to be determined.  Then
\begin{align*}
W\cdot \gr u &= w_1j^1_1u_1 + w_2j^2_1u_1 + w_1\tilde j^1_1u_2  + w_2\tilde j^2_1u_2, \\
\LP u + \la u &= J_1u_1 + \tilde J_1 u_2,
\end{align*}
where 
\begin{align} 
j^1_1 &= -\tfrac{\sqrt{n^2 - \la r^2}}{r} = \bigO(r^{\be_0-1}), 
\label{j_11ord} \\
j^2_1 &=- \tfrac{n}{r} = \bigO(r^{\be_0-1}), 
\label{j_12ord} \\
\tilde j^1_1 &= -\tfrac{\sqrt{(n-2k)^2 - \la r^2}}{r} = \bigO(r^{\be_0-1}), 
\label{tj_11ord} \\
\tilde j^2_1 &= \tfrac{n + 2k + \Phi^\pri}{r} = \bigO(r^{\be_0-1}), 
\label{tj_12ord} \\
J_1 &=  \tfrac{\la}{\sqrt{n^2 - \la r^2}} = \bigO\pr{r^{-\be_0}} ,
\label{J_1ordP} \\
\tilde J_1 &= \tfrac{\la}{\sqrt{(n-2k)^2 - \la r^2}}  - \tfrac{8nk + 2(n+2k)\Phi^\pri + (\Phi^\pri)^2 - i\Phi^{\pri\pri}}{r^2} = \bigO\pr{r^{3\be_0/2-2}}.
\label{tJ_1ord}
\end{align}
\\
 If we let
 \begin{align*}
 w_1 &= \frac{\tilde j_1^2 J_1 - j_1^2 \tilde J_1}{j_1^1 \tilde j_1^2 - j_1^2 \tilde j_1^1} = \tfrac{r^2}{n\sqrt{(n-2k)^2 - \la r^2} + \pr{n+2k + \Phi^\pri}\sqrt{n^2 - \la r^2}}\pr{j_1^2 \tilde J_1 - \tilde j_1^2 J_1} = \bigO\pr{r^{\be_0/2 - 1}} \\
 w_2 &= \frac{j_1^1 \tilde J_1 - \tilde j_1^1 J_1}{j_1^1 \tilde j_1^2 - j_1^2 \tilde j_1^1} = \tfrac{r^2}{n\sqrt{(n-2k)^2 - \la r^2} + \pr{n+2k + \Phi^\pri}\sqrt{n^2 - \la r^2}}\pr{\tilde j_1^1 J_1 - j_1^1 \tilde J_1}= \bigO\pr{r^{\be_0/2 - 1}}, \\
 \end{align*}
then (\ref{WPDE}) is satisfied.   Since $\be_0/2 - 1 = \tfrac{1}{2}(2-2P) - 1 = - P$, then we see that $|W| \le Cr^{-P}$.  This completes step 1C. \\ 
\\
\emph{Step 2: $r \in \brac{\rhoa{2}, \rhoa{3}}$}. The solution $u_2 = -br^{-n+2k}e^{iF(\vp)}\mu_{n-2k}( r )$ is rearranged to $u_3 = -br^{-n+2k}e^{i(n+2k)\vp}\mu_{n-2k}( r )$ by setting $u = -br^{-n+2k}\exp i\brac{\psi(r)\Phi(\vp) + (n+2k)\vp}\mu_{n-2k}( r )$, as in the proof of Lemma \ref{meshN}.  Let $W = w(\cos\vp, \sin\vp)$ for $w$ to be determined.  Then
\begin{align*}
W\cdot\gr u &= wd_2u, \\
\LP u + \la u &= D_2 u,
\end{align*}
where
\begin{align}
d_2 &=-\frac{\sqrt{(n-2k)^2 - \la r^2}}{r} + i\psi^\pri \Phi= \bigO\pr{r^{\be_0-1}}
\label{d_2ord},
\end{align}
and $D_2$ is as in (\ref{D_2ord}).
By the conditions on $n$ and $k$, $d_2 \ne 0$ so $|d_2| \ge Cr^{\be_0-1}$.  Let $\disp w = \frac{D_2}{d_2}$ so that by (\ref{d_2ord}) and (\ref{D_2ord}), $|W| \le Cr^{\be_0/2 -1} = Cr^{-P}.$  
This completes step 2. \\ 
\\
\emph{Step 3: $r \in \brac{\rhoa{3}, \rhoa{4}}$}.  The solution $u_3 = -br^{-n+2k}e^{i(n+2k)\vp}\mu_{n-2k}( r )$ is rearranged to $u_4 = -b_1r^{-n-2k}e^{i(n+2k)\vp}\mu_{n+2k}( r )$ by setting $u = u_3\brac{\psi + (1-\psi)g} = u_3 h$, as in Lemma \ref{meshN}.
Let $W = w(\cos\vp, \sin\vp)$ for some $w$ to be determined.  Then
\begin{align*}
W\cdot \gr u & = wd_3u, \\
\LP u + \la u &= D_3 u,
\end{align*}
where
\begin{align}
d_3 &= -\frac{\sqrt{(n-2k)^2 - \la r^2}}{r} + \frac{h^\pri}{h} = \bigO\pr{r^{\be_0-1}},
\label{d_3ord}
\end{align}
and $D_3$ is as in (\ref{D_3ord}).  By (\ref{hpBd}) and the conditions on $n$ and $k$, $d_3 \ne 0$ so that $|d_3| \ge Cr^{\be_0-1}$.  Let $\disp w = \frac{D_3}{d_3}$ so that by (\ref{d_3ord}) and (\ref{D_3ord}), $|W| \le Cr^{\be_0/2 -1} = Cr^{-P}.$ 
This completes step 3. \\ 
\\
\emph{Step 4: $r \in \brac{\rhoa{4}, \rhoa{6}}$}.  The solution $u_4 = -b_1r^{-n-2k}e^{i(n+2k)\vp}\mu_{n+2k}( r )$ is rearranged to $u_5 = ar^{-(n+k)}e^{-i(n+k)\vp}\mu_{n+k}( r )$.
Choose smooth cutoff functions $\psi_1$ and $\psi_2$ as in Lemma \ref{meshN} and set
$$u = \psi_1u_4 + \psi_2u_5.$$
\emph{Step 4A: $r \in \brac{\rhoa{4}, \rhoa{\tfrac{14}{3}}}$}.

On this annulus, by (\ref{uBd3}),
\begin{equation}
|u| \ge |u_4| - |u_5| \ge (1 - e^{-C})|u_4| \ge e^{C^\pri}|u_5| > 0.
\label{uBd5P}
\end{equation}
Let $W = w(i\sin\vp, -i\cos\vp)$ for $w$ to be determined.  Since $\psi_1 \equiv 1$ in this annulus, we have
\begin{align*}
W\cdot \gr u &= wd_4u_4, \\
\LP u + \la u&= D_4 u_4,
\end{align*}
where
\begin{align}
d_4 &=  \frac{n+2k}{r}  - \psi_2\frac{n+k}{r}\frac{u_5}{u_4} = \bigO\pr{r^{\be_0-1}},
\label{d_4ord}
\end{align}
and $D_4$ is as in (\ref{D_4ord}).  Since $\disp \abs{\frac{u_5}{u_4}} \le e^{-C}$, then $|d_4| \ge Cr^{\be_0-1}$.  Let $\disp w = \frac{D_4}{d_4}$ so that by (\ref{d_4ord}) and (\ref{D_4ord}), $|W| \le Cr^{\be_0/2 -1} = Cr^{-P}.$
This completes step 4A. \\ 
\\
\emph{Step 4B: $r \in \brac{\rhoa{\tfrac{16}{3}}, \rhoa{6}}$}.

On this annulus, by (\ref{uBd4}),
\begin{equation}
|u| \ge |u_5| - |u_4| \ge (1 -e^{-C})|u_5| \ge e^{C^\pri}|u_4| > 0.
\label{uBd6P}
\end{equation}
Let $W = w(i\sin\vp, -i\cos\vp)$ for $w$ to be determined.  Since $\psi_2 \equiv 1$ in this annulus, we have
\begin{align*}
W\cdot \gr u &= we_4 u_5, \\
\LP u + \la u &= E_4 u_5,
\end{align*}
where
\begin{align}
e_4 &= \psi_1\frac{n+2k}{r}\frac{u_4}{u_5}  - \frac{n+k}{r} = \bigO\pr{r^{\be_0-1}},
\label{e_4ord}
\end{align}
and $E_4$ is as in (\ref{E_4ord}).  Since $\disp \abs{\frac{u_4}{u_5}} < e^{-C}$, then $|e_4| \ge Cr^{\be_0-1}$.  Let $\disp w = \frac{E_4}{e_4}$ so that by (\ref{e_4ord}) and (\ref{E_4ord}), $|W| \le Cr^{\be_0/2 -1} = Cr^{-P}.$
This completes step 4B. \\
\\ 
\emph{Step 4C: $r \in \brac{\rhoa{\tfrac{14}{3}}, \rhoa{\tfrac{16}{3}}}$}.

On this annulus, $u = u_4 + u_5$.
Let $W = w_1(\cos \vp, \sin \vp) + w_2(i\sin\vp, -i\cos\vp)$ for $w_1$, $w_2$ to be determined.  Then
\begin{align*}
W\cdot \gr u &= w_1j^1_4u_4 + w_2j^2_4u_4 + w_1\tilde j^1_4u_5  + w_2\tilde j^2_4u_5, \\
\LP u + \la u &= J_4u_4 + \tilde J_4 u_5,
\end{align*}
where 
\begin{align} 
j^1_4 &= -\tfrac{\sqrt{(n+2k)^2 - \la r^2}}{r} = \bigO(r^{\be_0-1}), 
\label{j_11ord} \\
j^2_4 &= \tfrac{n+2k}{r} = \bigO(r^{\be_0-1}), 
\label{j_12ord} \\
\tilde j^1_4 &= -\tfrac{\sqrt{(n+k)^2 - \la r^2}}{r} = \bigO(r^{\be_0-1}), 
\label{tj_11ord} \\
\tilde j^2_4 &= -\tfrac{n + k}{r} = \bigO(r^{\be_0-1}), 
\label{tj_12ord} \\
J_4 &=  \tfrac{\la}{\sqrt{(n+2k)^2 - \la r^2}} = \bigO\pr{r^{-\be_0}} ,
\label{J_1ordP} \\
\tilde J_4 &= \tfrac{\la}{\sqrt{(n+k)^2 - \la r^2}} = \bigO\pr{r^{-\be_0}}.
\label{tJ_1ord}
\end{align}
\\
 If we let
 \begin{align*}
 w_1 &= \frac{\tilde j_4^2 J_4 - j_4^2 \tilde J_4}{j_4^1 \tilde j_4^2 - j_4^2 \tilde j_4^1} = -\tfrac{\la r}{\sqrt{(n+k)^2 - \la r^2}\sqrt{(n+2k)^2 - \la r^2}}\pr{\tfrac{\pr{n+k}\sqrt{(n+k)^2 - \la r^2} + \pr{n+2k}\sqrt{(n+2k)^2 - \la r^2}}{\pr{n+k}\sqrt{(n+k)^2 - \la r^2} + \pr{n+2k}\sqrt{(n+2k)^2 - \la r^2}}} = \bigO\pr{r^{1 - 2\be_0}} \\
 w_2 &= \frac{j_4^1 \tilde J_4 - \tilde j_4^1 J_4}{j_4^1 \tilde j_4^2 - j_4^2 \tilde j_4^1} = - \tfrac{\la r k (2n+3)}{\sqrt{(n+k)^2 - \la r^2}\sqrt{(n+2k)^2 - \la r^2}\pr{\pr{n+k}\sqrt{(n+k)^2 - \la r^2} + \pr{n+2k}\sqrt{(n+2k)^2 - \la r^2}}} = \bigO\pr{r^{1 - 5\be_0/2}},
 \end{align*}
then (\ref{WPDE}) is satisfied.  Since $\be_0 > 1$, then $1 - 2\be_0 \le \be_0/2 - 1$ and $\abs{W} \le C r^{-P}$.  This completes step 4C. \\
\\ 
We now prove the last statement of the lemma.  We first define a function $M(r)$ that will help estimate $m(r) = \max\set{|u(r,\vp)| : 0 \le \vp \le 2\pi}$.  Let
$$ M(r) = \left\{ \begin{array}{ll} 
r^{-n}\mu_n( r ) & \rho \le r \le \rhoo \\
br^{-(n-2k)}\mu_{n-2k}( r ) & \rhoo \le r \le \rhoa{3} \\
br^{-n+2k}\mu_{n-2k}( r )h( r ) & \rhoa{3} \le r \le \rhoa{4} \\
b_1r^{-(n+2k)}\mu_{n+2k}( r ) & \rhoa{4} \le r \le \rhoa{5} \\
ar^{-(n+k)}\mu_{n+k}( r ) & \rhoa{5} \le r \le \rhoa{6} \end{array}\right..$$
Note that $M(r)$ is equal to the modulus of the functions $u_1, \ldots, u_5$ from which our solution $u(r, \vp)$ is constructed.  Also, $M(r)$ is a continuous, piecewise smooth function for which $m(r) \le 2M(r)$ and $M(\rho) = m(\rho)$.  Therefore,
$$\ln m(r) - \ln m(\rho) \le \ln 2 + \ln M(r) - \ln M(\rho).$$
Since $\mu_n( r ) = \exp\pr{\frac{\la r^2}{4n} + \bigO\pr{\frac{r^4}{n^3}}}$, then everywhere on $\brac{\rho, \rho + 6\rho^{\al}}$, except at a finite number of points where $M(r)$ is not differentiable, we have
$$\frac{d}{dr}\ln M(r) = \frac{-n + \bigO(k)}{r} + \bigO\pr{\frac{r}{n}} \le -c r^{\be_0-1},$$
by the conditions on $n$, $k$.  Therefore,
$$ \ln m(r) - \ln m(\rho) \le \ln 2 + \int_\rho^r(\ln M(t))^\pri dt \le \ln 2 - c\int_\rho^r t^{\be_0-1}dt,$$
proving the lemma.
\end{pf}

We now use the lemmas to construct examples and prove Proposition \ref{consP}.  

\begin{pf}[Proof of Proposition \ref{consP}]
We recursively define a sequence of numbers $\set{\rho_j}_{j=1}^\iny$.  For $\rho_1$, we choose a sufficiently large positive number.  Then if $\rho_j$ has been chosen, we set $\rho_{j+1} = \rho_j + 6\rho_j^{\al}$.  Suppose that $N$ and $P$ are chosen so that $\be_0 = \frac{4-2N}{3} > 1$.  We then let $n_j = \Big\lfloor \frac{\rho_j^{\be_0}}{\log \rho_j}\Big\rfloor = \frac{\rho_j^{\be_0}}{\log \rho_j} - \eps_j$ and $k_j = n_{j+1} - n_j$.  In order to use Lemma \ref{meshN}, we must estimate $k_j$:
\begin{align*}
k_j &= \tfrac{\rho_{j+1}^{\be_0}}{\log \rho_{j+1}} - \tfrac{\rho_j^{\be_0}}{\log \rho_j} + \eps_j - \eps_{j+1} \\
&= \tfrac{(\rho_j + 6\rho^{\al}_j)^{\be_0}}{\log \pr{\rho_j + 6\rho^{\al}_j}} - \tfrac{\rho_j^{\be_0}}{\log \rho_j} - \Delta_\eps \\
&= \tfrac{\rho_j^{\be_0}}{\log \rho_j}\brac{(1 + 6\rho_j^{\al-1})^{\be_0}\pr{1 + \tfrac{\log\pr{1 + 6\rho_j^{\al - 1}}}{\log \rho_j}}^{-1} - 1} - \Delta_\eps \\
&=  \tfrac{\rho_j^{\be_0}}{\log \rho_j}\brac{\pr{1 + 6\be_0\rho_j^{\al-1} + 18\be_0(\be_0-1)\rho_j^{2\al-2} + \bigO(\rho_j^{3\al-3})}\pr{1 - \tfrac{6\rho_j^{\al - 1}}{\log \rho_j} + \bigO\pr{\tfrac{\rho_j^{2\al - 2}}{\log \rho_j}}} - 1} - \Delta_\eps \\
&= \tfrac{\rho_j^{\be_0}}{\log \rho_j}\brac{6\pr{\be_0 - \tfrac{1}{\log \rho_j}}\rho_j^{\al-1} + 18\be_0(\be_0-1)\rho_j^{2\al-2} + \bigO\pr{\tfrac{\rho_j^{2\al - 2}}{\log \rho_j}}} + \bigO(1) \\
&= 6\pr{\be_0 - \tfrac{1}{\log \rho_j}}\tfrac{\rho_j^{\be_0/2}}{\log \rho_j} + \tfrac{18\be_0(\be_0-1)}{\log \rho_j} + \bigO\pr{\tfrac{1}{\pr{\log \rho_j}^2}} + \bigO(1) \\
&\le 6\pr{\be_0 - \tfrac{1}{\log \rho_j}}\tfrac{\rho_j^{\be_0/2}}{\log \rho_j} + 10,
\end{align*}
for a sufficiently large $\rho_1$.  Therefore, $ \abs{k_j - 6\pr{\be_0 - \frac{1}{\log \rho_j}}\frac{\rho_j^{\be_0/2}}{\log \rho_j} } \le 10$. \\
If $N$ and $P$ are chosen so that $\disp \be_0 = 2-2P > 1$, then we let $\disp n_j = \lfloor \rho_j^{\be_0} \rfloor = \rho_j^{\be_0} - \eps_j$ and $k_j = n_{j+1} - n_j$.  In order to use Lemma \ref{meshP}, we must estimate $k_j$:
\begin{align*}
k_j &= \rho_{j+1}^{\be_0} - \rho_j^{\be_0} + \eps_j - \eps_{j+1} \\
&= (\rho_j + 6\rho^{\al}_j)^{\be_0}- \rho_j^{\be_0} - \Delta_\eps \\
&= \rho_j^{\be_0}\brac{(1 + 6\rho_j^{\al-1})^{\be_0} - 1} - \Delta_\eps \\
&=  \rho_j^{\be_0}\brac{1 + 6\be_0\rho_j^{\al-1} + 18\be_0(\be_0-1)\rho_j^{2\al-2} + \bigO(\rho_j^{3\al-3}) - 1} - \Delta_\eps \\
&= 6\be_0\rho_j^{\be_0/2}+ 18\be_0(\be_0-1)+ \bigO\pr{\rho_j^{-\be_0/2}} \\
&\le 6\be_0\rho_j^{\be_0/2} + 40,
\end{align*}
for a sufficiently large $\rho_1$.  Therefore, $\disp \abs{k_j - 6\be_0\rho_j^{\be_0/2} } \le 40$.

For $j = 1, 2, \ldots$, we let $u_j$ denote the solutions of equations of the form (\ref{VPDE}) or (\ref{WPDE}), denoted by $L_j u_j = 0$.  By Lemma \ref{meshN} and Lemma \ref{meshP}, these equations and their solutions be constructed in the annulus $\set{\rho_j \le r \le \rho_{j+1}}$.  Note that  $\disp u_j(\rho_j, \vp) = \rho_j^{-n_j}e^{-in_j\vp}\mu_{n_j}\pr{\rho_j}$ and $\disp u_j(\rho_{j+1}, \vp) = a_j\rho_{j+1}^{-n_{j+1}}e^{-in_{j+1}\vp}\mu_{n_{j+1}}\pr{\rho_{j+1}}$.

Set $\rho_0 = 0$ and denote by $g_0(r)$ a smooth function in $\brac{0, \rho_1}$ such that $g_0(r) = r^{n_1}$ in a neighbourhood of $0$ while $g_0(r) = r^{-n_1}$ in a neighbourbood of the point $\rho_1$.  We suppose also that $g_0(r) > 0$ on $(0, \rho_1)$.  Let $u_0 = g_0(r)e^{-in_1\vp}\mu_{n_1}( r )$ and denote by $L_0u_0 = 0$ the equation of the form (\ref{VPDE}) or (\ref{WPDE}) which the function $u_0$ satisfies.

We define a differential operator $L$ in $\R^2$ by setting $L = L_j$ for $\rho_j \le r \le \rho_{j+1}$, $j = 0, 1, \ldots$.  We define a $C^2$ function $u$ on $\R^2$ by setting $u(r, \vp) = u_j(r, \vp)$ if $\rho_j \le r \le \rho_{j+1}$, $j = 0, 1$, and
$$u(r, \vp) = \pr{\prod_{i=1}^{j-1}a_i}u_j(r, \vp),$$
if $\rho_j \le r \le \rho_{j+1}$, $j = 2, 3,\ldots$.  Then it is clear that $u$ satisfies $Lu = 0$ in $\R^2$.   

We must now estimate $\abs{u}$.  Set $m(r) = \max\set{|u(r, \vp)| : 0 \le \vp \le 2\pi}$.  For a given $r \in \R^+$, we choose $\ell \in \Z$ so that $\rho_\ell \le r \le \rho_{\ell+1}$.  Then
\begin{align*}
\ln m(r) = &(\ln m(r) - \ln m(\rho_\ell)) + (\ln m(\rho_\ell) - \ln m(\rho_{\ell-1})) + \cdots + (\ln m(\rho_2) - \ln m(\rho_{1})) + \ln m(\rho_1)
\end{align*}
If $\be_0 = \frac{4-2N}{3}$, then by Lemma \ref{meshN}, for sufficiently large $r$, we have
$$\ln m(r) \le \ell\ln 2 - c\int_{\rho_1}^r \frac{t^{\be_0-1}}{\log t}dt + \ln m(\rho_1).$$
Since $\ell \lesssim r$ and $\be_0  > 1$, then we get
$\disp \ln m(r) \le C - c\frac{r^{\be_0}}{\log r}.$  Therefore,
$$m(r) \le C\exp\pr{-c\frac{r^{\be_0}}{\log r}}$$
and (\ref{Vubd}) holds.
If $\be_0 = 2-2P$, then by Lemma \ref{meshP}, for sufficiently large $r$, we have
$$\ln m(r) \le \ell\ln 2 - c\int_{\rho_1}^r t^{\be_0-1}dt + \ln m(\rho_1).$$
Since $\ell \lesssim r$ and $\be_0  > 1$, then we get $\disp \ln m(r) \le C - cr^{\be_0}.$  Thus,
$$m(r) \le C\exp\pr{-cr^{\be_0}},$$
so (\ref{Wubd}) holds.  This completes the constructions of the examples for Theorem \ref{cons}(\ref{consa}).
\end{pf}

%
%
\subsection{Proof of Theorem \ref{cons}(\ref{consb})}

The following lemma is based on the observation that $u_1\pr{r} = \exp\pr{\sgn\pr{\arg \la}\sqrt{-\la} r}$ satisfies an equation of the form (\ref{epde}) with either $V = C r^{-1}$ and $W \equiv 0$, or $V \equiv 0$ and $W = C r^{-1}$.  And whenever $\la \in \C$ with $\arg \la \in [-\pi, \pi) \setminus \set{0}$, $\abs{u_1\pr{r}} \lesssim \exp\pr{-C r}$.  By adding lower order terms to the exponent of $u_1$, we can make the potentials decay faster and faster. This lemma serves as the main building block in the constructions that prove Theorem \ref{cons}(\ref{consb})

\begin{lem}
For any $m \in \N$, there exists a function $u_m$ such that
\begin{equation}
\LP u_m + \la u_m = \brac{\frac{d_m}{r^m} + \frac{d_{m+1}}{r^{m+1}} + \ldots + \frac{d_{2(m-1)}}{r^{2(m-1)}}}u_m,
\label{mCons}
\end{equation}
for some constants $d_m, d_{m+1}, \ldots d_{2(m-1)} \in \C$.  Furthermore, if $\la \in \C \setminus \R_{\ge 0}$ then there exist positive constants $C_m$ and $R_m$ such that
\begin{align}
&\abs{u_m(x)} \le \exp\pr{-C_m |x|} ,
\label{decay} \\
&\abs{\frac{\del_r u_m\pr{r}}{u_m\pr{r}}} \ge C_m,
\label{decayDer}
\end{align}
for all $|x| \ge R_m$.
\label{consbLem}
\end{lem}

We will prove Lemma \ref{consbLem} by induction.  All of our {$u_m$} functions depend only on the radius.

\begin{proof}
For $\la \in \C$, we write $\la = \abs{\la}e^{i \arg \la}$, where $\arg \la \in [-\pi, \pi)$.  By our restriction, $\arg \la \ne 0$. \\
\emph{Base cases:} Let $u_1\pr{r} = \exp\pr{f_1( r )}$, where $f_1( r ) = \sgn\pr{\arg \la}\sqrt{-\la}r$.  Since $\Re \pr{\sgn\pr{\arg \la}\sqrt{-\la}} < 0$, then conditions (\ref{decay})-(\ref{decayDer}) are satisfied with $R_1$ equal to any positive number.  Also,
\begin{align*}
\frac{\LP u_1}{u_1} + \la &= \frac{n-1}{r}f_1^\prime + f_1^{\pri\pri} + \pr{f_1^\pri}^2 + \la \\
&=  \frac{(n-1)\sgn\pr{\arg \la}\sqrt{-\la}}{r},
\end{align*}
so condition (\ref{mCons}) is satisfied.
This completes the case $m = 1$. \\
Let $u_2\pr{r} = \exp\pr{f_1\pr{r} - \frac{n-1}{2}\log r}$.  Since $f_1\pr{r} = \sgn\pr{\arg \la}\sqrt{-\la}r$, then it is possible to choose $R_2$ and $C_2$ so that (\ref{decay}) and (\ref{decayDer}) hold.  We see that
\begin{align*}
\frac{\LP u_2}{u_2} + \la &= \frac{n-1}{r}\pr{\sgn\pr{\arg \la}\sqrt{-\la} - \frac{n-1}{2r}} + \frac{n-1}{2 r^2} + \pr{\sgn\pr{\arg \la}\sqrt{-\la} - \frac{n-1}{2r}}^2 + \la \\
&= - \frac{(n-1)(n-3)}{4r^2},
\end{align*}
which completes the case $m = 2$. \\
\emph{Inductive hypothesis:} For any $m \ge 2$, we will assume that there exists a function $u_m\pr{r} = \exp\pr{f_m\pr{r}}$ such that (\ref{mCons}) and (\ref{decay}) are satisfied.  That is,
$$\frac{n-1}{r}f_m^\prime + f_m^{\prime\prime} + \pr{f_m^\prime}^2 + \la = \frac{d_m}{r^m} + \frac{d_{m+1}}{r^{m+1}} + \ldots + \frac{d_{2(m-1)}}{r^{2(m-1)}},$$
where $f_m$ takes the form $f_m\pr{r} = c_1 r + c_2 \log r + c_3 r^{-1} + \ldots c_m r^{2-m}$.  If $\Re \pr{c_1} < 0$, then there exist constants $C_m$ and $R_m$ so that (\ref{decay})-(\ref{decayDer}) holds.  \\
Let $u_{m+1}\pr{r} = \exp\pr{f_m\pr{r} + c_{m+1}r^{1-m}}$ for a constant $c_{m+1}$ to be determined.  Then
\begin{align*}
\frac{\LP u_{m+1}}{u_{m+1}} + \la 
&= \frac{n-1}{r}\brac{f_{m}^\prime - \pr{m-1}c_{m+1}r^{-m} } + f_{m}^{\pri\pri} + m\pr{m-1}c_{m+1}r^{-(m+1)} + \brac{f_{m}^\prime - \pr{m-1}c_{m+1}r^{-m}}^2 + \la \\
&= \frac{n-1}{r}f_m^\prime + f_m^{\prime\prime} + \pr{f_m^\prime}^2 + \la \\
&+\frac{\brac{m - \pr{n-1}}\pr{m-1}c_{m+1}}{r^{m+1}}  - 2\pr{m-1}c_{m+1}\brac{\frac{c_1}{r^m} + \frac{c_2}{r^{m+1}} + \ldots - \frac{\pr{m-2}c_m}{r^{2m-1}}} + \frac{\brac{\pr{m-1}c_{m+1}}^2}{r^{2m}} \\
&= \frac{d_m - 2(m-1)c_1c_{m+1}}{r^m} + \frac{d_{m+1} + \brac{m - (n-1) -2c_2}(m-1)c_{m+1}}{r^{m+1}} + \ldots + \frac{\brac{\pr{m-1}c_{m+1}}^2}{r^{2m}}.
\end{align*}
If we let $c_{m+1} = \frac{d_m}{2(m-1)c_1}$, then condition (\ref{mCons}) is satisfied.  Furthermore, we may find $R_{m+1} \ge R_m$, $C_{m+1} > 0$ so that our decay requirements are fulfilled.
\end{proof}

With Lemma \ref{consbLem}, we are able to prove the following proposition which immediately implies Theorem \ref{cons}(\ref{consb})

\begin{prop}  
For any $\la \in \C\setminus \R_{\ge 0}$, we have the following.
\begin{enumerate}[(a)]
\item If $\disp \be_0 = \frac{4-2N}{3} < 1$, then there exists a potential $V$ and a { radial} eigenfunction $u$ such that (\ref{VPDE}) and (\ref{Vbd}) hold.
\item  If $\be_0 = 2-2P < 1$, then there exists a potential $W$ and a { radial} eigenfunction $u$ such that (\ref{WPDE}) and (\ref{Wbd}) hold.
\end{enumerate}
In both cases,
\begin{equation}
|u(x)| \le C\exp\pr{-c |x|}.
\label{Wubd1}
\end{equation}
\label{consD}
\end{prop}

\begin{proof}
If $ \be_0 = \frac{4-2N}{3} < 1$, let $m = \lceil N \rceil$. Otherwise, let $m = \lceil P \rceil $.  Since $\be_0 < 1$, m$ \ge 1$.  Let $R_m$ and $u_m$ be as in Lemma \ref{consbLem}.  Without loss of generality, we may assume that $R_m \ge 1$.  Set
\begin{align*}
u\pr{r} &= 
\left\{\begin{array}{ll} 
u_m\pr{r} & \textrm{if}\; r \ge R_m \\ 
u_m\pr{R_m} & \textrm{if} \; r \le R_ m  \; \textrm{and} \; \be_0 = \frac{4-2N}{3} \\ 
C \exp\pr{-\frac{\la}{2n}r^2} & \textrm{if} \; r \le R_ m  \; \textrm{and} \; \be_0 = 2-2P
\end{array} \right., \\
V\pr{r} &= 
\left\{\begin{array}{ll} 
d_mr^{-m} + d_{m+1}r^{-(m+1)} + \ldots + d_{2(m-1)}r^{-2(m-1)} & \textrm{if}\; r \ge R_m \\ 
\la & \textrm{if} \; r \le R_ m  
\end{array} \right. \\
W\pr{r, {\vp}} &= 
\left\{\begin{array}{ll} 
\brac{\frac{d_m}{c_1 r^{m} + \ldots - (m-2)c_m r} + \ldots + \frac{d_{2(m-1)}}{c_1 r^{2(m-1)}+ \ldots - (m-2)c_m r^{m-1}}}\pr{\cos\vp, \sin \vp} & \textrm{if}\; r \ge R_m \\ 
-\frac{\la r}{n}\pr{\cos\vp, \sin\vp} & \textrm{if} \; r \le R_ m  
\end{array} \right.,
\end{align*}
where $C$ is chosen so that $u_m(R_m) = C\exp\pr{-\frac{\la}{2n} R_m^2}$. By (\ref{decay}) in Lemma \ref{consbLem}, $\abs{u(x)} \lesssim \exp\pr{- c |x|}$. \\
If $ \be_0 = \frac{4-2N}{3}$, then $\LP u + \la u = V u$, where $\abs{V(x)} \lesssim \langle x \rangle^{-m} \le \langle x \rangle^{-N}$.  \\
By (\ref{decayDer}), $W$ is well-defined.  If $\be_0 = 2 - 2P$, then $\LP u + \la u = W \cdot \gr u$, where $\abs{W(x)} \lesssim \langle x \rangle ^{-m} \le \langle x \rangle^{-P}$.
\end{proof}

\appendix
\newcommand{\appsection}[1]{\let\oldthesection\thesection
  \renewcommand{\thesection}{\oldthesection}
  \section{#1}\let\thesection\oldthesection}
  
  \appsection{}
\label{AppA}

This Appendix presents some technical lemmas that are used in the proofs of Propositions \ref{baseC} and \ref{IH}. In the proof of Proposition \ref{IH}, we use the following fact about the ratio of the weight functions. 

\begin{lem}
Suppose $T \ge T_*$ and $S = T^{\ga} - T >> T$.  Then there exists a constant $c_n$, depending on $T_*$ and the dimension $n$, such that
$$\log \brac{ \frac{w\pr{1 + \frac{T}{2S}}}{w\pr{1 + \frac{1}{S}}} }  \ge \frac{c_n T}{S}.$$
\label{loglBd}
\end{lem}

\begin{proof}
Since $\disp w\pr{r} = \bar r = \int_0^r e^{-\nu s^2} ds$, then 
$$w\pr{1 + \frac{T}{2S}} = \int_0^{1 + \frac{T}{2S}} e^{-\nu s^2} ds =  \int_0^{1 + \frac{1}{S}} e^{-\nu s^2} ds + \int_{1 + \frac{1}{S}}^{1 + \frac{T}{2S}} e^{-\nu s^2} ds. $$
Using a Taylor expansion, it can be shown that for some $\mu > 0$
$$\frac{T}{2 \, S} \ge \int_{1 + \frac{1}{S}}^{1 + \frac{T}{2S}} e^{-\nu s^2} ds \ge \frac{T}{(2+\mu) \, S}. $$
Since $S >> 1$, then 
$$ c_n \ge  \int_0^{1 + \frac{1}{S}} e^{-\nu s^2} ds \ge \tilde c_n.$$
It follows that 
$$1 >> \frac{T}{2\tilde c_n \, S} \ge \frac{ \int_{1 + \frac{1}{S}}^{1 + \frac{T}{2S}} e^{-\nu s^2} ds }{  \int_0^{1 + \frac{1}{S}} e^{-\nu s^2} ds } \ge \frac{T}{(2+\mu)c_n  \, S}. $$
Therefore,
\begin{align*}
\log \brac{ \frac{w\pr{1 + \frac{T}{2S}}}{w\pr{1 + \frac{1}{S}}} } &=  \frac{ \int_{1 + \frac{1}{S}}^{1 + \frac{T}{2S}} e^{-\nu s^2} ds }{  \int_0^{1 + \frac{1}{S}} e^{-\nu s^2} ds } - \frac{1}{2}\pr{ \frac{ \int_{1 + \frac{1}{S}}^{1 + \frac{T}{2S}} e^{-\nu s^2} ds }{  \int_0^{1 + \frac{1}{S}} e^{-\nu s^2} ds }}^2 + \ldots \\
&\ge \frac{1}{2} \frac{ \int_{1 + \frac{1}{S}}^{1 + \frac{T}{2S}} e^{-\nu s^2} ds }{  \int_0^{1 + \frac{1}{S}} e^{-\nu s^2} ds } \\
& \ge  \frac{c_n T}{S}.
\end{align*}

\end{proof}

Finally, we present a standard result in elliptic theory that is used in the proof of our two main propositions.

\begin{lem}[Caccioppoli's Inequality]
Suppose $u$ satisfies $\LP u + W \cdot \gr u + V u = 0$ in $\R^n$, where $|V| \le M$, $|W| \le N$.  Then
$$\int_{B_r} |\gr u|^2 \lesssim \pr{\frac{1}{r^2} + M + N^2} \int_{B_{2r}} |v|^2.$$
\label{Cacc}
\end{lem}

\begin{proof}
Let $\eta \in C^{\iny}_0(B_{2r})$ be a smooth cutoff function such that $0 \le \eta \le 1$ and $\eta \equiv 1$ in $B_r$.  Then $\vp = \eta^2 u$ is a test function.  Therefore,
$$\int 2\eta u \gr \eta \cdot \gr u + \eta^2 |\gr u|^2 = \int V \abs{u}^2 \eta^2 + W\cdot\gr u \eta^2 u.$$
Rearranging and using Cauchy-Schwarz, we get
$$\int \eta^2 |\gr u|^2 \lesssim \pr{||V||_\iny + ||W||_\iny^2 + |\gr \eta|^2 } \int_{B_{2r}}\abs{u}^2,$$
which gives the claimed fact.
\end{proof}

%
%
  \appsection{}
\label{AppB}

In this Appendix, we present a number of results related to the sequences that are used in the proof of Theorem \ref{MEst}.  The following lemma gives a condition for when the first choice is used in the definitions of $\be_{j+1}$ and $\ga_{j+1}$.

\begin{lem}
If $T_j >> 1$ and $\disp 2 - 2P \ge \frac{4-2N}{3}$, then $2-2P \ge h_j$.
\label{case1}
\end{lem}

\begin{proof}
If $\disp 2 - 2P \ge \frac{4-2N}{3}$, then $1 - 3P + N \ge 0$. We see that
\begin{align*}
&2 - 2P \ge h_j \\
\Iff\; &2 - 2P \ge 1 + P - N + \om_j - \de_j \\
\Iff\; &1 + \de_j\ge 3P - N+ \om_j \\
\Iff\; &T_j^{1 + \de_j} \ge T_j^{3P - N+ \om_j} \\
\Iff\; &T_j (C\log T_j) \ge T_j^{3P - N} + T_j \\
\Iff\; & C\log T_j \ge T_j^{-(1-3P+N)} + 1,
\end{align*}
which holds for $T_j$ sufficiently large, giving the result.
\end{proof}

The following lemma provides a bound on the exponent in terms of a simplified version of the exponent.

\begin{lem}
Let $\hat\be_j$ and $\hg_j$ be as defined in \S \ref{Proof1a}.  If $\de_1 \le 1$ and  $\be_k \ge h_k$ for $k = 1, \ldots, j+1$, then
\begin{equation*}
\be_{j+1} \le \hat\be_{j+1} + \frac{(2P)^j}{(\hg_1\ldots \hg_j)^2}\de_1 + \frac{(2P)^{j-1}}{(\hg_2\ldots \hg_{j})^2}\de_2  + \ldots + \frac{2P}{(\hg_j)^2}\de_j.
\end{equation*}
\label{beHatEstP}
\end{lem}

\begin{proof}
The condition on the smallness of $\de_1$ ensures that each $\de_k$ is small and therefore we can do a Taylor expansion.  The condition that $\be_k \ge h_k$ for $k = 1, \ldots, j+1$ ensures that the first definition is always used to define each $\be_k$.  We will use each estimate for $\be_k$ to estimate $\be_{k+1}$.
\begin{align*}
{ \be}_1 &= \hat{\be}_1, \\
\Rightarrow {\be}_2 &= 2 - \frac{2P}{\hat{\ga}_1 + \de_1} 
\le 2 - \frac{2P}{\hat{\ga}_1} + \frac{2P}{(\hg_1)^2}  \de_1
= \hat \be_2 + \frac{2P}{(\hg_1)^2}\de_1, \\
\Rightarrow \be_3 &\le 2 - \frac{2P}{\hat{\ga}_2 + \frac{2P}{\hg_1^2}\de_1 + \de_2} 
\le 2 - \frac{2P}{\hat{\ga}_2} + \frac{(2P)^2}{(\hg_1 \hg_2)^2}\de_1  + \frac{2P}{(\hg_2)^2}\de_2 
= \hat\be_3 + \frac{(2P)^2}{(\hg_1 \hg_2)^2}\de_1  + \frac{2P}{(\hg_2)^2}\de_2, \\
 &\vdots  \\
\Rightarrow\be_{j+1} &\le \hat\be_{j+1} + \frac{(2P)^j}{(\hg_1\ldots \hg_j)^2}\de_1 + \frac{(2P)^{j-1}}{(\hg_2\ldots \hg_{j})^2}\de_2  + \ldots + \frac{2P}{(\hg_j)^2}\de_j ,
\end{align*}
as required.
\end{proof}

And now we present another bound on the exponent.

\begin{lem}
If $\be_k \ge h_k$ for $k = 1, \ldots, j$, then for $T_1 >> 1$,
$$\be_{j+1} \le \hat\be_{j+1} + f(T_1, P)\cdot j,$$
where $f$ is independent of $P$ for $2P \ge 1$, $\disp \frac{\del f}{\del T_1} < 0$ and $\hat\be_{j+1}$ is given by (\ref{PbeHatDef}).
\label{beBoundP}
\end{lem}

\begin{proof}
If $2P < 1$, then we are in the same situation as Case 1 from \S \ref{Proof1a} and by Corollary \ref{gambeCorP}, $\be_{j+1} \le \hat\be_{j+1} + C_P\de_{j+1}$, giving the desired result.  So we will consider the case where $2P \ge 1$.
Since $\be_k \ge h_k$, then the first choice is used to define our sequence and $\disp \hg_j = \frac{\hG_j}{\hG_{j-1}} = \frac{1 + 2P + \ldots + (2P)^j}{1 + 2P + \ldots + (2P)^{j-1}}$. 
By Lemma \ref{beHatEstP},
\begin{align*}
\be_{j+1} &\le \hat\be_{j+1} + \frac{(2P)^j}{(\hg_1\ldots \hg_j)^2}\de_1 + \frac{(2P)^{j-1}}{(\hg_2\ldots \hg_{j})^2}\de_2  + \ldots + \frac{2P}{(\hg_j)^2}\de_j \\
&= \hat\be_{j+1} + \frac{(2P)^j}{(\hG_j)^2}\de_1 + \frac{(2P)^{j-1}}{(\hG_j)^2}\pr{\hG_1}^2\de_2  + \ldots + \frac{2P}{(\hG_j)^2}\pr{\hG_{j-1}}^2\de_j.
\end{align*}
Since $T_{k+1} = T_1^{\Ga_k} > T_1^{\hG_k}$, then by (\ref{epsEst}),
$$\de_{k+1} \le \frac{2}{\sqrt{3}}\frac{\log\log T_{k+1}}{\log T_{k+1}} \le \frac{2}{\sqrt{3}}\frac{\log\log T_1 + \log\hG_k}{\hG_k \log T_1} \le \frac{2}{\sqrt{3}}\pr{\frac{2}{\sqrt{3}}\frac{\de_1}{\hG_k} + \frac{\log\hG_k}{\hG_k \log T_1}} = \frac{4}{{3}}\frac{\de_1}{\hG_k} + \frac{2}{\sqrt{3}}\frac{\log\hG_k}{\hG_k \log T_1}.$$
Thus,
\begin{align*}
\be_{j+1} &\le \hat\be_{j+1} + \frac{(2P)^j}{(\hG_j)^2}\de_1 + \frac{(2P)^{j-1}}{(\hG_j)^2}\pr{\hG_1}^2\pr{\frac{4}{3}\frac{\de_1}{\hG_1} + \frac{2}{\sqrt{3}}\frac{\log\hG_1}{\hG_1 \log T_1}}  + \ldots + \frac{2P}{(\hG_j)^2}\pr{\hG_{j-1}}^2\pr{\frac{4}{3}\frac{\de_1}{\hG_{j-1}} + \frac{2}{\sqrt{3}}\frac{\log\hG_{j-1}}{\hG_{j-1} \log T_1}} \\
&< \hat\be_{j+1} + \frac{4}{3}\frac{\de_1}{\pr{\hG_j}^2}\brac{(2P)^j + (2P)^{j-1}\hG_1 + \ldots + 2P\hG_{j-1}} + \frac{2}{\sqrt{3}}\frac{1}{\pr{\hG_j}^2 \log T_1}\brac{(2P)^{j-1}\hG_1\log\hG_1  + \ldots + 2P\hG_{j-1}\log\hG_{j-1}} \\
&< \hat\be_{j+1} + \frac{4}{3}\frac{\de_1}{\pr{\hG_j}^2}\brac{(2P)^j + (2P)^{j-1}\hG_1 + \ldots + 2P\hG_{j-1}}
+ \frac{2}{\sqrt{3}}\frac{1}{\pr{\hG_j}^2 \log T_1}\brac{(2P)^{j-1}\hG_1\log\hG_1  + \ldots + 2P\hG_{j-1}(j-1)\log\hG_1} \\
&< \hat\be_{j+1} + \frac{4}{3}\frac{\de_1}{\pr{\hG_j}^2}j\sum_{k=1}^j (2P)^k + \frac{1}{\sqrt{3}}\frac{\log\hG_1}{\pr{\hG_j}^2 \log T_1}j^2\sum_{k=1}^j(2P)^k \\
&\le \hat\be_{j+1} + \frac{4}{3}\de_1 + j\frac{\log\hG_1}{\sqrt{3}\log T_1},
\end{align*}
since $\hG_j \ge j$ for $2P \ge 1$.  Since $\de_1$ and $\frac{\hG_1}{\log T_1}$ decrease with $T_1$ and there is no dependence on $P$, we get the result.
\end{proof}

Using the above bound on the exponent, we may show that, in certain cases, the exponent will reach a specific lower bound.

\begin{lem}
If $T_1 >> 1$, $P > N$ whenever $2P \ge 1$, and $3P - N = 1 + \De$ for some $\De > 0$, then there exists a $J $ such that $\be_{J-1} = h_{J-1}$.  Furthermore, $J \le \mathcal{J}\pr{N,P, { T_1}}$.
\label{JLem}
\end{lem}

\begin{proof}
If $\be_1 < h_1$, replace $\be_1$ with $h_1$ and we are done.  Otherwise, there exists a $k \ge 2$ such that $\be_j \ge h_j$ for $j = 1, \ldots, k$.  By Lemma \ref{beBoundP}, $\be_{k+1} \le \hat\be_{k+1} + f(T_1, P)\cdot k$, where
$$\hat\be_{k+1} = 2 - \frac{2P}{\hg_k} = 2 - \frac{2P\pr{1 + \ldots + (2P)^{k-1}}}{1 + \ldots + (2P)^k} = 1 + \frac{1}{1 + \ldots + (2P)^k}$$
We will show that for $k >> 1$ and $T_1 >> 1$, we can ensure that $\be_{k+1} \le h_{k+1}$. \\
If $2P < 1$, then as shown in the proof of Lemma \ref{beBoundP}, $\be_{k+1} \le \hat\be_{k+1} + C_p\de_{k+1}$ and by (\ref{hatGamP}) from Case 1 in \S \ref{Proof1a}, $\hat\be_{k+1} \le 2 - \frac{2P}{1 + (2P)^k} \le 2 - 2P + (2P)^{k+1}$.  Thus,
\begin{align*}
& \be_{k+1} \le h_{k+1} \\
\Leftarrow\; & \hat\be_{k+1} + C_p\de_{k+1} \le h_{k+1} \\
\Leftarrow\; & 2 - 2P + (2P)^{k+1} + C_p\de_{k+1} \le 1 + P - N + \om_{k+1} - \de_{k+1} \\
\Iff \; & (2P)^{k+1} + \pr{1+C_p}\de_{k+1} \le \De + \om_{k+1}
\end{align*}
If we choose $k$, $T_1$ so that $(2P)^{k+1} \le \De/2$ and $\disp \de_1 \le \frac{\De}{2(1 + C_P)}$, then it follows that $\be_{k+1} \le h_{k+1}$.  Therefore, there exists a $J-1 \le k+1$ such that $\be_{J-1} \le h_{J-1}$ but $\be_{J-2} > h_{J-2}$.  As per the rules of our construction, we set $\be_{J-1} = h_{J-1}$. \\
Now we consider the case where $2P \ge 1$.
\begin{align*}
& \be_{k+1} \le h_{k+1} \\
\Leftarrow\; & 1 + \frac{1}{1 + \ldots + (2P)^k} + f(T_1)\cdot k \le 1 + P - N + \om_{k+1} - \de_{k+1} \\
\Iff\; & \frac{1}{1 + \ldots + (2P)^k} + f(T_1)\cdot k  + \de_{k+1}  \le  P - N + \om_{k+1}
\end{align*}
First we will choose $k$ so that $\disp \frac{1}{1 + \ldots + (2P)^k}  \le \frac{P-N}{2}$.  Then we will choose $T_1 >> 1$ so that $f(T_1) \le \frac{P-N}{4k}$ and $\de_1 \le \frac{P-N}{4}$ to get our result.  Again, there exists a $J-1 \le k+1$ such that $\be_{J-1} = h_{J-1}$.
\end{proof}

Once the lower specific lower bound has been reached, the exponent will continue to decrease in a controlled way.

\begin{lem}
Suppose $P > N$.  For $T_1 >> 1$, if $\be_{j-1} = h_{j-1}$, then $\be_j \le \ell_j$.
\label{utol}
\end{lem}

\begin{proof}
If $\be_{j-1} = h_{j-1}$, then $\ga_{j-1} = 3P - N + \om_{j-1}$ and $\be_j = 2 - \frac{2P}{3P - N + \om_{j-1}}$.  
\begin{align*}
& \be_ j \le \ell_j \\
\Iff\;& 2 - \frac{2P}{3P - N + \om_{j-1}} \le 1 + P - N + \frac{\om_j}{3} - \de_j \\
\Leftarrow\;& \pr{1 - P + N + \de_j}\pr{3P - N + \om_{j-1}} \le 2P \\
\Iff\;& \de_j\pr{3P - N} + \om_{j-1}\pr{1-P + N} + \de_j\om_{j-1} \le 2P - \pr{1 - P + N}\pr{3P - N} \\
\Iff\;&  \de_j\pr{3P - N} + \om_{j-1}\pr{1-P + N} + \de_j\om_{j-1} \le \De\pr{P - N}
\end{align*}
If we choose $\de_1 \le \frac{\De\pr{P-N}}{3\pr{3P - N}}$, since $\om_{j-1}, \de_{j-1} \le \de_1 \le 1$ and $1 - P + N < 1 < 3P - N$ , we get the result.
\end{proof}

%
%
  \appsection{}
\label{AppC}

In this appendix, we will discuss the specifics of $\Ga_j$.  These results are useful in the proof of Theorem \ref{MEst}.  Recall that $\Ga_j = \ga_1 \ldots \ga_j$.

\begin{lem}
Suppose we are in either Case 1 or Case 2.  That is, there is no switching in the way the sequences are defined.  Then
$$ \Ga_j = S_j + a\sum_{1\le k\le j} S_{k-1}S_{j-k}\de_k + a^2\sum_{1\le k < \ell \le j}S_{k-1}S_{\ell -k -1}S_{j-\ell}\de_k\de_\ell +a^3 \sum_{1\le k < \ell < m \le j}S_{k-1}S_{\ell -k -1}S_{m -\ell -1}S_{j-m}\de_k\de_\ell\de_m + \ldots + a^j\de_1\ldots\de_j,$$
where $\disp a = \left\{ \begin{array}{ll} 1 & \textrm{in Case 1} \\ 3 & \textrm{in Case 2} \end{array}\right.$ and $\disp S_k =\left\{ \begin{array}{ll} 1 + 2P+ \ldots + (2P)^k & \textrm{in Case 1} \\ 1 + 2N+ \ldots + (2N)^k & \textrm{in Case 2} \end{array}\right.$.
\label{GaEst}
\end{lem}

\begin{lem}
Suppose we are in Case 3.  Then $3P - N = 1 + \De$ for some $\De > 0$.  Let $S_\ell = 1 + 2N + \ldots + (2N)^\ell$, $V_0 = 1$, $V_{\ell} = 1 + S_{\ell - 1}\De$ for all $\ell \ge 1$, $\eps_1 = \om_{J-1}$, $\eps_\ell = 3\de_{J-2+\ell}$ for all $\ell \ge 2$.  Then for all $j \ge 1$,
$$\frac{\Ga_{J-2+j}}{\Ga_{J-2}} = V_j + \sum_{1 \le k \le j}V_{k-1}S_{j-k}\eps_k + \sum_{1 \le k < \ell \le j}V_{k-1}S_{\ell - k -1}S_{j-\ell}\eps_k\eps_\ell  + \sum_{1 \le k < \ell < m \le j}V_{k-1}S_{\ell - k -1}S_{m-\ell -1}S_{j-m}\eps_k\eps_\ell\eps_m + \ldots + \eps_1\eps_2\ldots\eps_j.$$
\label{GaEstC3}
\end{lem}

Since $\om_{J-1} << \de_{J-1}$ in Case 3, then we get the following corollary.

\begin{cor}
With the notation as in the statement of Lemma \ref{GaEstC3} and $a = 3$, for all $j \ge 1$,
$$\frac{\Ga_{J-2+j}}{\Ga_{J-2}} \le V_j + a\sum_{1 \le k \le j}V_{k-1}S_{j-k}\de_{J-2+k} + a^2\sum_{1 \le k < \ell \le j}V_{k-1}S_{\ell - k -1}S_{j-\ell}\de_{J-2+k}\de_{J-2+\ell} + \ldots + a^j\de_{J-1}\de_{J}\ldots\de_{J-2+j}.$$
\label{GaEstC3Cor}
\end{cor}

For simplicity, we will let $Q = P$ in Case 1 and $Q = N$ in Case 2 and 3.  A straightforward computation gives the following lemma.

\begin{lem}
For any $j \in \N$, $S_jS_1 - S_{j-1}2Q = S_{j+1}$.
\label{hGForm}
\end{lem}

\begin{proof}[Proof of Lemma \ref{GaEst}]
We will present the proof for Case 1.  Case 2 is similar.  We will proceed by induction.  \\
For $j = 1$, we get $\Ga_1 = S_1 + \de_1 = 1 + 2P + \de_1$, as required. \\
Assume the given formula holds for all $k \le j$. \\
By definition, $\ga_{j+1} = \be_{j+1} - 1 + 2P + \de_{j+1} = 2 - \frac{2P}{\ga_j} -1 + 2P+ \de_{j+1}$, so $\Ga_{j+1} = \Ga_j\pr{S_1 + \de_{j+1}} - \Ga_{j-1}2P$.  Substituting the formulas for $\Ga_j$ and $\Ga_{j-1}$, we get 
\begin{align*}
\Ga_{j+1} &= \Ga_j\pr{S_1 + \de_{j+1}} - \Ga_{j-1}2P \\
&= \pr{S_j + \sum_{1\le k\le j} S_{k-1}S_{j-k}\de_k + \sum_{1\le k < \ell \le j}S_{k-1}S_{\ell -k -1}S_{j-\ell}\de_k\de_\ell + \ldots + \de_1\ldots\de_j}\pr{S_1+ \de_{j+1}} \\
& - \pr{S_{j-1} + \sum_{1\le k\le j-1} S_{k-1}S_{j-1-k}\de_k + \sum_{1\le k < \ell \le j-1}S_{k-1}S_{\ell -k -1}S_{j-1-\ell}\de_k\de_\ell + \ldots + \de_1\ldots\de_{j-1}}2P 
\end{align*}
Simplifications and applications of Lemma \ref{hGForm} give the required formula.
\end{proof}

\begin{proof}[Proof of Lemma \ref{GaEstC3}]
Again, we proceed by induction. \\
For $j=1$, $\disp \frac{\Ga_{J-2+1}}{\Ga_{J-2}} = \ga_{J-1} = h_{J-1} - 1 + 2P + \de_{J-1} = 3P - N + \om_{J-1} = 1 + \De + \eps_1 = V_1 + \eps_1$, as required. \\
Assume that the formula above holds for all $k \le j$. \\
As before, $\Ga_{J-2+j+1} = \Ga_{J-2+j}\pr{1 + 2N + \eps_{j+1}} - \Ga_{J-2+j-1}2N$.  Substituting the formulas for $\disp \frac{\Ga_{J-2+j}}{\Ga_{J-2}}$ and $\disp \frac{\Ga_{J-2+j-1}}{\Ga_{J-2}}$ and simplifying, we get
\begin{align*}
\frac{\Ga_{J+j+1}}{\Ga_{J-2}} &= \pr{V_j + \sum_{1 \le k \le j}V_{k-1}S_{j-k}\eps_k + \sum_{1 \le k < \ell \le j}V_{k-1}S_{\ell - k -1}S_{j-\ell}\eps_k\eps_\ell  + \ldots + \eps_1\eps_2\ldots\eps_j}\pr{1 + 2N + \eps_{j+1}}\\
& - \pr{V_{j-1} + \sum_{1 \le k \le j-1}V_{k-1}S_{j-1-k}\eps_k + \sum_{1 \le k < \ell \le j-1}V_{k-1}S_{\ell - k -1}S_{j-1-\ell}\eps_k\eps_\ell  + \ldots + \eps_1\eps_2\ldots\eps_{j-1}}2N \\
\end{align*}
Applying Lemma \ref{hGForm} and simplifying gives the result.
\end{proof}

We would like an upper bound for $\Ga_j$ when $\be_c < 1$.

\begin{lem}
Suppose $\be_c < 1$ and $T_1 >> 1$.  If we are in Case 1 or 2, then $\Ga_j \le C S_j$.  If we are in Case 3, then $\disp \frac{\Ga_{J-2+j}}{\Ga_{J-2}} \le C V_j $.  In all cases, $C$  is a constant that depends on $N$, $P$ and $\de_1$.
\label{hGUB}
\end{lem}

To prove this lemma, we first need a few technical facts.

\begin{lem}
Suppose $\be_c < 1$, $\disp \log\log T_1 \ge 2\max\set{\log(2Q), \log\pr{\frac{1}{2Q-1}}}$ and that we are in either Case 1 or Case 2.  For any $k \ge 0$, $S_{k}\de_{k+1} \le d \pr{k+1} \de_1$, where $d$ is a constant.
\label{aBL1}
\end{lem}

\begin{proof}
If $k = 0$, then this is clearly true. For $k \ge 1$, by definition and Lemma \ref{GaEst}, $T_{k+1} = T_1^{\Ga_k} > T_1^{S_k}$.  Therefore, 
$$\frac{\log\log T_{k+1}}{\log T_{k+1}} < \frac{\log\log T_1 + \log S_k}{S_k\log T_1}= \frac{\log\log T_1}{S_k\log T_1}\pr{1 + \frac{\log S_k}{\log\log T_1}}.$$
Since $\log S_k < \pr{k+1}\log\pr{2Q} + \log\pr{\frac{1}{2Q-1}}$ and $\disp \de_{\ell} \sim \frac{\log\log T_{\ell}}{\log T_{\ell}}$, then by our assumption on $T_1$, the result follows.
\end{proof}

\begin{lem}
Suppose $\be_c < 1$, $\disp \log\log T_{J-1} \ge 2\max\set{\log(2N), \log\pr{\frac{1}{2N-1}} + \log(1 + \De)}$ and that we in Case 3.  
\begin{itemize}
\item[(a)] For any $k \ge 0$, $V_{k}\de_{J-2+k+1} \le d \pr{k+1} \de_{J-1}$;
\item[(b)] for any $k \ge 1$, $S_{k-1}\de_{J-2+k+1} \le d \pr{k+1} \de_{J-1}$, 
\end{itemize}
where $d$ is a constant.
\label{aBL1C3}
\end{lem}

\begin{proof}
If $k = 0$, since $V_0 = 1$, the claimed result holds.  Suppose $k \ge 1$.  By definition and Lemma \ref{GaEstC3}, $T_{J-2+k+1} = T_{J-1}^{\ga_{J-1}\ldots\ga_{J-2+k}}= T_{J-1}^{\frac{\Ga_{J-2+k}}{\Ga_{J-2}}} > T_{J-1}^{V_k}$.  Since $V_k = 1 + S_{k-1}\De > S_{k-1}\De$, then $T_{J-2+k+1}> T_{J-1}^{S_{k-1}\De}$ as well.  Therefore, 
\begin{align*}
& \frac{\log\log T_{J-2+k+1}}{\log T_{J-2+k+1}} < \frac{\log\log T_{J-1} + \log V_k}{V_k\log T_{J-1}}= \frac{\log\log T_{J-1}}{V_k\log T_{J-1}}\pr{1 + \frac{\log V_k}{\log\log T_{J-1}}}, \;\textrm{and} \\
& \frac{\log\log T_{J-2+k+1}}{\log T_{J-2+k+1}} < \frac{\log\log T_{J-1} + \log \pr{S_{k-1}\De}}{S_{k-1}\De\log T_{J-1}}= \frac{\log\log T_{J-1}}{S_{k-1}\De\log T_{J-1}}\pr{1 + \frac{\log \pr{S_{k-1}\De}}{\log\log T_{J-1}}}
\end{align*}
Since $\log\pr{S_{k-1}\De} < \log V_k \le \log\brac{S_{k-1}(1+\De)}< k\log\pr{2N} + \log\pr{\frac{1}{2N-1}} + \log\pr{1 + \De}$ and $\disp \de_{\ell} \sim \frac{\log\log T_{\ell}}{\log T_{\ell}}$, then by our assumption on $T_{J-1}$, the result follows. \\
\end{proof}

\begin{lem}
Suppose $\be_c < 1$ and $\log\log T_1 \ge \log\pr{2Q}$.  If $\ell > k$, then $\disp \frac{\de_\ell}{\de_{\ell - k}} \le c\pr{\frac{1}{2Q}}^k\frac{\ell}{\ell - k}$, where $c$ is a constant.
\label{aBL2}
\end{lem}

\begin{proof}
By definition, $T_\ell = T_{\ell - k}^{\ga_{\ell-k}\ldots\ga_{\ell -1}} > T_{\ell - k}^{\pr{2Q}^k}$.  Thus,
$$\frac{\log\log T_{\ell}}{\log T_\ell} < \pr{\frac{1}{2Q}}^k \frac{\log\log T_{\ell - k} + k\log\pr{2Q}}{\log T_{\ell - k}} = \pr{\frac{1}{2Q}}^k \frac{\log\log T_{\ell - k}}{\log T_{\ell - k}}\pr{1 + \frac{k\log\pr{2Q}}{\log\log T_{\ell - k}}}. $$
Since $\disp T_{\ell - k} = T_1^{\ga_1\ldots \ga_{\ell - k -1}} > T_1^{\pr{2Q}^{\ell - k -1}}$, then $\log\log T_{\ell - k} > \pr{\ell - k -1}\log(2Q) + \log\log T_1 \ge \pr{\ell- k}\log(2Q)$, by the assumption on $T_1$.  Since $\disp \de_{k} \sim \frac{\log\log T_{k}}{\log T_{k}}$ for any $k$, the result follows.
\end{proof}

\begin{cor}
Suppose we are in Case 3 with $\be_c < 1$, $\log\log T_{J-1} \ge \log\pr{2N}$.  If $\ell > k$, then $\disp \frac{\de_{J-2+\ell}}{\de_{J-2 +\ell - (k-1)}} \le c\pr{\frac{1}{2N}}^{k-1}\frac{\ell}{\ell - k+1}$, where $c$ is a constant.
\label{aBL2Cor}
\end{cor}

\begin{lem}
If $\be_c < 1$ then $\disp \sum_{k=1}^j kS_{j-k} \le \frac{\pr{2Q}^{j+2}}{\pr{2Q-1}^3}$.
\label{aBL3}
\end{lem}

\begin{proof}
By rearranging the terms, we see that
$$\sum_{k=1}^j kS_{j-k} <  \frac{\pr{2Q}^{j-1}}{2}\sum_{k=2}^\iny k\pr{k-1}\pr{\frac{1}{2Q}}^{k-2}.$$
The result follows from the facts that $\disp \frac{2}{\pr{1-x}^3} = 2 + 2\cdot 3 x + 3\cdot 4 x^2 + \ldots $ for any $|x| < 1$ and $\be_c < 1$, so $1/2Q < 1$.
\end{proof}

We are now prepared to prove Lemma \ref{hGUB}.

\begin{proof}[Proof of Lemma \ref{hGUB}]
We will first assume that we are in either Case 1 or Case 2.  Then
$$\Ga_j = S_j + a\sum_{1\le k\le j} S_{k-1}S_{j-k}\de_k + a^2\sum_{1\le k < \ell \le j}S_{k-1}S_{\ell -k -1}S_{j-\ell}\de_k\de_\ell + a^3\sum_{1\le k < \ell < m \le j}S_{k-1}S_{\ell -k -1}S_{m -\ell -1}S_{j-m}\de_k\de_\ell\de_m + \ldots + a^j\de_1\ldots\de_j.$$
We will analyze each sum individually, starting with the first.
$$\sum_{1\le k\le j} S_{k-1}S_{j-k}\de_k = \sum_{1\le k\le j} \pr{S_{k-1}\de_k} S_{j-k} \le d\de_1 \sum_{1\le k\le j} kS_{j-k}  \le d\de_1\frac{\pr{2Q}^{j+2}}{\pr{2Q-1}^3},$$
By Lemmas \ref{aBL1} and \ref{aBL3}.  Now we consider the second and third terms:
\begin{align*}
\sum_{1\le k < \ell \le j}S_{k-1}S_{\ell -k -1}S_{j-\ell}\de_k\de_\ell =& \sum_{1\le k < \ell \le j}\pr{S_{k-1}\de_k}\pr{S_{\ell -k -1}\de_{\ell -k}}S_{j-\ell}\frac{\de_\ell}{\de_{\ell -k}} \\
&\le c\pr{d \de_1}^2 \sum_{1\le k < \ell \le j}k \pr{\ell - k}S_{j-\ell}\pr{\frac{1}{2Q}}^k\frac{\ell}{\ell - k} \quad \textrm{(by Lemmas \ref{aBL1} and \ref{aBL2})} \\
&< c\pr{d \de_1}^2 \pr{\sum_{k=1}^{\iny}k \pr{\frac{1}{2Q}}^k}\pr{\sum_{1 \le \ell \le j} \ell\; S_{j-\ell}} \\
&\le \frac{c\pr{d \de_1}^2\pr{2Q}^{j+3}}{\pr{2Q-1}^5} \quad \textrm{(by Lemma \ref{aBL3})},
\end{align*}
\begin{align*}
\sum_{1\le k < \ell < m \le j}S_{k-1}S_{\ell -k -1}S_{m -\ell -1}S_{j-m}\de_k\de_\ell\de_m &= \sum_{1\le k < \ell < m \le j}\pr{S_{k-1}\de_k}\pr{S_{\ell -k -1}\de_{\ell-k}}\pr{S_{m -\ell -1}\de_{m-\ell}}S_{j-m}\frac{\de_\ell}{\de_{\ell -k}}\frac{\de_m}{\de_{m-\ell}} \\
&\le c^2\pr{d\de_1}^3 \sum_{1\le k < \ell < m \le j} k \pr{\ell-k}\pr{m-\ell}S_{j-m}\pr{\frac{1}{2Q}}^k\frac{\ell}{\ell - k}\pr{\frac{1}{2Q}}^\ell\frac{m}{m-\ell} \\
&< c^2\pr{d\de_1}^3 \pr{\sum_{k=1}^\iny k \pr{\frac{1}{2Q}}^k }\pr{\sum_{\ell = 1}^{\iny} \ell\pr{\frac{1}{2Q}}^\ell}\pr{\sum_{1\le m \le j}m \; S_{j-m}} \\
&\le c^2\pr{d\de_1}^3 \frac{\pr{2Q}^{j+4}}{\pr{2Q-1}^7}
\end{align*}
Continuing on, we see that the term involving the product of $n$ $\de$ terms is bounded above by $\disp \frac{\pr{2Q}^{j+1}}{c\pr{2Q-1}}\brac{\frac{a c d \de_1}{\pr{2Q-1}^2}}^n$.  That is,
$$\Ga_j \le S_j + \frac{\pr{2Q}^{j+1}}{c\pr{2Q-1}}\sum_{n=1}^{j}\brac{\frac{a c d \de_1}{\pr{2Q-1}^2}}^n.$$
If we choose $T_1 >> 1$ so that $\de_1 < \frac{\pr{2Q-1}^2}{acd}$, then the sum in the estimate above does not depend on $j$ and we see that
$$\Ga_j < S_j + \frac{\pr{2Q}^{j+1}}{\pr{2Q-1}} \frac{ad\de_1}{\pr{2Q-1}^2 - acd\de_1} = S_j\pr{1 + \frac{\pr{2Q}^{j+1}}{\pr{2Q}^{j+1}-1}\frac{ad\de_1}{\pr{2Q-1}^2 - acd\de_1}} \le CS_j ,$$
as required. \\
If we are in Case 3, then we perform a similar analysis on the expansion given in Corollary \ref{GaEstC3Cor}, except that we apply Lemma \ref{aBL1C3} and Corollary \ref{aBL2Cor}
to show that the term involving the product of $n$ $\de$ terms is bounded above by $\disp \frac{\pr{2N}^j}{c\pr{2N-1}}\brac{acd\de_{J-1}\pr{\frac{2N}{2N-1}}^2}^n$.  If we choose $T_1 >> 1$ so that $\disp \de_{J-1} < \brac{acd\pr{\frac{2N}{2N-1}}^2}^{-1}$, then we see that
$$\frac{\Ga_{J-2+j}}{\Ga_{J-2}} < V_j + \frac{\pr{2N}^j}{c\pr{2N-1}}\sum_{n=1}^j\brac{acd\de_{J-1}\pr{\frac{2N}{2N-1}}^2}^n < V_j + \frac{\pr{2N}^j}{c\pr{2N-1}}\frac{{acd\de_{J-1}\pr{2N}^2}}{\pr{2N-1}^2 -{acd\de_{J-1}\pr{2N}^2}}.$$
Since $\pr{2N}^j < \frac{2N}{\De} V_j$, then $\disp \frac{\Ga_{J-2+j}}{\Ga_{J-2}} \le C V_j$ and we are done.
\end{proof}

\begin{cor}
If $\be_c < 1$ { and $T_1 >> 1$}, then $\disp \frac{\Ga_m(T_1)}{\Ga_m(T_0)} \le C$, where $C$ depends on $N$, $P$ and $T_1$.
\label{GaQuo}
\end{cor}

\begin{proof}
If we are in either Case 1 or Case 2, then $\Ga_m(T_1) \ge S_m$.  By Lemma \ref{hGUB}, $\Ga_{m}(T_1) \le C S_m$, so we get our result.  \\
If we are in Case 3 and $m \le J-2$, then we always use the first definition to define all of our sequences and we are essentially in Case 1.  Otherwise, $m > J-2$ and Lemma \ref{hGUB} implies that $\frac{\Ga_{m}(T_1)}{\Ga_{J-2}(T_1)} \le C V_{m-J + 2}$, or $\Ga_{m}(T_1) \le C V_{m-J + 2}\Ga_{J-2}(T_1) \le C V_{m-J+2}S_{J-2}$, where we used Case 1 for the second inequality.  Since $\Ga_m(T_0) \ge V_{m-J+2}S_{J-2}$, we get the desired inequality.
\end{proof}

We will now use our expansions for $\Ga_j$ to find lower bounds for the cases when $\be_c = 1$.  Note that in these cases, $S_k = k+1$ and $V_k = 1+ k\De$.  

\begin{lem}
Suppose $\be_c = 1$ and we are in either Case 1 or Case 2.  For any $1 \le k \le j$, $\disp \de_k \ge c\frac{j\log k}{k \log j}\de_j$, where $c$ is a constant.
\label{C12B1}
\end{lem}

\begin{proof}
Since $\ga_1 = 2 + \de_1 > 2$ and $\ga_{j+1} = 2 - \frac{1}{\ga_j} + \de_{j+1} > 2 - \frac{1}{\ga_j}$, then it follows that $\ga_j > \frac{j+1}{j}$ for all $j \ge 1$.  Consequently, $T_j = T_k^{\ga_k\ldots\ga_{j-1}} > T_k^{\frac{j}{k}}$ and $\log\log T_k > \log k + \log\log T_1> \log k$.  Therefore,
$$\frac{\log\log T_j}{\log T_j} < \frac{k}{j}\frac{\log\log T_k + \log\pr{j/k}}{\log T_k} = \frac{k}{j}\frac{\log\log T_k}{\log T_k}\pr{1 + \frac{\log\pr{j/k}}{\log\log T_k}} <  \frac{k}{j}\frac{\log\log T_k}{\log T_k}\pr{1 + \frac{\log\pr{j/k}}{\log k}}.$$
The result follows from the fact that $\de_\ell \sim \frac{\log\log T_\ell}{\log T_\ell}$.
\end{proof}

\begin{lem}
Suppose $\be_c = 1$ and we are in Case 3.  For any $1 \le k \le j$, $\disp \de_{J-2+k} \ge c\frac{V_{j-1}\log V_{k-1}}{V_{k-1}\log V_{j-1}}\de_{J-2+j}$, where $c$ is a constant.
\label{C3B1}
\end{lem}

\begin{proof}
Since $\ga_{J-1} = 1 + \De$ and $\ga_{J-2+j+1} = 2 - \frac{1}{\ga_{J-2+j}} + \de_{J-2+j+1} > 2 - \frac{1}{\ga_{J-2+j}}$, then it follows that $\ga_{J-2+j} \ge \frac{1+ j\De}{1 +(j-1)\De} = \frac{V_j}{V_{j-1}}$ for all $j \ge 1$.  Consequently, $T_{J-2+j} = T_{J-2+k}^{\ga_{J-2+k}\ldots\ga_{J-2+j-1}} > T_k^{\frac{V_{j-1}}{V_{k-1}}}$.  Therefore,
$$\frac{\log\log T_{J-2+j}}{\log T_{J-2+j}} < \frac{V_{k-1}}{V_{j-1}}\frac{\log\log T_{J-2+k} + \log\pr{\frac{V_{j-1}}{V_{k-1}}}}{\log T_{J-2+k}} < \frac{V_{k-1}}{V_{j-1}}\frac{\log\log T_{J-2+k}}{\log T_{J-2+k}}\pr{1 + \frac{\log\pr{\frac{V_{j-1}}{V_{k-1}}}}{\log\log T_{J-2+k}}}.$$
The result follows from $\de_\ell \sim \frac{\log\log T_\ell}{\log T_\ell}$ and $\log\log T_{J-2+k} > \log V_{k-1} + \log\log T_{J-1} > \log V_{k-1}$.
\end{proof}

For Case 1 and Case 2, we apply Lemma \ref{C12B1} to the equation in Lemma \ref{GaEst} to get the following.
\begin{align*}
 \Ga_j &> a\sum_{1\le k\le j} k \pr{j-k+1}\de_k \\
 &\ge ac\frac{j}{\log j}\de_j\sum_{1\le k\le j} \pr{\log k} \pr{j-k+1} \\
 &= ac\frac{j\pr{j+1}}{\log j}\de_j\sum_{1\le k\le j}\log k - ac\frac{j}{\log j}\de_j\sum_{1\le k\le j}k \log k \\
 &> \frac{ac j^3}{4}\de_j,
 \end{align*}
 if $j$ is sufficiently large.  Similarly, we apply Lemma \ref{C3B1} to the equation in Corollary \ref{GaEstC3Cor}.
 \begin{align*}
\frac{\Ga_{J-2+(j+1)}}{\Ga_{J-2}} &> a\sum_{1 \le k \le j}V_{k-1}S_{j-k}\de_{J-2+k} \\
&\ge ac\frac{V_{j}}{\log V_{j}}\de_{J-2+j+1}\sum_{1 \le k \le j}\log \pr{k\De}\pr{j+1-k} \\
&= ac\frac{V_{j}}{\log V_{j}}\de_{J-2+j+1}\set{(j+1)\sum_{1 \le k \le j}\log{k} - \sum_{1 \le k \le j}k\log{k} + \log{\De}\sum_{1 \le k \le j }k} \\
&> \frac{ac\De j^3}{5}\de_{J-2+j+1},
\end{align*}
if $j$ is sufficiently large.  Since $\ga_{j} > 1$ for all $j \le J-2$, then it follows that
$$\Ga_{J-1+j} > \frac{ac\De j^3}{5}\de_{J-2+j+1}.$$
These estimates give the following lemma.

\begin{lem}
If $\be_c = 1$ and $j$ is sufficiently large, then $\Ga_j \gtrsim j^3 \de_j$.
\label{2Q1GamEst}
\end{lem}

\begin{lem}
If $\be_c =1$ and $j$ is sufficiently large, then $\disp \be_{j+1} - 1 \gtrsim \frac{1}{j}$.
\label{be1Diff}
\end{lem}

\begin{proof}
For Case 1 and 2, we saw above that $\disp \ga_j > 1 + \frac{1}{j}$.  Therefore, for Case 1, $\disp \be_{j+1} -1=\frac{1}{3}\pr{1 - \frac{1}{\ga_j}} > \frac{1}{3j}$; and for Case 2, $\disp \be_{j+1}-1 = 1 - \frac{1}{\ga_j} > \frac{1}{j}$.  For Case 3, we determined that $\disp \ga_{J-2+j} > \frac{1 + j\De}{1 + (j-1)\De}$, so it follows that $\disp \be_{J-2+j+1} -1= \frac{1}{3}\pr{1 - \frac{1}{\ga_{J-2+j}}} > \frac{\De}{3\pr{1 + j\De}}$. If $j \ge 2J$, say, then the result follows.
\end{proof}

\begin{lem}
Suppose $\be_c = 1$ and $m = C\frac{\log R}{\pr{\log\log R}^k}$ for some $k \in \N$.  If $T_{m+1} = R$, then $\disp \lim_{R\to \iny} \log T_1 = 0$.
\label{lT1Zero}
\end{lem}

\begin{proof}
Recall that $T_{m+1} = T_1^{\Ga_m}$.  Therefore, 
$$\log T_1 = \frac{1}{\Ga_m}\log T_{m+1} \lesssim \frac{1}{m^3}\frac{1}{\de_m}\log R \lesssim \brac{\frac{\pr{\log\log R}^k}{\log R}}^3\frac{\log R}{\log\log R}\log R = \frac{\pr{\log\log R}^{3k -1}}{\log R},$$
by Lemma \ref{2Q1GamEst}, the choice of $m$ and the size of $\de_m$.  The result follows.
\end{proof}

%
%
  \appsection{}
\label{AppD}

The goal of this appendix is to estimate the size of the imaginary part of $\disp \frac{ r^{-2k}\mu_n}{b\mu_{n-2k}}$ and show that it is arbitrarily small.  The following estimate is useful in Step 1C of the proof of Lemma \ref{meshN}.  Without this estimate, our construction would only work for real eigenvalues.

\begin{lem}
For $b$ as in the proof of Lemma \ref{meshN}, there exists a constant $C$ such that for all $r \in \brac{\rhoa{\tfrac{2}{3}}, \rhoa{\tfrac{4}{3}}}$, $\disp \abs{\textrm{Im}\pr{\frac{r^{-2k}\mu_n\pr{r}}{b \mu_{n-2k}\pr{r}}}} \le \frac{C}{n}.$  In particular, $\rho_0$ may be chosen large enough so that $\disp \abs{\textrm{Im}\pr{\frac{r^{-2k}\mu_n\pr{r}}{b \mu_{n-2k}\pr{r}}}} \le \frac{1}{2}\sin\pr{\frac{\pi}{7}} .$
\label{imEst}
\end{lem}

\begin{proof}
For a given eigenvalue, $\la \in \C$, we will let $\te$ denote the argument of $\la$.  That is, $\la = \abs{\la} e^{i \te}$.  Let $\tilde r = r - \pr{\rhoo}$.  For $r \in \brac{\rhoa{\tfrac{2}{3}}, \rhoa{\tfrac{4}{3}}}$, we see that $\tilde r \sim r^{\al} = r^{\be_0/2 -1}$.  Since
$$\mu_n\pr{r} = \exp\pr{\frac{\la r^2}{4n} + \frac{\la^2 r^4}{32 n^3} + \frac{\la^3 r^6}{96 n^5} + \ldots},$$
then
\begin{align*}
\log \brac{\frac{\mu_n\pr{r}}{\mu_{n-2k}\pr{r}}\frac{\mu_{n-2k}\pr{\rhoo}}{\mu_n\pr{\rhoo}}} 
&= \Ga\pr{\cos\te + i \sin\te} + \bigO\pr{\frac{r^2}{n^3}}\pr{\cos \tilde\te + i \sin\tilde\te},
\end{align*}
where $\disp \Ga = - \frac{|\la| k \tilde r \pr{\rhoo}\pr{1 + \frac{\tilde r}{\rhoo}}}{n(n-2k)} = \bigO\pr{\frac{1}{n}}$.
Thus,
\begin{align*}
\frac{\mu_n\pr{r}}{\mu_{n-2k}\pr{r}}\frac{\mu_{n-2k}\pr{\rhoo}}{\mu_n\pr{\rhoo}} 
 &= \pr{1 + \Ga\cos\te + \frac{1}{2}\pr{\Ga\cos\te}^2 + \ldots}\brac{1 + i \Ga\sin\te - \frac{1}{2}\pr{\Ga \sin\te}^2 + \ldots} \\
 &\times\pr{1 + \bigO\pr{\frac{r^2}{n^3}}}\brac{1 + i\bigO\pr{\frac{r^2}{n^3}}}.
\end{align*}
The result follows from the fact that
$$\textrm{Im}\brac{ \frac{\mu_n\pr{r}}{\mu_{n-2k}\pr{r}}\frac{\mu_{n-2k}\pr{\rhoo}}{\mu_n\pr{\rhoo}}} = \Ga\sin\te + \bigO\pr{\frac{r^2}{n^3}}.$$
\end{proof}

\nid \textbf{Acknowledgement} I would like to { extend special thanks to} my advisor, Carlos Kenig, for suggesting this problem and for his guidance and support.  In addition, I would like to thank Francis Chung for our helpful conversations.  { Finally, I'd like to express my gratitude to the referee for their very careful examination of this work and their useful feedback.}

\bibliography{refs}
\bibliographystyle{chicago}

\end{document}